\documentclass[12pt]{report}

\usepackage{amsmath, amsfonts, amssymb}
\usepackage{amsthm}
\usepackage{esint}
\usepackage[
bookmarks=false]{hyperref}
\usepackage{cite}
\usepackage{array}
\usepackage{enumerate}
\usepackage{pgfplots}

\newcommand{\R}{\mathbb{R}}
\newcommand{\N}{\mathbb{N}}
\newcommand{\Q}{\mathbb{Q}}
\newcommand{\Z}{\mathbb{Z}}

\newcommand{\C}{\mathbb{C}}

\newcommand{\be}{\begin{equation*}}
\newcommand{\ee}{\end{equation*}}
\newcommand{\bel}{\begin{equation}}
\newcommand{\eel}{\end{equation}}
\newcommand{\bee}{\begin{eqnarray*}}
\newcommand{\eee}{\end{eqnarray*}}
\newcommand{\eps}{\varepsilon}
\newcommand{\dbar}{\overline{\partial}}

\newtheorem{thm}{Theorem}[chapter]
\newtheorem{lem}[thm]{Lemma}
\newtheorem{prop}[thm]{Proposition}
\newtheorem{dfn}[thm]{Definition}

\newenvironment{dedication}
  {\clearpage           
   \thispagestyle{empty}
   \vspace*{\stretch{1}}
   \itshape          
   \raggedleft        
  }
  {\par 
   \vspace{\stretch{2}} 
   \clearpage      
  }

\begin{document}

\pagenumbering{gobble}

\begin{titlepage}

\center{\huge{Scuola Normale Superiore}}
\bigskip \bigskip \bigskip

\begin{figure}[h]
\centering
\includegraphics[width=4cm]{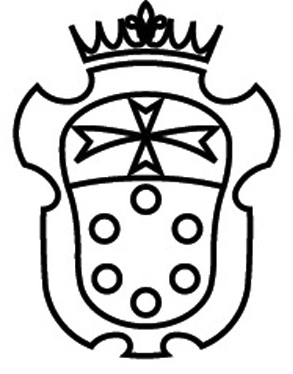}
\end{figure}

\center{Tesi di Perfezionamento}

\center{\huge{\textbf{Matrix Schr\"odinger Operators \\ and Weighted Bergman Kernels}}}

\center{\ \\ \ \\ \ \\ \ \\ \ \\
Perfezionando: \\
Gian Maria Dall'Ara \\
\ \\
Relatore: \\
Prof. Fulvio Ricci
\ \\ \ \\ \ \\ 

}

\end{titlepage}

\begin{dedication}

\end{dedication}

\begin{dedication}
Ai miei genitori
\end{dedication}

\begin{dedication}

\end{dedication}

\tableofcontents

\chapter*{Introduction}
\addcontentsline{toc}{chapter}{Introduction}
\pagenumbering{roman}

The aim of the present thesis is twofold: to study the problem of discreteness of the spectrum of Schr\"odinger operators with matrix-valued potentials in $\R^d$ (Chapter \ref{m-ch}), and to prove new pointwise bounds for weighted Bergman kernels in $\C^n$ (chapters \ref{w-ch} and \ref{C2-ch}). These two themes are connected by the observation that the weighted Bergman kernel may be effectively studied by means of the analysis of the weighted Kohn Laplacian, which in turn is unitarily equivalent to a Sch\"odinger operator with matrix-valued potential. 

The Kohn Laplacian is a key differential operator in complex analysis, and here we will deal with its weighted variant. It was Berndtsson to use for the first time the observation above to deduce results in complex analysis from known facts in mathematical physics \cite{berndtsson}. Later, the same idea was successfully exploited by Christ \cite{christ}, Fu and Straube \cite{fu-straube}, \cite{fu-straube-2}, Christ and Fu \cite{christ-fu}, and Haslinger \cite{haslinger-magn}. All of these papers concern one-dimensional complex analysis, while our goal is to extend part of these technology to several variables, a problem raised by Christ in \cite{christ}.\newline

Let us begin with the first part. Given a magnetic potential $A:\R^d\rightarrow\R^d$ and an electric potential $V:\R^d\rightarrow \R$, the corresponding Schr\"odinger operator is formally defined as follows:\be
\mathcal{H}_{A,V}\psi:=-\Delta_A\psi + V\psi,
\ee 
where $\Delta_A\psi:=\sum_{k=1}^d\left(\frac{\partial}{\partial x_k}-iA_k\right)\left(\frac{\partial}{\partial x_k}-iA_k\right)\psi$. Under appropriate assumptions on $A$ and $V$ (see \cite{avron-herbst-simon} for precise statements), one can extend $\mathcal{H}_{A,V}$ as a self-adjoint operator on $L^2(\R^d)$. It is a classical problem in mathematical physics to study how the properties of the spectrum of this operator depend on $V$ and $A$. A particularly interesting property, which may or may not hold, is discreteness of the spectrum or, in quantum physics terminology, quantization of energy. When the potential $V$ is non-negative, this turns out to be equivalent \cite{iwatsuka} to the existence of a function $\mu:\R^d\rightarrow[0,+\infty]$ such that $\lim_{x\rightarrow\infty}\mu(x)=+\infty$ and \bel\label{intro-coerc}
(\mathcal{H}_{V,A}\psi,\psi)=\int_{\R^d}|\nabla_A\psi|^2+\int_{\R^d}V|\psi|^2\geq \int_{\R^d}\mu^2 |\psi|^2\qquad\forall \psi.
\eel
The estimate above obviously holds for $\mu^2=V\geq0$ Therefore if \be\lim_{x\rightarrow\infty}V(x)=+\infty,\ee then the spectrum of $\mathcal{H}_{V,A}$ is discrete. Nevertheless this condition is not necessary for discreteness of the spectrum: a weaker necessary and sufficient condition was found in 1953 by Molchanov and significantly refined in 2005 by Maz'ya and Shubin (see Section \ref{mash-sec} for details in the case $A=0$ and $V\geq0$, and the paper \cite{ko-ma-sh} for the magnetic case, still with $V\geq0$). 

In Chapter \ref{m-ch} we study the problem of discreteness of the spectrum for Schr\"odinger operators with matrix-valued potentials, which from now on will be simply called matrix Schr\"odinger operators. These are natural generalizations of ordinary Schr\"odinger operators obtained by replacing the Hilbert space $L^2(\R^d)$ with $L^2(\R^d,\C^m)$, that is by considering vector-valued wave-functions. The electric potential is defined as a mapping \be V:\R^d\rightarrow H_m,\ee where $H_m$ is the space of $m\times m$ Hermitian matrices, while magnetic potentials are defined as before. The corresponding matrix Schr\"odinger operator is formally defined in analogy with the scalar case:\be
\mathcal{H}_{A,V}\psi:=-\Delta_A\psi + V\psi,
\ee where now $\psi$ is $\C^m$-valued, $\Delta_A$ acts diagonally, i.e., componentwise, and $V\psi$ is the pointwise matrix-times-vector product. Since $V$ is Hermitian, $\mathcal{H}_{A,V}$ is formally self-adjoint. 

When $A=0$ and $V$ is locally integrable and the matrix $V(x)$ is non-negative at each $x\in\R^d$, it is easy to extend $\mathcal{H}_V\equiv\mathcal{H}_{V,0}$ to a densely-defined self-adjoint operator on $L^2(\R^d,\C^m)$. The operators $\mathcal{H}_V$ were already considered in the literature: see, e.g., \cite{frank-lieb-seiringer}.

The first original result we present (Theorem \ref{gen-thm}) is a characterization of discreteness of the spectrum of $\mathcal{H}_V$ when $V$ is in a matrix-valued Muckenhoupt $A_\infty$ class, that we introduce for this purpose. A by-product of our definition of the appropriate $A_\infty$ class is that a fairly natural generalization of Maz'ya-Shubin characterization to matrix Schr\"odinger operators cannot hold. We also discuss the relation of our Muckenhoupt class of matrices with the matrix-valued $A_2$ class introduced by Treil and Volberg in the context of vector-valued harmonic analysis.

The other results of Chapter \ref{m-ch} originate from the comparison of matrix and scalar Schr\"odinger operators. If $A\in H_m$, we denote by $\lambda(A)$ the smallest eigenvalue of $A$. If $V:\R^d\rightarrow H_m$ is a non-negative and locally integrable matrix potential, we have the scalar potential $\lambda(V)$, and it is easy to see that:\bel\label{intro-impl}
\mathcal{H}_{\lambda(V)}\text{ has discrete spectrum}\quad\Longrightarrow\quad \mathcal{H}_V\text{ has discrete spectrum}.
\eel
The reverse implication fails rather dramatically in general. In fact, Theorem \ref{osc2-thm} gives a sufficient condition for discreteness of the spectrum, which is satisfied by certain matrix-valued potentials $V$ such that $V(x)$ has rank $1$ for every $x$. If $m\geq 2$, we have $\lambda(V)\equiv0$ for these potentials, and hence $\mathcal{H}_{\lambda(V)}=-\Delta$ is well-known not to have discrete spectrum. The sufficient condition of Theorem \ref{osc2-thm} is formulated in terms of a new concept of oscillation for mappings associating to each point $x\in\R^d$ a linear subspace of $\C^m$ (i.e., measurable vector-bundles of non-constant rank). 

The last result of the first part (Theorem \ref{pol-thm}) is somewhat complementary to Theorem \ref{osc2-thm}, and states that the reverse implication of \eqref{intro-impl} holds for $2\times 2$ potentials with real polynomial entries, showing that under a rigidity hypothesis on the potential the problem of discreteness of the spectrum may be reduced to a scalar problem.\newline

In the second part of the thesis (chapters \ref{w-ch} and \ref{C2-ch}) the basic datum is a weight\be
\varphi:\C^n\longrightarrow\R,
\ee
and the principal objects of study are:\begin{enumerate}
\item the weighted Bergman space $A^2(\C^n,\varphi)$, i.e., the closed subspace of the Hilbert space \be
L^2(\C^n,\varphi):=\left\{f:\C^n\rightarrow\C\ \text{ such that }\ \int_{\C^n}|f|^2e^{-2\varphi}<+\infty\right\}
\ee consisting of holomorphic functions, and the corresponding orthogonal projector $B_\varphi: L^2(\C^n,\varphi)\rightarrow A^2(\C^n,\varphi)$,
\item the weighted $\dbar$ problem, that is, the problem of finding a function $f\in L^2(\C^n,\varphi)$ such that \be
\dbar f=u,
\ee where $u\in L^2_{(0,1)}(\C^n,\varphi)$ (i.e., the space of $(0,1)$-forms with coefficients in $L^2(\C^n,\varphi)$) is such that $\dbar u=0$,
\item the weighted Kohn Laplacian $\Box_\varphi$, which is a self-adjoint second order operator on $L^2_{(0,1)}(\C^n,\varphi)$, arising naturally in the study of the weighted $\dbar$ problem.
\end{enumerate}

Bergman operators, $\dbar$ problems and Kohn Laplacians have a long history. Their analysis started with the works of H\"ormander, Bergman and Kohn among others, and we invite the reader to consult the bibliography of \cite{chen-shaw} for a rich historical background.

The three objects listed above are strongly interrelated. In fact, our main goal is to prove pointwise bounds for the integral kernel $B_\varphi(z,w)$ of the weighted Bergman projector $B_\varphi$. By an adaptation due to Delin \cite{delin} of an argument of Kerzman \cite{kerzman}, it is enough to prove pointwise estimates on the solutions of the weighted $\dbar$ problem that are orthogonal to $A^2(\C^n,\varphi)$. These bounds may in turn be obtained from a careful study of the Kohn Laplacian $\Box_\varphi$. 

This brings our attention on $\Box_\varphi$ and, as anticipated at the beginning of the introduction, $\Box_\varphi$ is unitarily equivalent to a scalar multiple of a matrix Schr\"odinger operator $\mathcal{H}_{V,A}$ acting on $L^2(\C^n,\C^n)$, whose magnetic and electric potentials depend on $\varphi$. 

This observation was exploited in one-complex dimension by Christ, who proved in\cite{christ} that if $\varphi:\C\rightarrow\R$ is subharmonic and satisfies some mild regularity and uniformity assumptions, then \be
|B_\varphi(z,w)|\leq Ce^{\varphi(z)+\varphi(w)}\frac{e^{-\eps d_0(z,w)}}{\rho_0(z)\rho_0(w)}\qquad\forall z,w\in\C,
\ee
where \bel\label{intro-rho}
\rho_0(z):=\sup\left\{r>0:\ \int_{|w-z|\leq r}\Delta\varphi\leq 1\right\},
\eel
and $d_0$ is the Riemannian distance on $\C\equiv\R^2$ associated to the metric $\rho_0(z)^{-2}(dx^2+dy^2)$. Christ's proof appeals to Agmon theory, a powerful tool developed to establish exponential decay of eigenfunctions of Schr\"odinger operators \cite{agmon}.

In trying to extend this approach to several complex variables, one has to face two difficulties:\begin{enumerate}
\item $\Box_\varphi$ is no longer an ordinary Schr\"odinger operator, but a matrix Schr\"odinger operator, 
\item the electrical potential $V$ is no longer non-negative.
\end{enumerate}
The most serious obstacle is the second one, and to avoid it we have to follow a different route. Before describing it, it is worth mentioning that Delin obtained in \cite{delin} a generalization of Christ's result in several variables under the assumption that $\varphi$ be uniformly strictly plurisubharmonic. Unfortunately, Delin's result does not apply to weakly plurisubharmonic weights.

Coming back to our approach, we say that $\Box_\varphi$ is $\mu$-coercive if\be
(\Box_\varphi u,u)_\varphi\geq \int_{\C^n}\mu^2 |u|^2e^{-2\varphi}\qquad\forall u\in L^2_{(0,1)}(\C^n,\varphi),
\ee where $(\cdot,\cdot)_\varphi$ is the scalar product in $L^2_{(0,1)}(\C^n,\varphi)$, and $\mu:\C^n\rightarrow[0,+\infty]$. This condition should be compared with \eqref{intro-coerc}.

In Chapter \ref{w-ch} we prove (Theorem \ref{bergman-thm}) that if $\varphi:\C^n\rightarrow\R$ is a plurisubharmonic weight such that $\Box_\varphi$ is $\kappa^{-1}$-coercive, and $\varphi$ and $\kappa$ meet some additional mild restrictions, then we have the following estimate for the weighted Bergman kernel:\bel\label{intro-mu}
|B_\varphi(z,w)|\leq Ce^{\varphi(z)+\varphi(w)}\frac{\kappa(z)}{\rho(z)}\frac{e^{-\eps d_\kappa(z,w)}}{\rho(z)^n\rho(w)^n}\qquad\forall z,w\in\C^n,
\eel
where \bel\label{intro-rho2}
\rho(z):=\sup\left\{r>0:\ \max_{|w-z|\leq r}\Delta\varphi(w)\leq r^{-2}\right\},
\eel
and $d_\kappa$ is the Riemannian distance associated to the metric \be\kappa(z)^{-2}\sum_{j=1}^n(dx_j^2+dy_j^2).\ee
Notice that \eqref{intro-rho} is better than \eqref{intro-rho2}, in the sense that when $n=1$ $\rho_0$ is larger than $\rho$. This difference comes from the use of $L^\infty$ rather than $L^1$ bounds in our arguments. We do not consider this a serious limitation, since $\rho_0$ and $\rho$ are comparable when the weight $\varphi$ is a polynomial, a particularly interesting case often used as a model of the more general finite-type case.

The second important thing to observe is that the distance $d_0$ in Christ's estimate is replaced by $d_\kappa$ in our estimate, and that a factor $\frac{\kappa(z)}{\rho(z)}$ appears. 

If $\Box_\varphi$ is $c\rho^{-1}$-coercive (for some $c>0$), which is the case when the eigenvalues of the complex Hessian of $\varphi$ are comparable (Theorem \ref{bergman-comp}), then our result is the natural generalization of Christ's, that is\be
|B_\varphi(z,w)|\leq Ce^{\varphi(z)+\varphi(w)}\frac{e^{-\eps d_\rho(z,w)}}{\rho(z)^n\rho(w)^n}\qquad\forall z,w\in\C^n.
\ee

In general, $\Box_\varphi$ is not $c\rho^{-1}$-coercive, and Theorem \ref{bergman-thm} is essentially a conditional result giving a non-trivial estimate for $B_\varphi(z,w)$ whenever one knows that $\Box_\varphi$ is $\mu$-coercive.

It is then natural to consider the problem of establishing $\mu$-coercivity of $\Box_\varphi$ for weights in a given class. In Chapter \ref{C2-ch} we study this problem for a class of polynomial weights in $\C^2$, which we name \emph{homogeneous model weights}, since they are homogeneous with respect to certain non necessarily isotropic dilations (see Section \ref{model-sec} for the precise definition). Theorem \ref{model-thm} states that if $\varphi:\C^2\rightarrow[0,+\infty)$ is a homogeneous model weight, then $\Box_\varphi$ is $\mu$-coercive, where \be
\mu(z,w)=c(1+|z|^\sigma+|w|^\tau).
\ee Here $c$ is a positive constant depending on $\varphi$ and $\sigma$ and $\tau$ are non-negative real numbers which can be easily computed from the Newton diagram of $\varphi$.

The proof of this result is based both on explicit computations and on a more conceptual tool: a \emph{holomorphic uncertainty principle} (Lemma \ref{hol-unc}) which is again inspired by the relation of weighted Kohn Laplacians and matrix Schr\"odinger operators.

\chapter*{Notation}
\addcontentsline{toc}{chapter}{Notation}

Let us introduce some notation concerning matrices that will be used throughout the work.

The symbol $(v,w)$ denotes the Hermitian scalar product of $v,w\in\C^m$, while $x\cdot y$ is the euclidean scalar product of $x,y\in\R^m$.

In this thesis we often deal with matrices and matrix-valued functions. We denote by $H_m$ ($m\in\N$) the real vector space of Hermitian matrices, i.e. $A=(A_{jk})_{j,k=1}^m\in H_d$ if and only if $A_{jk}\in \C$ for every $j,k$ and $A_{jk}=\overline{A_{kj}}$. The matrix $A\in H_m$ acts as a self-adjoint linear operator on $\C^m$, i.e., \be
(Av,w)=(v,Aw)\qquad\forall v,w\in\C^m,
\ee and by the spectral theorem there exists an orthonormal basis $v_1,\dots,v_m$ of $\C^m$ consisting of eigenvectors of $A$. The eigenvalues $\lambda_1,\dots,\lambda_m$ are real, the smallest one is \be
\lambda(A):=\min_{|v|=1}(Av,v),
\ee and the largest one is \be
\mu(A):=\max_{|v|=1}(Av,v).
\ee
The operator norm of $A$, defined as usual as $||A||:=\max_{|v|\leq 1}|Av|$, is easily seen to be equal to $\mu(A)$.

A function $V:E\rightarrow H_m$ ($E\subseteq\R^d$) is said to be (locally) integrable if its pointwise operator norm $||V||$ is (locally) integrable. If it is integrable $\int_EV$ is the element of $H_m$ defined by the identity\be
\left(\left(\int_EV\right)v,w\right)=\int_E(Vv,w)\qquad\forall v,w\in\C^m.
\ee

\begin{dedication}

\end{dedication}

\chapter*{Acknowledgements}
\addcontentsline{toc}{chapter}{Acknowledgements}

\par I would like to express my deep gratitude to Fulvio Ricci, who guided me through this beautiful part of mathematics and without whom I would have surely been lost. 

I am also very grateful to Paolo Ciatti, who introduced me to harmonic analysis while I was an undergraduate student in Padua and dedicated to me a lot of his time, and to Alexander Nagel, who has been a wonderful host (together with Yvonne, of course) and taught me a lot of mathematics during my two visits to Madison, Wisconsin. 

Let me try to write an (incomplete) list of the people that shared with me a bit of their mathematical knowledge in the last years: Dario Trevisan, Francesco Ghiraldin, Bozhidar Velichkov, Berardo Ruffini, Samuele Mongodi, Alessio Martini, Laura Cremaschi, and all the guys of the ``Seminario Informale di Analisi". Talking with them helped me to overcome many of the difficulties emerged in my research. 

Finally, I would like to thank all the people that have been close to me during these years in Pisa: my family (my mother Ivana, my father Franco, my uncle Ezio, and my aunt Marzia), Bert\'imene, my good friends Ivano and Alessio, Michele, Luchino and Nika, Luca Lopez, Francesco and Massimo Taronna, Ruggero Roydelberg, Paolo Natali, Levacci and Gaia, Giorgio Lando, and my neighbor Biancamaria.  

\begin{dedication}

\end{dedication}

\chapter{Matrix Schr\"odinger operators}\label{m-ch}
\pagenumbering{arabic}

In this chapter we present our results concerning matrix Schr\"odinger operators. Sections 1.1 through 1.5 contain a few basic facts about them and the statement of the Maz'ya-Shubin characterization of discreteness of the spectrum for scalar Schr\"odinger operators. Sections 1.6 through 1.16 contain our original results.

\section{Magnetic potentials, gradients\\ and Laplacians}\label{mag-sec}

A \emph{magnetic potential} on $\R^d$ is, for our purposes, any $C^1$ vector field\be
A:\R^d\rightarrow\R^d.
\ee
The \emph{magnetic gradient} (relative to $A$) of a $C^1$ scalar function $f:\R^d\rightarrow\C$ is defined as\be
\nabla_Af:=\left(\frac{\partial f}{\partial x_1}-iA_1f,\dots,\frac{\partial f}{\partial x_d}-iA_df\right).
\ee Notice that $\nabla_Af$ is $\C^d$-valued. If $A=0$, then $\nabla_A$ coincides with the ordinary Euclidean gradient. 

The formal adjoint $\nabla_A^*$ of the magnetic gradient is the divergence-like differential operator that takes a $C^1$ vector-valued function $F:\R^d\rightarrow\C^d$ into the scalar function\be
\nabla_A^*F=\sum_{k=1}^d\left(-\frac{\partial}{\partial x_k}+iA_k\right)F_k.
\ee
Of course, we have the identity $\int_{\R^d}\nabla_A f\cdot \overline{F}=\int_{\R^d} f\overline{\nabla_A^*F}$, whenever $f$ and $F$ are $C^1$, and either $f$ or $F$ is compactly supported.

In analogy with the definition of the ordinary Laplacian as the composition of divergence and gradient, we define the \emph{magnetic Laplacian} of a $C^2$ scalar function $f:\R^d\rightarrow\C$:\bee
\Delta_Af&:=&-\nabla_A^*\nabla_Af\\
&=&-\sum_{k=1}^d\left(-\frac{\partial}{\partial x_k}+iA_k\right)\left(\frac{\partial}{\partial x_k}-iA_k\right)f\\
&=&\Delta f-2i\sum_{k=1}^dA_k\frac{\partial f}{\partial x_k}-i\text{div}(A)f-|A|^2f,
\eee where $\text{div}(A)=\sum_{j=1}^d\frac{\partial A_j}{\partial x_j}$ denotes the ordinary divergence of $A$. An integration by parts yields the identity\be
-\int_{\R^d}\Delta_Af\overline{f}=\int_{\R^d}|\nabla_Af|^2=-\int_{\R^d}f\overline{\Delta_Af}\qquad\forall f\in C^2_c(\R^d),
\ee  showing in particular that $\Delta_A$ is formally self-adjoint and non-positive.

Of course, one can define $\nabla_A$, $\nabla_A^*$ and $\Delta_A$ distributionally on less regular functions, but in this section we are only interested in formal definitions.
\newline

The next lemma contains an important basic fact about magnetic Laplacians: the \emph{diamagnetic inequality}.

\begin{lem}\label{mag-diamag}If $f\in C^1(\R^d)$ then \be
|\nabla_A f(x)|\geq |\nabla|f|(x)|\qquad\text{a.e. } x\in\R^d.
\ee 
\end{lem}
Notice that $|f|$ is Lipschitz and thus $\nabla |f|$ exists almost everywhere by Rademacher's theorem.

\begin{proof} If $\nabla|f|(x)$ exists, then $|\nabla |f|(x)|\leq |\nabla f(x)|$. If $f(x)=0$ then $\nabla_Af(x)=\nabla f(x)$ and the proof is complete.

If $f(x)\neq0$, in a neighborhood of $x$ we can write $f=re^{i\theta}$, where $r$ and $\theta$ are smooth and real-valued. We have
\bee
|\nabla_A f|^2&=&\sum_{k=1}^d\left|\left(\frac{\partial }{\partial x_k}-iA_k\right) (re^{i\theta})\right|^2\\
&=&\sum_{k=1}^d\left|\left(\frac{\partial r}{\partial x_k}+ir\frac{\partial \theta}{\partial x_k}-irA_k\right)e^{i\theta}\right|^2\\
&=&\sum_{k=1}^d\left|\frac{\partial r}{\partial x_k}\right|^2+r^2\left|\frac{\partial \theta}{\partial x_k}-A_k\right|^2\\
&\geq&\sum_{k=1}^d\left|\frac{\partial r}{\partial x_k}\right|^2=|\nabla |f||^2.
\eee\end{proof}

\section{Magnetic matrix Schr\"odinger operators}\label{schrod-sec}

A \emph{magnetic matrix Schr\"odinger operator} acts on $\C^m$-valued functions defined on a Euclidean space $\R^d$, and is determined by the following two data:\begin{enumerate}
\item an $m\times m$ Hermitian matrix-valued \emph{electric potential} \be
V:\R^d\rightarrow H_m,
\ee which we assume to be locally integrable,

\item a $C^1$ \emph{magnetic potential} \be
A:\R^d\rightarrow\R^d.
\ee
\end{enumerate}

The formal expression of the magnetic matrix Schr\"odinger operator is the following:\be
\mathcal{H}_{V,A}\psi:=-\Delta_A\psi+V\psi\qquad\forall \psi\in C^2(\R^d,\C^m).
\ee
Here $\Delta_A$ acts diagonally, i.e., componentwise, on $\psi$, and $V$ acts by pointwise matrix multiplication. Explicitly, \be
\mathcal{H}_{V,A}\psi=\left(-\Delta_A\psi_k+\sum_{\ell=1}^mV_{k\ell}\psi_\ell \right)_{k=1}^m,
\ee
or in other words, $\mathcal{H}_{V,A}$ is the matrix differential operator:\be
\mathcal{H}_{V,A}=\begin{bmatrix}
-\Delta_A+V_{11} &\cdots& V_{1m}\\
\\
\vdots&\ddots&\vdots\\
\\
V_{m1} &\cdots& -\Delta_A+V_{mm}
\end{bmatrix}.
\ee

Since $V$ is Hermitian at every point and $\Delta_A$ is formally self-adjoint, $\mathcal{H}_{V,A}$ is formally self-adjoint too: \be
\int_{\R^d}(\mathcal{H}_{V,A}\psi,\phi)=\int_{\R^d}(\psi,\mathcal{H}_{V,A}\phi) \qquad\forall \psi,\phi\in C^2_c(\R^d,\C^m),
\ee where $(\cdot,\cdot)$ is the hermitian scalar product in $\C^m$. Notice that the integrals are absolutely convergent because $V$ is locally integrable.

We now define the \emph{energy functional} associated to $\mathcal{H}_{V,A}$, at least for $\psi\in C^2_c(\R^d,\C^m)$: \bel\label{schrod-energy}
\mathcal{E}_{V,A}(\psi):=\int_{\R^d}(\mathcal{H}_{V,A}\psi,\psi)=\int_{\R^d}|\nabla_A\psi|^2+\int_{\R^d}(V\psi,\psi),
\eel
where $|\nabla_A\psi|^2=\sum_{k=1}^m|\nabla_A\psi_k|^2$. The first term of the right-hand side of \eqref{schrod-energy} is called the \emph{kinetic energy}, while the second is the \emph{potential energy} of $\psi$. Notice that $(V\psi,\psi)$ is the pointwise evaluation of the quadratic form associated to $V$ on $\psi$. 

If $m=1$, $\mathcal{H}_{V,A}$ is the usual magnetic Schr\"odinger operator $-\Delta_A+V$ with scalar potential $V$, and the energy takes the form $\int_{\R^d}|\nabla_A\psi|^2+\int_{\R^d}V|\psi|^2$.

The case $m\geq2$ can not be reduced to the scalar one in general, unless the matrices $V(x)$ ($x\in\R^d$) can be simultaneously diagonalized.

In this section we introduced magnetic matrix Schr\"odinger operators in a formal way, that is on domains of smooth (vector-valued) functions. In the next section we show how to extend them to self-adjoint operators in the easier case when $V$ is non-negative and $A=0$. 

\section{Extending $\mathcal{H}_{V,0}$ to a self-adjoint operator\\ when $V$ is non-negative}\label{ext-sec}

We assume throughout this section that \be
V:\R^d\longrightarrow H_m
\ee 
is locally integrable and that $V\geq0$ everywhere, i.e.,\be
(V(x)v,v)\geq0\qquad\forall v\in\C^m,\ \forall x\in\R^d.
\ee 
We begin by extending the domain of $\mathcal{E}_{V,0}$, which we simply denote by $\mathcal{E}_V$. Analogously, we replace the symbol $\mathcal{H}_{V,0}$ with $\mathcal{H}_{V}$.

Let us introduce the linear space
\be
\mathcal{D}(\mathcal{E}_V):=\left\{\psi\in L^2(\R^d,\C^m):\ \frac{\partial\psi}{\partial x_j}\in L^2(\R^d,\C^m) \ \forall j, \quad(V\psi,\psi)\in L^1(\R^d)\right\},
\ee where $\frac{\partial}{\partial x_j}$ is to be interpreted componentwise and in the sense of distributions. If $\psi\in\mathcal{D}(\mathcal{E}_V)$, then\be\mathcal{E}_V(\psi):=\int_{\R^d}|\nabla\psi|^2+\int_{\R^d}(V\psi,\psi)\ee is well-defined. This clearly extends the definition given in Section \ref{schrod-sec}.

\begin{prop}\label{ext-hilb}
The expression $||\psi||_V:=\sqrt{\mathcal{E}_V(\psi)+||\psi||^2}$, where $||\cdot||$ is the norm of $L^2(\R^d,\C^m)$, is a Hilbert space norm on $\mathcal{D}(\mathcal{E}_V)$. 
\end{prop}

\begin{proof}
$||\psi||_V$ is easily seen to be a norm coming from a scalar product.

Completeness is obtained observing that $\psi_\ell$ is a Cauchy sequence with respect to $||\cdot||_V$ if and only if $\psi_\ell$ and $\frac{\partial\psi_\ell}{\partial x_j}$ are Cauchy sequences in $L^2(\R^d,\C^m)$, and $(V\psi_\ell,\psi_\ell)$ is a Cauchy sequence in $L^1(\R^d)$ (one needs to use the fact that $V$ is non-negative). \end{proof}

\begin{prop}\label{ext-dense}
The subspace $C^\infty_c(\R^d,\C^m)$ is dense in $(\mathcal{D}(\mathcal{E}_V),||\cdot||_V)$.  
\end{prop}

\begin{proof}
This can be seen in three steps.
\begin{enumerate}
\item[(a)] Compactly supported elements of $\mathcal{D}(\mathcal{E}_V)$ are dense: if $\psi\in \mathcal{D}(\mathcal{E}_V)$ and $\eta\in C^\infty_c(\R^d)$ is such that $\eta(0)=1$, then $\eta(\eps x)\psi(x)$ converges to $\psi$ in $\mathcal{D}(\mathcal{E}_V)$ when $\eps\rightarrow0$.
\item[(b)] Bounded elements of $\mathcal{D}(\mathcal{E}_V)$ are dense: if $\psi\in \mathcal{D}(\mathcal{E}_V)$, then defining $\psi_R(x):=\psi(x)$ when $|\psi(x)|\leq R$, and $\psi_R(x):=R\frac{\psi(x)}{|\psi(x)|}$ otherwise, one has the convergence $\psi_R\rightarrow \psi$ in $\mathcal{D}(\mathcal{E}_V)$, by the chain rule for Sobolev spaces and dominated convergence.
\item[(c)] Smooth compactly supported elements of $\mathcal{D}(\mathcal{E}_V)$ are dense: if $\psi\in \mathcal{D}(\mathcal{E}_V)$ is bounded and compactly supported and $\{\varphi_\eps\}_{\eps>0}$ is a scalar approximate identity, then $\varphi_\eps*\psi\rightarrow\psi$ in $\mathcal{D}(\mathcal{E}_V)$.\end{enumerate}
The omitted details are standard.
\end{proof}

\begin{prop}\label{ext-prop}
The operator\be
\mathcal{H}_V:=-\Delta+V,\ee 
where $\Delta$ acts componentwise and distributionally, defined on the dense domain
\be
\mathcal{D}(\mathcal{H}_V):=\{\psi\in \mathcal{D}(\mathcal{E}_V):\ -\Delta\psi+V\psi\in L^2(\R^d,\C^m)\}\subseteq L^2(\R^d,\C^m),
\ee is self-adjoint. Moreover, \bel\label{ext-quad}
\int_{\R^d}(\mathcal{H}_V\psi,\psi)=\mathcal{E}_V(\psi)\qquad\forall \psi\in \mathcal{D}(\mathcal{H}_V), 
\eel and in particular $\mathcal{H}_V$ is non-negative.
\end{prop}

\begin{proof} The proof follows the standard construction of the Friedrich's extension.

Fix $\phi\in L^2(\R^d,\C^m)$ and consider the linear functional $\psi\mapsto \int_{\R^d}(\psi,\phi)$ on the space $\mathcal{D}(\mathcal{E}_V)$. Since\be
\left|\int_{\R^d}(\psi,\phi)\right|\leq ||\psi||_V\cdot ||\phi|| \qquad\forall \psi\in \mathcal{D}(\mathcal{E}_V),
\ee by Proposition \ref{ext-hilb} and the Riesz Lemma, there exists a unique $N(\phi)\in \mathcal{D}(\mathcal{E}_V)$ such that \bel\label{ext-riesz}
\int_{\R^d}(\psi,\phi) = \int_{\R^d}(\psi,N(\phi))+ \sum_{j=1}^d\int_{\R^d} \left(\frac{\partial \psi}{\partial x_j},\frac{\partial N(\phi)}{\partial x_j}\right) + \int_{\R^d} (V\psi,N(\phi))
\eel for every $\psi\in \mathcal{D}(\mathcal{E}_V)$. Moreover, \be
||N(\phi)||_V\leq ||\phi||.
\ee

Since \eqref{ext-riesz} holds for every $\psi\in C^\infty_c(\R^d,\C^m)$, we see that $N(\phi)\in \mathcal{D}(\mathcal{H}_V)$ and that \be
 N(\phi)-\Delta N(\phi)+V N(\phi)=\phi,
\ee 
in the sense of distributions.

This proves that the operator $1+\mathcal{H}_V$ defined on $\mathcal{D}(\mathcal{H}_V)$ is surjective with a bounded right inverse $N$. Notice that $1+\mathcal{H}_V$ is injective too: if $\phi'\in \mathcal{D}(\mathcal{E}_V)$ and $\phi'+\mathcal{H}_V\phi'=0$, then \be
\int_{\R^d}(\psi,\phi')+ \sum_{j=1}^d\int_{\R^d} \left(\frac{\partial \psi}{\partial x_j},\frac{\partial\phi'}{\partial x_j}\right) + \int_{\R^d} (V\psi,\phi')=0
\ee for every $\psi\in C^\infty_c(\R^d,\C^m)$. By Proposition \ref{ext-dense}, this implies $||\phi'||_V=0$ and hence $\phi'=0$.

Evaluating \eqref{ext-riesz} in $\psi=N(\phi')$ one immediately sees that $N$ is self-adjoint. Hence its inverse $1+\mathcal{H}_V$ and $\mathcal{H}_V$ are self-adjoint.

Identity \eqref{ext-quad} is just a rewriting of \eqref{ext-riesz}.
\end{proof}

\section{The problem of discreteness of the spectrum}\label{schrodpp-sec}

We continue to assume that $V:\R^d\rightarrow H_m$ is non-negative and locally integrable. Thanks to Proposition \ref{ext-prop}, we have the non-negative self-adjoint operator $\mathcal{H}_V$, to which spectral theory may be applied.

We consider the following problem: \emph{for which potentials $V$, does $\mathcal{H}_V$ have discrete spectrum?} Here we adopt a common terminology saying that a self-adjoint operator has discrete spectrum if its spectrum is a discrete subset of $\R$ consisting of eigenvalues of finite multiplicity. 

We have a classical characterization of discreteness of the spectrum, whose proof is identical to the one in the scalar case, as can be found, e.g., in \cite[pp. 190-191]{ko-sh}.

\begin{prop}\label{schrodpp-prop}
$\mathcal{H}_V$ has discrete spectrum if and only if for every $\eps>0$ there is $R<+\infty$ such that 
\bel\label{schrodpp-cond}
\int_{|x|\geq R}|\psi|^2\leq \eps\cdot \mathcal{E}_V(\psi)\qquad\forall \psi\in \mathcal{D}(\mathcal{E}_V).
\eel
\end{prop} 

Observe that, thanks to Proposition \ref{ext-dense}, it is enough to verify \eqref{schrodpp-cond} for every $\psi\in C^\infty_c(\R^d,\C^m)$.

The following statement is a consequence of Proposition \ref{schrodpp-prop}.

\begin{prop}\label{schrodpp-lambda}
Let $V:\R^d\rightarrow H_m$ be non-negative and locally integrable. If $\mathcal{H}_{\lambda(V)}$ has discrete spectrum, then $\mathcal{H}_V$ has discrete spectrum too.
\end{prop}
Proposition \ref{ex-counter} will show that the converse implication fails in general.

Recall that $\lambda(V):=\min_{u\in\C^m:\ |u|=1}(Vu,u)$ is the minimal eigenvalue of $V$, and therefore it is a non-negative and locally integrable scalar potential.

\begin{proof}
By Proposition \ref{schrodpp-prop}, for every $\eps>0$ there exists $R<+\infty$ such that \bel\label{schrodpp-scalar}
\int_{|x|\geq R}|\phi|^2\leq \eps\left(\int_{\R^d}|\nabla \phi|^2+\int_{\R^d}\lambda(V)|\phi|^2\right)\qquad\forall \phi\in C^\infty_c(\R^d).
\eel
Given $\psi\in C^\infty_c(\R^d,\C^m)$, we apply \eqref{schrodpp-scalar} to $\phi=\psi_j$ ($j=1,\dots,m$):\bee
\int_{|x|\geq R}|\psi|^2&=&\sum_{j=1}^m\int_{|x|\geq R}|\psi_j|^2\\
&\leq& \eps\left(\sum_{j=1}^m\int_{\R^d}|\nabla \psi_j|^2+\sum_{j=1}^m\int_{\R^d}\lambda(V)|\psi_j|^2\right)\\
&\leq&\eps\left(\int_{\R^d}|\nabla \psi|^2+\int_{\R^d}(V\psi,\psi)\right).
\eee
By Proposition \ref{schrodpp-prop}, the matrix Schr\"odinger operator $\mathcal{H}_V$ has discrete spectrum.
\end{proof}

\section{The Maz'ya-Shubin characterization of \\ discreteness of the spectrum in the scalar case}\label{mash-sec}

The problem of discreteness of the spectrum of $\mathcal{H}_V$ has been deeply studied in the scalar case ($m=1$). Molchanov obtained in \cite{molcanov} a necessary and sufficient condition, and Maz'ya and Shubin significantly improved it in \cite{ma-sh}. It is a slightly simplified version of their result that we present now.

In order to state it, we recall the notion of \emph{Wiener capacity} from classical potential theory. If $d\geq 3$ and $F\subseteq\R^d$ is compact, we define its capacity as follows:
\bel\label{mash-cap3}
\text{Cap}(F):=\inf\left\{\int_{\R^d}|\nabla \eta|^2\colon \eta\in C^\infty_c(\R^d,[0,1]) \text{ s. t. }\eta=1 \text{ on } F\right\}.
\eel 
Unfortunately, one needs to modify the definition when $d<3$. If $F\subseteq Q$ is compact and $Q\subseteq\R^2$ is a square, we define the \emph{capacity of $F$ relative to $Q$}:
\bel\label{mash-cap2}
\text{Cap}(F,Q):=\inf\left\{\int_{\R^2}|\nabla \eta|^2\colon \eta\in C^\infty_c(2Q,[0,1]) \text{ s. t. }\eta=1 \text{ on } F\right\}.
\eel 
Here $2Q$ denotes the square with same center as $Q$ and double side length.

Finally, if $F\subseteq Q$ is compact and $Q\subseteq \R$ is an interval of length $\ell$, we put $\text{Cap}(F,Q)=0$ if $F=\varnothing$ and $\text{Cap}(F,Q)=\ell^{-1}$ if $F\neq\varnothing$.

If $Q$ is a cube in $\R^d$ ($d\geq3$) and $\gamma>0$, we denote by $\mathcal{N}_\gamma(Q)$ the \emph{set of $\gamma$-negligible subsets of $Q$}, that is, the collection of compact subsets $F$ of $Q$ such that $\text{Cap}(F)\leq\gamma \text{Cap}(Q)$. If $d=1$ or $2$, one has to replace $\text{Cap}(F)$ with $\text{Cap}(F,Q)$. In particular, if $d=1$ and $\gamma<1$, then $\mathcal{N}_\gamma(Q)=\{\varnothing\}$.

We denote by $Q(x,\ell)$ the cube centered at $x\in\R^d$ with sides parallel to the axes and of length $\ell$.

\begin{thm}\label{mash-thm} Let $V\in L^1_{\text{loc}}(\R^d)$ be non-negative. 

Assume that there exists $\gamma>0$ such that\bel\label{mash-cond}
\lim_{x\rightarrow\infty}\  \inf_{F\in \mathcal{N}_\gamma(Q(x,\ell))}\int_{Q(x,\ell)\setminus F}V(y)dy=+\infty\qquad\forall \ell>0.
\eel Then the scalar Schr\"odinger operator $\mathcal{H}_V$ has discrete spectrum. 

Viceversa, if $\mathcal{H}_V$ has discrete spectrum then \eqref{mash-cond} holds for any $\gamma<1$.
\end{thm}

Observe that if $d=1$ the infimum disappears and the characterization is simply in terms of the integrals $\int_{x-\frac{\ell}{2}}^{x+\frac{\ell}{2}}V(y)dy$.

Theorem \ref{mash-thm} and its proof are rather complicated because the hypothesis on $V$ is very weak. Stronger assumptions on the potential may produce characterizations that are easier to check. A particularly interesting example is that of $A_{\infty,\text{loc}}$ potentials. We recall the definition.

\begin{dfn}\label{mash-Ainfloc} A function $V\in L^1_{\text{loc}}(\R^d)$ lies in the class $A_{\infty,\text{loc}}(\R^d)$ if it satisfies one of the following two equivalent properties: 
\begin{enumerate}
\item[\emph{(i)}] there exist $\ell_0,\delta,c>0$ such that
\bel\label{mash-Ainfloc-1}
\left|\left\{x\in Q: V(x)\geq \frac{\delta}{|Q|}\int_QV\right\}\right|\geq c|Q|
\eel 
holds for every cube $Q$ with side length less than or equal to $\ell_0$.
\item[\emph{(ii)}] there exist $\alpha,\beta\in(0,1)$ and $\ell_0>0$ such that for every cube $Q$ with side length less than or equal to $\ell_0$ the following holds:
\bel\label{mash-Ainfloc-2}
\forall A\subseteq Q\text{ measurable }\colon\quad|A|\geq\alpha|Q|\Longrightarrow \int_A V\geq\beta \int_Q V.
\eel 
\end{enumerate}
\end{dfn}
The reader may refer to \cite[Ch. 9]{grafakos} for the proof of the equivalence of (i) and (ii) in the case of $A_\infty$. The proof for the local version of this class is analogous. 

\begin{thm}\label{mash-Ainf} Assume that $V\in A_{\infty,\text{loc}}(\R^d)$. 

If there exists $\ell>0$ such that\bel\label{mash-Ainf-cond}
\lim_{x\rightarrow\infty}\  \int_{Q(x,\ell)}V(y)dy=+\infty,
\eel then the scalar Schr\"odinger operator $\mathcal{H}_V$ has discrete spectrum. 
\end{thm}

Of course \eqref{mash-Ainf-cond} is also necessary for discreteness of spectrum, even without assuming $V\in A_{\infty,\text{loc}}(\R^d)$, as a consequence of Theorem \ref{mash-thm}. 

\begin{proof}
One can use Theorem \ref{mash-thm} and the comparison of Wiener capacity and Lebesgue measure. We omit the details, because this is a particular case of Theorem \ref{gen-thm}. 
\end{proof}

\section{Matrix-valued analogues of $A_{\infty,loc}(\R^d)$}\label{matrixAinf-sec}

This section and the next one are devoted to the proof of a generalization of Theorem \ref{mash-Ainf} to matrix-valued potentials.
The first problem we have to solve is to define an appropriate matrix-valued $A_{\infty,\text{loc}}$ class. Looking at Definition \ref{mash-Ainfloc}, one easily realizes that both condition (i) and condition (ii) may be generalized to the matrix-valued setting without any typographical correction, if one interprets the symbol $\geq$ applied to Hermitian matrices as the inequality between the corresponding quadratic forms: if $V,W\in H_m$ then\be
V\geq W \quad\Longleftrightarrow \quad (Vv,v)\geq (Wv,v)\quad\forall v\in\C^m. 
\ee The integral gives no problem: if $V:\R^d\rightarrow H_m$ is locally integrable and $F\subseteq\R^d$ is compact, then $\int_FV$ is a well defined element of $H_m$, and it is non-negative if $V$ is everywhere non-negative. Therefore we have two definitions.

\begin{dfn}\label{matrixAinf-1} A locally integrable and everywhere non-negative function\newline $V:\R^d\rightarrow H_m$ lies in the class $A_{\infty,loc}(\R^d,H_m)$ if there exist $\ell_0,\delta,c>0$ such that
\bel\label{matrixAinf-ineq-1}
\left|\left\{x\in Q: V(x)\geq \frac{\delta}{|Q|}\int_QV\right\}\right|\geq c|Q|
\eel 
holds for every cube $Q$ with side length less than or equal to $\ell_0$.
\end{dfn}

\begin{dfn}\label{matrixAinf-2} A locally integrable and everywhere non-negative function $V:\R^d\rightarrow H_m$ lies in the class $\widetilde{A_{\infty,loc}}(\R^d,H_m)$ if there exist $\alpha,\beta\in(0,1)$ and $\ell_0>0$ such that for every cube $Q$ with side length less than or equal to $\ell_0$ the following holds:
\bel\label{matrixAinf-ineq-2}
\forall A\subseteq Q\text{ measurable }\colon\quad|A|\geq\alpha|Q|\Longrightarrow \int_A V\geq\beta \int_Q V.
\eel 
\end{dfn}

The next result may be surprising. 

\begin{prop}\label{matrixAinf-strict}
\begin{enumerate}
\item[\emph{(i)}] If $W$ satisfies Definition \ref{matrixAinf-1} with parameters $\ell_0$, $c$ and $\delta$, then it satisfies Definition \ref{matrixAinf-2} with parameters $\ell_0$, $\alpha=1-c/2$ and $\beta=c\delta/2$.
\item[\emph{(ii)}] If $m\geq2$, the class $A_{\infty,loc}(\R^d,H_m)$ is strictly smaller than $\widetilde{A_{\infty,loc}}(\R^d,H_m)$.
\end{enumerate}
\end{prop}

\begin{proof}
(i) If $A\subseteq Q$, $Q$ has side length less than or equal to $\ell_0$, and $|A|\geq \left(1-\frac{c}{2}\right)|Q|$, the intersection of $A$ with the set on the left of \eqref{matrixAinf-ineq-1} has measure $\geq \frac{c}{2}|Q|$, and thus $\int_AW\geq \frac{c\delta}{2}\int_QW$. 

(ii) Consider the set $\mathcal{W}_{d,m}$ of everywhere non-negative functions $W:\R^d\rightarrow H_m$ satisfying the following properties:\begin{enumerate}
\item[(a)] the entries of $W$ are polynomials,
\item[(b)] $\det(W)\equiv0$,
\item[(c)] there is no $u\in \C^m\setminus\{0\}$ for which $(Wu,u)$ is identically zero.
\end{enumerate} 
Notice that $\mathcal{W}_{d,1}=\varnothing$, but $\mathcal{W}_{d,m}\neq\varnothing$ when $m\geq2$. An example of an element of $\mathcal{W}_{1,2}$ is $W_0(x)= \begin{bmatrix}
    1 & x \\
    x & x^2 \\
  \end{bmatrix}, $ and analogous examples may be immediately exhibited for any $d$ whenever $m\geq2$.

We claim that $\mathcal{W}_{d,m}\subseteq\widetilde{A_{\infty,loc}}(\R^d,H_m)\setminus A_{\infty,loc}(\R^d,H_m)$.

Let us prove the claim. Scalar non-negative polynomials are in $A_\infty(\R^n)$ with constants depending only on the degree (see \cite[Section $2$]{ricci-stein}). Thus we can apply \eqref{mash-Ainfloc-2} to the family $\{(Wu,u)\}_{u\in\C^m}$. Recalling that $\int_Q(Wu,u)=\left(\left(\int_QW\right)u,u\right)$, this shows that any matrix-valued polynomial is in $\widetilde{A_{\infty,loc}}(\R^d,H_m)$. 

Next, observe that if $W\in\mathcal{W}_{d,m}$, then $\int_QW>0$ for every cube $Q$. In fact, if this was not the case, there would be $u\in\C^m\setminus\{0\}$ such that $\int_Q(Wu,u)=0$. The non-negativity of $W$ would then force $(Wu,u)$ to vanish identically on $Q$, and hence on $\R^d$, because it is a polynomial. This contradicts property (c) above. Since $\det(W)\equiv0$ and $\int_QW>0$ on every cube, $W$ cannot satisfy \eqref{matrixAinf-ineq-1}, thus proving the claim.\end{proof}

We conclude this section proving a doubling property of elements of $A_{\infty,loc}(\R^d,H_m)$. 

\begin{prop}\label{matrixAinf-doubling}
Given $W\in A_{\infty,loc}(\R^d,H_m)$ and $\ell>0$, there exists a constant $D$ such that the following doubling property holds: if $Q$ is a cube with side length less than or equal to $\ell$ and $Q'\subseteq Q$ has half the side length of $Q$, then \bel\label{matrixAinf-doubling-0}
\int_QV\leq D\int_{Q'}V.
\eel
\end{prop}

\begin{proof} If $W$ satisfies Definition \ref{matrixAinf-1} with parameters $\ell_0$, $c$ and $\delta$, by part (i) of Proposition \ref{matrixAinf-strict}, it also satisfies Definition \ref{matrixAinf-2} with parameters $\ell_0$, $\alpha$ and $\beta$. If $Q$ and $Q'$ are as in the statement, one can find a sequence $Q_0\subseteq Q_1\cdots\subseteq Q_N$ of nested cubes such that $Q_0=Q'$, $Q_N=Q$ and $|Q_k|\geq \alpha|Q_{k+1}|$ ($k=0,\dots,N-1$), with $N$ depending on $\alpha$ and the dimension $d$. Using \eqref{matrixAinf-ineq-2}, we get
\bel\label{matrixAinf-doubling-1}
\int_QV\leq \beta^{-N}\int_{Q'}V.
\eel
 for every $Q$ with side length less than or equal to $\ell_0$.

To go beyond the threshold $\ell_0$, one can split any cube $Q$ with side length less than or equal to $\ell$ in $2^N$ cubes of side length less than or equal to $\ell_0/2$, and use \eqref{matrixAinf-doubling-1} to see that if $Q'$ and $Q''$ are two of the small cubes that are adjacent, then $\int_{Q'}V\leq \beta^{-N}\int_{Q''}V$. This clearly implies \eqref{matrixAinf-doubling-0}.

\end{proof}

\section{The generalization of Theorem \ref{mash-Ainf}\\ to matrix-valued potentials}\label{gen-sec}

We finally state and prove our generalization of Theorem \ref{mash-Ainf}.

If $A\in H_m$, we denote the minimal eigenvalue of $A$ by $\lambda(A)$. 

\begin{thm}\label{gen-thm} Let $V:\R^d\rightarrow H_m$ be locally integrable and non-negative. Consider the following conditions:
\begin{enumerate}
\item[\emph{(a)}] $\mathcal{H}_V$ has discrete spectrum,
\item[\emph{(b)}] for every $\ell>0$ we have \be
\lim_{x\rightarrow\infty}\lambda \left(\int_{Q(x,\ell)}V\right)=+\infty,
\ee
\item[\emph{(c)}] there exists $\ell_1>0$ such that \bel\label{gen-singleell}
\lim_{x\rightarrow\infty}\lambda \left(\int_{Q(x,\ell_1)}V\right)=+\infty.
\eel
\end{enumerate}
Then:\begin{enumerate}
\item[\emph{(i)}] We have the implications \emph{(a)}$\Rightarrow$\emph{(b)}$\Rightarrow$\emph{(c)}. 

\item[\emph{(ii)}] If $V\in A_{\infty,loc}(\R^d,H_m)$, then \emph{(c)}$\Rightarrow$\emph{(a)}.

\item[\emph{(iii)}] For every $d$ and $m\geq2$, there exists $V_0\in \widetilde{A_{\infty,loc}}(\R^d,H_m)$ such that $\mathcal{H}_{V_0}$ has not discrete spectrum, but \emph{(b)} holds.

\end{enumerate}
\end{thm}

\begin{proof}

(i) The implication (b)$\Rightarrow$(c) is obvious. 

Assume that $\mathcal{H}_V$ has discrete spectrum. Fix $\ell>0$ and let $\eta\in C^\infty_c(\R^d,[0,1])$ be non-trivial and identically $0$ outside $Q(0,\ell)$. If $x\in \R^d$ and $u\in \C^m$ has norm $1$, we put 
\be\eta_{x,u}(y):=\eta\left(y-x\right)u.\ee Fix $\eps>0$. By Proposition \ref{schrodpp-prop}, there is $R$ such that \eqref{schrodpp-cond} holds.

 If $Q(x,\ell)\subseteq \{|y|\geq R\}$,\bee
\int_{\R^d}\eta^2=\int_{\R^d}|\eta_{x,u}|^2&\leq&\eps\left(\int_{\R^d}|\nabla\eta_{x,u}|^2+\int_{\R^d}(V\eta_{x,u},\eta_{x,u})\right)\\
&\leq &\eps\left(\int_{\R^d}|\nabla\eta|^2+\int_{Q(x,\ell)}(Vu,u)\right).
\eee If $\eps_0$ is such that $\eps_0\int_{\R^d}|\nabla\eta|^2\leq\frac{1}{2}\int_{\R^d}\eta^2$ and $\eps\leq \eps_0$ this implies \be
\int_{\R^n}\eta^2\leq 2\eps\int_{Q(x,\ell)}(Vu,u)=2\eps\left(\left(\int_{Q(x,\ell)}V\right)u,u\right).
\ee Taking the minimum as $u$ varies on the unit sphere of $\C^m$, we get \be\lambda\left(\int_{Q(x,\ell)}V\right)\geq(2\eps)^{-1}\int_{\R^n}\eta^2.\ee
By the arbitrariness of $\eps$, we get the thesis.\newline

(ii) If $Q$ is a cube of side $\ell$, we define the non-negative matrix \be M(Q):=\ell^{2-n}\int_QV.\ee It will be useful to introduce the collections of dyadic cubes\be
\mathcal{D}_N:=\left\{2^Nx+[0,2^N]^d:\ x\in \Z^d\right\}\qquad(N\in\Z).
\ee
By Proposition \ref{matrixAinf-doubling}, there exists a constant $D<+\infty$ such that if $Q_1\subseteq Q_2$ and $Q_j\in\mathcal{D}_{N_j}$($j=1,2$), we have the bound\bel\label{gen-doubling}
M(Q_2)\leq D^{N_2-N_1}M(Q_1).
\eel

Fix $N_0\in \Z$ such that $2^{N_0}\geq \ell_1$. Given $N\in\N$, by assumption there exists $R<+\infty$ such that if $Q\in \mathcal{D}_{N_0}$ intersects $\{|y|\geq R\}$, we have $M(Q)\geq D^N\mathbb{I}_d$. Let $Q$ be any such cube. By \eqref{gen-doubling} we have\be
M(Q')\geq \mathbb{I}_m \qquad\forall Q'\in \mathcal{D}_{N_0-N}: Q'\subseteq Q.
\ee Fix now $Q'\subseteq Q$ such that $Q'\in \mathcal{D}_{N_0-N}$. If $N$ is so large that $2^{N_0-N}\leq \ell_0$, \eqref{matrixAinf-ineq-1} tells us that
\bel\label{gen-Q'bound}
V(x)\geq \frac{\delta}{|Q'|}\int_{Q'}V=\delta 4^{N-N_0}M(Q')\geq \delta 4^{N-N_0} \mathbb{I}_d,
\eel on a set $E(Q')\subseteq Q'$ of measure $\geq c|Q'|$. If $\psi\in C^\infty_c(\R^d,\C^m)$, we integrate the trivial inequality \be
|\psi(x)|^2\leq 2|\psi(x)-\psi(y)|^2+2|\psi(y)|^2
\ee as $(x,y)$ varies in $Q'\times E(Q')$. We get\be
|E(Q')|\int_{Q'}|\psi|^2\leq 2\int_{Q'\times E(Q')}|\psi(x)-\psi(y)|^2dxdy+ 2|Q'|\int_{E(Q')}|\psi|^2.
\ee Using \eqref{gen-Q'bound}, the lower bound on $|E(Q')|$ and Poincar\'e inequality \be
\int_{Q'\times Q'}|\psi(x)-\psi(y)|^2\leq C_d\ell_{Q'}^2|Q'|\int_{Q'}|\nabla\psi|^2,
\ee we find\bee
c\int_{Q'}|\psi|^2&\leq& 2C_d\ell_{Q'}^2\int_{Q'}|\nabla\psi|^2+2\delta^{-1}4^{N_0-N}\int_{E(Q')}(V\psi,\psi)\\
&\leq& 2C_d4^{N_0-N}\int_{Q'}|\nabla\psi|^2+2\delta^{-1}4^{N_0-N}\int_{Q'}(V\psi,\psi).
\eee In the second line we used the non-negativity of $V$.
Summing over all the cubes $Q'\in\mathcal{D}_{N_0-N}$ contained in $Q$, and then over all $Q\in\mathcal{D}_{N_0}$ intersecting $\{|y|\geq R\}$, we obtain  \be
\int_{|y|\geq R}|\psi|^2\leq 2c^{-1}4^{N_0-N}\max\{\delta^{-1},C_d\}\mathcal{E}_V(\psi).
\ee By the arbitrariness of $N$, we conclude that condition \eqref{schrodpp-cond} of Proposition \ref{schrodpp-prop} holds for every $\psi\in C^\infty(\R^d,\C^m)$.

(iii) We begin by proving the statement when $d=1$ and $m=2$. In this case we claim that $V_0(x)= \begin{bmatrix}
    x^4 & x^5 \\
    x^5 & x^6 \\
  \end{bmatrix}=x^4\begin{bmatrix}
    1 & x \\
    x & x^2 \\
  \end{bmatrix}$ is one of the potentials we are looking for. 
  
Let $\eta$ be a non-trivial test function identically equal to $1$ on $[-1,1]$, and define the vector-valued function $\psi_x(y):=\eta(y-x)(-y,1)$ ($x\in\R$). Notice that\bel\label{gen-L2}
\int_\R|\psi_x|^2\geq \int_{|x-y|\leq1}(y^2+1)\geq cx^2,
\eel
where $c>0$ is independent of $x$. Since $V_0(y)$ annihilates the vector $(-y,1)$ for every $y\in\R$, we also have\bel\label{gen-energy}
\mathcal{E}_{V_0}(\psi_x)=\int_\R|\eta'(y-x)(-y,1)+\eta(y-x)(-1,0)|^2\leq Cx^2,
\eel
for every large enough $x$, where $C$ is independent of $x$. Letting $x$ tend to $\infty$ and taking into account the bounds \eqref{gen-L2} and \eqref{gen-energy}, we see that condition \eqref{schrodpp-cond} of Proposition \ref{schrodpp-prop} cannot hold. This proves that $\mathcal{H}_{V_0}$ has not discrete spectrum.

To complete the proof of (iii) in the case $d=1$ and $m=2$, we estimate for every $x\in\R$ and $\ell>0$ the following quantities:\bee
 \det\left(\int_{x-\ell/2}^{x+\ell/2}V_0(y)dy\right) &=&
  \det
  \begin{bmatrix}
   \frac{(x+\ell/2)^5-(x-\ell/2)^5}{5} & \frac{(x+\ell/2)^6-(x-\ell/2)^6}{6} \\
    \frac{(x+\ell/2)^6-(x-\ell/2)^6}{6}  &  \frac{(x+\ell/2)^7-(x-\ell/2)^7}{7} \\
  \end{bmatrix}=\frac{\ell^4x^8}{12}\\&+&\text{terms of lower degree in $x$}, \\
 \text{tr}\left(\int_{x-\ell/2}^{x+\ell/2}V_0(y)dy\right) &=& \ell x^6+\text{terms of lower degree in $x$}.
  \eee
Since $\det(A)\leq \lambda(A)\text{tr}(A)$ for every non-negative $2\times 2$ matrix $A$, we conclude that 
\bel\label{gen-diverge}\liminf_{x\rightarrow\pm\infty} \frac{\lambda\left(\int_{x-\ell/2}^{x+\ell/2}V_0(y)dy\right)}{x^2}\geq\frac{\ell^3}{12},\eel which is stronger than condition (b) of the statement. This completes the proof of (iii) in the case under consideration.\newline

If $d\geq 1$ and $m=2$, we consider the potential $V_d(x)=\sum_{j=1}^dV_0(x_j)$. Notice that $H_{V_d}$ is the sum of the $n$ pairwise commuting operators obtained by letting $H_{V_d}$ act on each variable separately. By spectral theory, $H_{V_d}$ has not discrete spectrum. To see that condition (b) is satisfied by $V_d$, we use the concavity of $\lambda$: \be
\lambda(A+B)\geq\lambda(A)+\lambda(B)\qquad\forall A,B\geq0,\ee which follows from the fact that $\lambda$ is the infimum of the family of linear functionals $\{A\mapsto(Au,u)\}_{|u|=1}$. 
For every $x\in\R^d$ and $\ell>0$, we have\bee
 \lambda\left(\int_{Q(x,\ell)}V_d(y)dy\right) &=&  \lambda\left(\ell^{d-1}\sum_{j=1}^d\int_{x_j-\ell/2}^{x_j+\ell/2}V_d(y)dy\right) \\
 &\geq & \sum_{j=1}^d\ell^{d-1}\lambda\left(\int_{x_j-\ell/2}^{x_j+\ell/2}V_d(y)dy\right).
 \eee  By \eqref{gen-diverge}, we conclude that
\be\liminf_{x\rightarrow\infty} \frac{\lambda\left(\int_{Q(x,\ell)}V_d(y)dy\right)}{|x|^2}\geq c_d\ell^{d+2}.\ee

\par For the remaining case $d\geq 1$, $m\geq 3$, we put $V_{d,m}(x)=\begin{bmatrix}
  V_d(x) & \mathbb{O}_{2,m-2} \\
   \mathbb{O}_{m-2,2}  &  |x|^2\mathbb{I}_{m-2} \\
  \end{bmatrix}$, where $\mathbb{O}$ and $\mathbb{I}$ are the zero and identity matrices of the dimensions indicated by the subscripts. We omit the elementary verification that $H_{V_{d,m}}$ has not discrete spectrum, and that $V_{d,m}$ satisfies condition (b).

\end{proof}

\section{A natural extension of the Maz'ya-Shubin condition is not sufficient when $m\geq2$}\label{nomash-sec}

Our discussion until this point leaves open the possibility that a natural extension of Theorem \ref{mash-thm} may hold for matrix-valued potentials. In particular, the statements of Theorem \ref{mash-thm} and Theorem \ref{gen-thm} could suggest the conjecture that a necessary and sufficient condition for discreteness of the spectrum of $\mathcal{H}_V$ may be the validity of\bel\label{nomash-wrong}
\lim_{x\rightarrow\infty}\ \inf_{F\in\mathcal{N}_\gamma(Q(x,\ell))}\lambda\left(\int_{Q(x,\ell)\setminus F} V(y)dy\right)=+\infty\qquad\forall \ell>0,
\eel 
for some, and hence every, $\gamma\in(0,1)$.

This condition is indeed necessary. This can be seen adapting the necessity argument of \cite{ma-sh}, in the same spirit of our proof of the implication (i)$\Rightarrow$(ii) of Theorem \ref{gen-thm}. More precisely, one can take the argument of page $931$ of \cite{ma-sh} and replace the scalar test function $\chi_\delta (1-P_F)$ with the vector-valued test function $\chi_\delta (1-P_F)v$, where $v\in\C^m$ has norm $1$ and is otherwise arbitrary. Carrying out their computations and minimizing in $v$ at the end as in the aforementioned proof of Theorem \ref{gen-thm}, one sees that if the vector-valued operator $\mathcal{H}_V$ has discrete spectrum, then $V$ satisfies condition \eqref{nomash-wrong}. We omit the details since they are not interesting for us, in view of the fact that the converse does not hold. 

By part (iii) of Theorem \ref{gen-thm} we have seen that for every $d\geq1$ and $m\geq2$, there is a non-negative $m\times m$ polynomial potential $W_{d,m}$ on $\R^d$ such that $H_{W_{d,m}}$ has not discrete spectrum, $W_{d,m}\in \widetilde{A_{\infty,loc}}(\R^d,H_m)$, and $\lim_{x\rightarrow\infty}\lambda\left(\int_{Q(x,\ell)}W_{n,d}\right)=+\infty$ for every $\ell>0$. We are going to show that such a potential $W_{d,m}$ satisfies \eqref{nomash-wrong} for $\gamma$ small enough.

Assume that $d\geq3$. If $F\in\mathcal{N}_\gamma(Q(x,\ell))$, we can use the comparison between Lebesgue measure and capacity
\be |F|\leq C_d\text{Cap}(F)^\frac{d}{d-2},\ee
and the fact that $\text{Cap}(Q(x,\ell))=\ell^{d-2}\text{Cap}(Q(0,1))$, to conclude that \be|F|\leq C_d\gamma^\frac{d}{d-2}|Q(x,\ell)|.\ee If $\gamma$ is small, Definition \ref{matrixAinf-2} implies that $\int_{Q(x,\ell)\setminus F}W_{d,m}\geq \beta\int_{Q(x,\ell)}W_{d,m}$ for some $\beta>0$ depending on $W_{d,m}$, and hence \be
\lim_{x\rightarrow\infty}\ \inf_{F\in\mathcal{N}_\gamma(Q(x,\ell))}\lambda\left(\int_{Q(x,\ell)\setminus F}W_{d,m}\right)\geq \beta\cdot\lim_{x\rightarrow\infty}\lambda\left(\int_{Q(x,\ell)}W_{d,m}\right)=+\infty,\ee as we wanted. We omit the elementary observations needed to cover the cases $d=1$ and $d=2$.

\section{Relation of $A_{\infty,\text{loc}}(\R^d,H_m)$\\ with the Treil-Volberg matrix $A_2$ class}\label{tv-sec}

In \cite{treil-volberg}, Treil and Volberg introduced a Muckenhoupt $A_2$ class of non-negative matrix-valued functions. In this subsection we prove that the local version of this class is contained in $A_{\infty,loc}(\R^d,H_m)$. 

\begin{dfn}[cf. \cite{treil-volberg} ]\label{treil-volberg-def}
A non-negative function $W:\R^d\rightarrow H_m$ belongs to $A_{2,loc}(\R^d,H_m)$ if it is almost everywhere invertible, both $W$ and $W^{-1}$ are locally integrable and there exists $\ell_0>0$ such that\be
[W]_{A_{2,\text{loc}}}:=\sup_Q \left|\left| \left(\frac{1}{|Q|}\int_Q W\right)^{\frac{1}{2}}\left(\frac{1}{|Q|}\int_Q W^{-1}\right)^\frac{1}{2} \right|\right|_{op}<+\infty,
\ee where $Q$ varies over all cubes with side length less than or equal to $\ell_0$.
\end{dfn}

\begin{prop}
$A_{2,loc}(\R^d,H_m)\subseteq A_{\infty,loc}(\R^d,H_m)$.
\end{prop}

\begin{proof} Fix $W\in A_{2,loc}(\R^d,H_m)$. We prove that $W$ satisfies Definition \ref{matrixAinf-1} estimating the measure of the complement of the set appearing in \eqref{matrixAinf-ineq-1}. If $A$ and $B$ are two non-negative matrices, $A\ngeq B$ is not equivalent to $A<B$ (except in the case $d=1$), but we can use the fact that when $A$ is invertible, $A\geq B$ if and only if $||B^{\frac{1}{2}}A^{-\frac{1}{2}}||_{op}\leq 1$ (Lemma $V.1.7$ of \cite{bhatia}). Since $W(x)>0$ almost everywhere, this allows to write (for $Q$ any cube of side $\leq \ell_0$)\bee
&&\left|\left\{x\in Q: W(x)\ngeq \frac{\delta}{|Q|}\int_QW\right\}\right|=\left|\left\{x\in Q: \left|\left|\left(\frac{\delta}{|Q|}\int_QW\right)^{\frac{1}{2}}W(x)^{-\frac{1}{2}}\right|\right|_{op}>1\right\}\right|\\
&\leq& \int_Q\left|\left|\left(\frac{\delta}{|Q|}\int_QW\right)^{\frac{1}{2}}W(x)^{-\frac{1}{2}}\right|\right|_{op}^2dx\\
&=& \delta\int_Q\left|\left|\left(\frac{1}{|Q|}\int_QW\right)^{\frac{1}{2}}W(x)^{-1}\left(\frac{1}{|Q|}\int_QW\right)^{\frac{1}{2}}\right|\right|_{op}dx,
\eee where the last equality follows from the identity $||A^*A||_{op}=||A||^2$ which holds for any matrix $A$. Since $\int_Q||U||_{op}\leq m||\int_QU||_{op}$ for any non-negative and integrable $U:Q\rightarrow H_m$ (Lemma $3.1$ of \cite{treil-volberg}), we find\bee
&&\delta\int_Q\left|\left|\left(\frac{1}{|Q|}\int_QW\right)^{\frac{1}{2}}W(x)^{-1}\left(\frac{1}{|Q|}\int_QW\right)^{\frac{1}{2}}\right|\right|_{op}dx\\
&\leq& \delta m|Q|\left|\left|\left(\frac{1}{|Q|}\int_QW\right)^{\frac{1}{2}}\frac{1}{|Q|}\int_QW(x)^{-1}dx\left(\frac{1}{|Q|}\int_QW\right)^{\frac{1}{2}}\right|\right|_{op}\\
&=&\delta m|Q|\left|\left|\left(\frac{1}{|Q|}\int_QW\right)^{\frac{1}{2}}\left(\frac{1}{|Q|}\int_QW(x)^{-1}dx\right)^{\frac{1}{2}}\right|\right|_{op}^2\leq \delta m[W]_{A_{2,loc}}^2|Q|.
\eee
Putting everything together, we conclude that \be
\left|\left\{x\in Q: W(x)\geq \frac{\delta}{|Q|}\int_QW\right\}\right|> (1-\delta m[W]_{A_{2,loc}}^2)|Q|.
\ee When $\delta$ is small enough, this is inequality \eqref{matrixAinf-ineq-1}.
\end{proof}

It is worth noticing that matrix-valued $A_p$ classes for $p\notin\{2,\infty\}$ were introduced in \cite{volberg}. The definition is rather different from the one of $A_2$ and it may be interesting to investigate how they are related to $A_{\infty,loc}(\R^d,H_m)$.

\section{Oscillation of subspace-valued mappings}\label{osc-sec} 

The goal of this section is to define a notion of oscillation for mappings defined on $\R^d$ and whose values are linear subspaces of $\C^m$. This will allow us to formulate in Section \ref{osc2-sec} a new sufficient condition for discreteness of the spectrum of matrix Schr\"odinger operators.

We denote by $\mathcal{V}(m)$ the set of non-trivial linear subspaces of $\C^m$. 

A mapping \be\mathcal{S}:\R^d\longrightarrow \mathcal{V}(m)\ee is said to be measurable if there exist finitely many measurable mappings \be v_1,\dots,v_k:\R^d\rightarrow \C^m\ee such that $\mathcal{S}(x)$ is the span of $\{v_1(x),\dots,v_k(x)\}$ for every $x\in\R^d$.

If $\mathcal{S}$ is as above and $Q\subseteq\R^d$ is a cube, we introduce the set of \emph{unit sections} of $\mathcal{S}$ on $Q$:\be
\Gamma(Q,\mathcal{S}):=\{v:Q\rightarrow \C^m \text{ meas.}\colon v(x)\in\mathcal{S}(x)\text{ and } |v(x)|=1\quad \text{a.e. } x\in Q\}.
\ee 
The measurability in $x$ of $\mathcal{S}(x)$ guarantees that there are plenty of unit sections. 

Notice that the terminology is consistent with the usual one in geometry, if we identify $\mathcal{S}$ with the vector bundle $\{(x,v)\in\R^d\times\C^m:\ v\in \mathcal{S}(x)\}$.

We are ready to give our key definition.

\begin{dfn}\label{osc-dfn} Let $\mathcal{S}:\R^d\longrightarrow \mathcal{V}(m)$ be measurable. The oscillation of $\mathcal{S}$ on a cube $Q$ is the following quantity:
\bel
\omega(Q,\mathcal{S}):=\inf_{v\in\Gamma(Q,\mathcal{S})}\sqrt{\frac{1}{|Q|}\int_Q\left|v(y)-\frac{1}{|Q|}\int_Qv\right|^2dy}.
\eel 
\end{dfn}

Notice that \be
\sqrt{\frac{1}{|Q|}\int_Q\left|v(y)-\frac{1}{|Q|}\int_Qv\right|^2dy}=\inf_{b\in\C^m}\sqrt{\frac{1}{|Q|}\int_Q\left|v(y)-b\right|^2dy}
\ee and hence $\omega(Q,\mathcal{S})$ is the distance in $L^2(Q,\C^m)$ between the set $\Gamma(Q,\mathcal{S})$ and the set of constant vector fields. 

By using the fact that $|v|\equiv1$, it is also easy to verify that 
\bel\label{osc-id}
\omega(Q,\mathcal{S})^2= 1-\sup_{v\in\Gamma(Q,\mathcal{S})}\left|\frac{1}{|Q|}\int_Qv\right|^2.
\eel It is then obvious that $\omega(Q,\mathcal{S})\in[0,1]$. The following proposition gives us an idea of which feature of $\mathcal{S}$ on $Q$ is measured by $\omega(Q,\mathcal{S})$.

\begin{prop}\label{osc-prop}\begin{enumerate}
\item[\emph{(i)}] $\omega(Q,\mathcal{S})=0$ if and only if there is a subset $F\subseteq Q$ such that $|Q\setminus F|=0$ and\bel\label{osc-intersec}
\cap_{x\in F}\mathcal{S}(x)\neq \{0\},
\eel i.e., there exists $u\in\C^m\setminus \{0\}$ such that $u\in \mathcal{S}(x)$ for almost every $x\in Q$.

\item[\emph{(ii)}] If $T:Q\rightarrow Q$ is bijective and such that, for every $E\subseteq Q$, $T^{\leftarrow}(E)$ is measurable if and only if $E$ is measurable and $|T^{\leftarrow}(E)|=|E|$, then \bel\label{osc-rearrang}
\omega(Q,\mathcal{S}\circ T)=\omega(Q,\mathcal{S}).
\eel
\end{enumerate}
\end{prop}

Part (i) tells us that the quantity $\omega(Q,\mathcal{S})$ does not see non-generic oscillations: if, as $x$ varies in $Q$, $\mathcal{S}(x)$ rotates around a fixed axis, then $\omega(Q,\mathcal{S})=0$.

Part (ii) says that the oscillation introduced above is rearrangement-invariant. Notice that $\omega(\mathcal{S},Q)$ only depends on the values of $\mathcal{S}$ on $Q$ and therefore the left-hand side of \eqref{osc-rearrang} is meaningful.

\begin{proof} (i) The non-trivial direction is the \emph{only if}. If $\omega(Q,\mathcal{S})=0$, by \eqref{osc-id} there is a sequence $\{v_k\}_{k\in\N}\subseteq \Gamma(Q,\mathcal{S})$ such that $\lim_{k\rightarrow+\infty}\left|\frac{1}{|Q|}\int_Qv_k\right|=1$. Passing to a subsequence, we can suppose that there is $u_0$ in the unit sphere of $\C^m$ to which $\left\{\frac{1}{|Q|}\int_Qv_k\right\}_{k\in\N}$ converges as fast as we like. In particular, we can assume that \bel\label{osc-k^2}
\frac{1}{|Q|}\int_Q\Re (v_k,u_0)\geq 1-\frac{1}{k^2}\qquad\forall k\in\N,
\eel where $\Re(v_k,u_0)$ denotes the real part of the scalar product, which is pointwise $\leq 1$ by the Cauchy-Schwarz inequality. Consider the sets \be A_{k,n}:=\{x\in Q: \Re(v_k(x),u_0)> 1-1/n\} \qquad (k,n\in\N).\ee It is easy to see that \eqref{osc-k^2} implies the inequality $|Q\setminus A_{k,n}|\leq \frac{n}{k^2}|Q|$. Since $\sum_{k=1}^{+\infty}|Q\setminus A_{k,n}|<+\infty$ for every $n$, the Borel-Cantelli Lemma gives\be
|\cup_{\ell\in \N}\cap_{k\geq \ell}A_{k,n}|=|Q|\qquad\forall n,
\ee and hence \be
|\cap_{n\in\N}\cup_{\ell\in \N}\cap_{k\geq \ell}A_{k,n}|=|Q|.
\ee Unravelling the notation, this means that there is a set of full measure $F$ such that $\Re(v_k(x),u_0)$ converges to $1$ at every $x\in F$. Since $|u_0|=|v_k|\equiv1$, the strict convexity of the sphere implies that $v_k(x)$ converges to $u_0$ at every $x\in F$. Since we may assume that $v_k(x)\in\mathcal{S}(x)$ for every $k$ and every $x\in F$, we conclude that $u_0\in\cap_{x\in F}\mathcal{S}(x)$, as we wanted.\newline 
(ii) It immediately follows from Definition \ref{osc-dfn} and the elementary fact that if $g$ is a measurable function defined on $Q$, then $\int_Qg\circ T=\int_Qg$.
\end{proof}

The following proposition complements Proposition \ref{osc-prop}. It states that if $\cap_{x\in Q}\mathcal{S}(x)=\{0\}$ in a certain quantitative sense, then $\omega(Q,\mathcal{S})$ has an explicit lower bound.

\begin{prop}\label{osc-prop-2}\begin{enumerate}
\item[\emph{(i)}] Let $\mathcal{S}:\R^d\rightarrow\mathcal{V}(m)$ be measurable and let $Q$ be a cube. Assume that for some $N\in\N$ we have the partition 
\be Q=\cup_{j=1}^NA_j,\ee 
where $A_j$ is measurable for every $j$. Put \be
\eta:=\inf_{j=1,\dots,N}\frac{|A_j|}{|Q|}.
\ee 
Assume moreover that there is a $\delta>0$ such that the following property holds for each $j$: for every $x\in A_j$ and $v\in \mathcal{S}(x)$, there exists $k\neq j$ such that
\bel\label{osc-angle}
|(v,w)|\leq(1-\delta)|v||w| \qquad\forall w\in\mathcal{S}(y),\ \forall y\in A_k.
\eel 
Then 
\bel\label{osc-lower-bound}
\omega(Q,\mathcal{S})\geq \sqrt{\delta\eta}.
\eel

\item[\emph{(ii)}] If $\mathcal{S}$ is constant on each $A_j$, i.e., $\mathcal{S}(x)=\mathcal{S}_j$ for every $x\in A_j$, then the property above is satisfied for some $\delta>0$ if and only if $\cap_{j=1}^N\mathcal{S}_j=\{0\}$. The value of $\delta$ may be chosen to depend only on the subspaces $\{\mathcal{S}_j\}_{j=1}^N$.
\end{enumerate}
\end{prop}

\begin{proof}
(i): Let $u\in\Gamma(Q,\mathcal{S})$. Then \bel\label{osc-angle-est}
\left|\int_Qu\right|^2=\left|\sum_{j=1}^N\int_{A_j}u\right|^2=\sum_{j=1}^N\left|\int_{A_j}u\right|^2+\sum_{j=1}^N\int_{A_j}\int_{A_j^c}(u(x),u(y))dydx.
\eel 
Fix $x\in A_j$ such that $u(x)\in\mathcal{S}(x)$ (almost every $x$ has this property). By the hypothesis, there exists $k$ such that \eqref{osc-angle} holds. Since $|A_k|\geq\eta|Q|$, we have\bee
\left|\int_{A_j^c}(u(x),u(y))dy\right|&\leq&\left|\int_{(A_j\cup A_k)^c}(u(x),u(y))dy\right|+\left|\int_{A_k}(u(x),u(y))dy\right|\\
&\leq& (|Q|-|A_j|-|A_k|)+|A_k|(1-\delta)\\
&=&(|Q|-|A_j|)-|A_k|\delta\leq (|Q|-|A_j|)-\delta\eta|Q|,
\eee where in the second line we used the Cauchy-Schwarz inequality and \eqref{osc-angle}. Identity \eqref{osc-angle-est} then gives\bee
\left|\int_Qu\right|^2&\leq& \sum_{j=1}^N|A_j|^2+\sum_{j=1}^N |A_j|(|Q|-|A_j|)-\sum_{j=1}^N|A_j|\delta\eta|Q|\\
&=&\left(\sum_{j=1}^N|A_j|\right)|Q|(1-\delta\eta)=(1-\delta\eta)|Q|^2.
\eee Recalling \eqref{osc-id} and the arbitrariness of $u\in \Gamma(Q,\mathcal{S})$, we get \eqref{osc-lower-bound}.
\ \newline

(ii): If $\cap_{j=1}^N\mathcal{S}_j\neq\{0\}$, one can take a non-zero vector $v$ in the intersection to see that \eqref{osc-angle} cannot be satisfied for any $\delta>0$. 

Now assume that the property does not hold for a given $\delta>0$. In this case there is $j_\delta$ and $v_\delta\in\mathcal{S}_{j_\delta}$ of unit norm such that for every $k\neq j_\delta$, there exists $w_{\delta,k}\in\mathcal{S}_k$ of unit norm satisfying \bel\label{osc-contrad}
|(v_\delta,w_{\delta,k})|>(1-\delta).
\eel If the property does not hold for any $\delta$, extracting a subsequence $\delta_n$ such that $\delta_n\rightarrow0$ we can assume that $j_{\delta_n}\equiv j_0$, $v_{\delta_n}$ converges to $v$, and $w_{\delta_n,k}$ converges to $w_k$ for every $k\neq j_0$. Passing to the limit in \eqref{osc-contrad}, we find $|(v,w_k)|=1$. Since $|v|=|w_k|=1$, the equality condition for the Cauchy-Schwarz inequality tells us that $w_k=e^{i\phi_k}v$ for some $\phi_k\in\R$. Hence $v\in \cap_{j=1}^N\mathcal{S}_j$, i.e., the intersection is not trivial.

The statement about the value of $\delta$ follows by observing that in the above argument we never used the sets $A_j$. \end{proof}

By Proposition \ref{schrodpp-prop}, we know that the discrete spectrum property has to do with the behavior of the potential at infinity. For this reason, we introduce an asymptotic variant of Definition \ref{osc-dfn}.

\begin{dfn}\label{osc-asympt} Let $\mathcal{S}:\R^d\rightarrow\mathcal{V}(m)$ be measurable. Given $\ell>0$, the asymptotic oscillation at scale $\ell$ is defined as follows:\be
\omega_{\infty}(\ell,\mathcal{S}):=\liminf_{x\rightarrow\infty,\ x\in\Z^d}\omega(Q(x\ell, \ell),\mathcal{S}).
\ee
\end{dfn}


\section{A sufficient condition for discreteness of the spectrum}\label{osc2-sec}

\begin{thm}\label{osc2-thm}
Let $V:\R^d\rightarrow H_m$ be an everywhere non-negative and locally integrable potential. Assume that there are two measurable mappings\be
\mathcal{S},\mathcal{L}:\R^d\longrightarrow\mathcal{V}(m),
\ee
such that the following properties hold:\begin{enumerate}
\item[\emph{(a)}] for every $x\in\R^d$ we have the $V(x)$-invariant orthogonal decomposition \be
\C^m=\mathcal{S}(x)\oplus\mathcal{L}(x).
\ee 
\item[\emph{(b)}] $\lim_{x\rightarrow\infty}\lambda(V(x)_{|\mathcal{L}(x)})=+\infty$,
\item[\emph{(c)}] $\limsup_{\ell\rightarrow0+}\ell^{-1}\omega_\infty(\ell,\mathcal{S})=+\infty$.
\end{enumerate}

Then $\mathcal{H}_V$ has discrete spectrum.
\end{thm}

Before proving the theorem, we discuss a bit the hypotheses and the idea of the proof.\newline

The $V(x)$-invariance of hypothesis (a) means that $V(x)\mathcal{S}(x)\subseteq \mathcal{S}(x)$ and $V(x)\mathcal{L}(x)\subseteq \mathcal{L}(x)$ for every $x\in\R^d$. In other words, $\mathcal{S}(x)$ and $\mathcal{L}(x)$ are direct sums of eigenspaces of $V(x)$. By hypothesis (b) the minimal eigenvalue of the Hermitian matrix $V(x)$ restricted to $\mathcal{L}(x)$ is required to diverge at $\infty$. Finally, hypothesis (c) puts no constraint on the eigenvalues of $V(x)$ on $\mathcal{S}(x)$, but it states that $\mathcal{S}$ oscillates at $\infty$ or, more precisely, the asymptotic oscillation $\omega_\infty(\ell,\mathcal{S})$ does not decay too fast when $\ell$ tends to $0$. 

Succinctly, Theorem \ref{osc2-thm} gives a sufficient condition for discreteness of the spectrum that holds for potentials whose minimal eigenvalue $\lambda(V)$ does not necessarily diverge at infinity, if this is compensated by the oscillation of the ``small" eigenspaces. That some property of this kind should be required is revealed by the obvious example $V(x)=\begin{bmatrix}
0 & 0\\
0 & \mu(x)
\end{bmatrix}$: no matter how large $\mu:\R^d\rightarrow[0,+\infty)$ is, the operator $\mathcal{H}_V$ does not have discrete spectrum, because it is the direct sum of the two scalar operators $-\Delta$ and $-\Delta+\mu$, the first of which does not have discrete spectrum.\newline

By Proposition \ref{schrodpp-prop}, the proof of Theorem \ref{osc2-thm} reduces to showing that for every $\eps>0$ there exists $R<+\infty$ such that\bel\label{osc2-tail}
\int_{|x|\geq R}|\psi|^2\leq \eps \left(\int_{\R^d}|\nabla\psi|^2+\int_{\R^d}(V\psi,\psi)\right)\qquad\forall \psi\in C^\infty_c(\R^d;\C^m).
\eel 
The heuristics behind our argument is as follows: \eqref{osc2-tail} fails only if there are vector-valued wave functions $\psi$ localized far from the origin and with small energy $\mathcal{E}_V(\psi)$. But if the potential energy $\int_{\R^d}(V\psi,\psi)$ is small, then in some average sense the vector $\psi(x)$ has to be close to $\mathcal{S}(x)$, because of condition (b). By condition (c), if this happens then $\psi$ has to oscillate a lot, and thus the kinetic energy $\int_{\R^d}|\nabla\psi|^2$ has to be large. Therefore in any case the energy of $\psi$ is large and \eqref{osc2-tail} cannot fail.

\section{Proof of Theorem \ref{osc2-thm}}\label{osc3-sec}

Fix $\psi\in C^\infty_c(\R^d;\C^m)$ and $\eps>0$. Our goal is to find $R<+\infty$ depending only on $\eps$ such that \eqref{osc2-tail} holds.\newline

First of all we decompose $\psi$ in its radial and angular part, writing 
\be\psi=\phi u,\ee 
where $\phi=|\psi|$ is scalar and $u$ has unit norm. Notice that $u=\frac{\psi}{|\psi|}$ on $\Omega:=\{\psi\neq0\}$ and  may be arbitrarily defined on $\Omega^c$: this ambiguity will not affect the rest of the proof. Both $\phi$ and $u$ are smooth on $\Omega$. Differentiating the identity $(u,u)\equiv1$ we find that $u$ and $\frac{\partial u}{\partial x_j}$ are orthogonal on $\Omega$, where we have\be
\left|\frac{\partial\psi}{\partial x_j}\right|^2=\left|\frac{\partial\phi}{\partial x_j} u+\phi\frac{\partial u}{\partial x_j}\right|^2=\left|\frac{\partial\phi}{\partial x_j}u\right|^2+\left|\phi\frac{\partial u}{\partial x_j}\right|^2=\left|\frac{\partial\phi}{\partial x_j}\right|^2+\phi^2\left|\frac{\partial u}{\partial x_j}\right|^2.
\ee 
Since $\nabla\psi(x)=0$ at almost every $x$ such that $\psi(x)=0$ and the same holds for $\phi$,  we can write \bel\label{osc3-Veff}
\int_{\R^d}|\nabla\psi|^2+\int_{\R^d}(V\psi,\psi)=\int_{\R^d}|\nabla\phi|^2+\int_{\R^d}\widetilde{V}\phi^2,
\eel
where \be
\widetilde{V}:=(Vu,u)+|\nabla u|^2 \quad\text{on} \quad\Omega\qquad \text{and}\quad \widetilde{V}:=0 \quad\text{on}\quad \Omega^c.
\ee 
Identity \eqref{osc3-Veff} allows to interpret the energy of a vector-valued function with respect to a matrix potential as the energy of its radial part with respect to an ``effective potential" in which the angular part is incorporated. The potential $\widetilde{V}$ is not locally integrable in general, but $\int\widetilde{V}\phi^2$ is finite, as a consequence of the computations above.

By assumption (c), we can choose $\ell>0$ such that $\ell^{-1}\omega_\infty(\ell,\mathcal{S})\geq \eps^{-1}$, and by assumption (b), we can then choose $R$ so that for every cube $Q$ with side length $\ell$ intersecting $\{|y|\geq R\}$, we have  \bel\label{osc3-large}
(V(x)v,v)\geq \ell^{-2}|v|^2\qquad\forall v\in \mathcal{L}(x),\ \forall x\in Q.
\eel 
Enlarging $R$ if necessary, we can assume that for every such cube $Q$ we have\bel\label{osc3-omega}
\omega(Q,\mathcal{S})\geq \frac{\omega_\infty(\ell,\mathcal{S})}{2}.
\eel We are going to prove that:
\bel\label{osc3-est}
\ell^{-2}\omega_\infty(\ell,\mathcal{S})^2\int_Q\phi^2\leq C_d\left( \int_Q|\nabla\phi|^2+\int_Q\widetilde{V}\phi^2\right),
\eel
for every cube $Q$ with side length $\ell$ intersecting $\{|y|\geq R\}$. Recalling identity \eqref{osc3-Veff} and our choice of $\ell$, we see that this implies \eqref{osc2-tail}, as we wanted. We can turn to the proof of \eqref{osc3-est}.\newline 

To analyze separately the contributions of the two terms on the right-hand side of \eqref{osc3-est}, we consider the compact set\be
F:=\left\{x\in Q: \phi(x)^2\leq\frac{1}{4|Q|}\int_Q\phi^2\right\}.
\ee Let $c_d$ be a small constant depending only on $d$ and to be fixed later: we split our analysis depending on whether the capacity of $F$ is smaller or larger than $c_d\cdot\ell^{d-2}\omega_\infty(\ell,\mathcal{S})^2$.\newline

\par If $\text{Cap}(F)\geq c_d\cdot\ell^{d-2}\omega_\infty(\ell,\mathcal{S})^2$, we can apply Lemma $2.2$ of \cite{ko-ma-sh} that in particular states that $\int_Q\phi^2\leq C_d\frac{|Q|}{\text{Cap}(F)}\int_Q|\nabla\phi|^2$. We obtain
\be
c_d\cdot\ell^{-2}\omega_\infty(\ell,\mathcal{S})^2\int_Q\phi^2\leq\frac{\text{Cap}(F)}{|Q|}\int_Q\phi^2 \leq C_d\int_Q|\nabla\phi|^2, 
\ee which implies \eqref{osc3-est}. 

\par If $\text{Cap}(F)\leq c_d\cdot\ell^{d-2}\omega_\infty(\ell,\mathcal{S})^2$, we define \be
\mathcal{S}_{\text{unit}}(x):=\{v\in\C^m: v\in \mathcal{S}(x)\text{ and }|v|=1\}
\ee
and we split the analysis into two further sub-cases, depending this time on whether the quantity \bel\label{osc3-intdist}
\frac{1}{|Q|}\int_{Q\setminus F}\text{dist}(u(x),\mathcal{S}_{\text{unit}}(x))^2dx
\eel is smaller or larger than $c_d\cdot\omega_\infty(\ell,\mathcal{S})^2$. The integral \eqref{osc3-intdist} measures in $L^2$ average and on $Q\setminus F$ how far $u$ is from the distribution of subspaces $\mathcal{S}$. Notice that if $d=1$, taking $c_d\leq1/2$, we have $F=\varnothing$ and part of what follows becomes more elementary.

If \eqref{osc3-intdist} $\geq c_d\cdot\omega_\infty(\ell,\mathcal{S})^2$, we write \be
u(x)=u'(x)+u''(x),\quad u'(x)\in\mathcal{S}(x),\ u''(x)\in\mathcal{L}(x),\qquad\forall x\in Q.
\ee
By the invariance under $V$ of $\mathcal{S}$ and $\mathcal{L}$ and the non-negativity of $V$, we have \be
(Vu,u)=(Vu',u')+(Vu'',u'')\geq (Vu'',u'')\geq \ell^{-2}|u''|^2\qquad\text{ on }Q.
\ee The last inequality follows from \eqref{osc3-large}. The elementary inequality \be
|u''(x)|^2\geq\frac{1}{2}\text{dist}(u(x),\mathcal{S}_{\text{unit}}(x))^2
\ee allows to estimate\bee
&&\int_Q\widetilde{V}\phi^2\geq \int_{Q\setminus F}(Vu,u)\phi^2 \geq \frac{1}{4|Q|}\int_{Q\setminus F}(Vu,u)\cdot\int_Q\phi^2 \\
&\geq& \frac{\ell^{-2}}{8}\frac{1}{|Q|}\int_{Q\setminus F}\text{dist}(u(x),\mathcal{S}_{\text{unit}}(x))^2dx\cdot\int_Q\phi^2\geq \frac{c_d}{8}\ell^{-2}\omega_\infty(\ell,\mathcal{S})^2\int_Q\phi^2.
\eee In the first line we used the fact that $F$ contains the zero set of $\psi$, and hence $\widetilde{V}\geq (Vu,u)$ on $Q\setminus F$. Inequality \eqref{osc3-est} is proved if \eqref{osc3-intdist} $\geq c_d\cdot\omega_\infty(\ell,\mathcal{S})^2$.

The remaining case is: \eqref{osc3-intdist} $\leq c_d\cdot\omega_\infty(\ell,\mathcal{S})^2$. We are going to prove that if $c_d$ is small enough, then \bel\label{osc3-grad}
\frac{1}{|Q|}\int_{Q\setminus F}|\nabla u|^2\geq\frac{\ell^{-2}\omega_\infty(\ell,\mathcal{S})^2}{C_d}.
\eel
Since $\int_Q\widetilde{V}\phi^2\geq \frac{1}{4|Q|}\int_{Q\setminus F}|\nabla u|^2\int_Q\phi^2$, this concludes the proof of Theorem \ref{osc2-thm}.

By the measurability of $\mathcal{S}$ and the hypothesis on \eqref{osc3-intdist}, we can find $v\in\Gamma(Q,\mathcal{S})$ such that \bel\label{osc3-v} 
\frac{1}{|Q|}\int_{Q\setminus F}|u-v|^2dx\leq c_d\cdot\omega_\infty(\ell,\mathcal{S})^2.
\eel Moreover, since we are under the assumption that $F$ has small capacity, there exists $\eta\in C^\infty_c(\R^d,[0,1])$ such that $\eta\equiv1$ on $F$ (and $\eta\equiv0$ on $(2Q)^c$ if $d=2$) and \bel\label{osc3-grad-eta}
\int_{\R^d}|\nabla\eta|^2\leq 2c_d\cdot\ell^{d-2}\omega_\infty(\ell,\mathcal{S})^2.
\eel By H\"older inequality and Sobolev embedding, we have\bel\label{osc3-L2}
\frac{1}{|Q|}\int_Q\eta^2\leq C_dc_d\cdot\omega_\infty(\ell,\mathcal{S})^2.
\eel 
Recalling \eqref{osc3-omega}, we have\bee
\frac{\omega_\infty(\ell,\mathcal{S})}{2}&\leq& \omega(Q,\mathcal{S})\leq \sqrt{\frac{1}{|Q|}\int_Q\left|v-\frac{1}{|Q|}\int_Q v\right|^2}\\
&\leq &\sqrt{\frac{1}{|Q|}\int_Q\left|(1-\eta)v-\frac{1}{|Q|}\int_Q (1-\eta)v\right|^2}+\sqrt{\frac{1}{|Q|}\int_Q\left|\eta v-\frac{1}{|Q|}\int \eta v\right|^2}\\
&\leq &\sqrt{\frac{1}{|Q|}\int_Q\left|(1-\eta)v-\frac{1}{|Q|}\int_Q (1-\eta)v\right|^2}+\sqrt{\frac{1}{|Q|}\int_Q\eta^2}.
\eee In the last line we used the fact that $w\mapsto w-\frac{1}{|Q|}\int_Qw$ is a projection operator.
Applying \eqref{osc3-L2} and choosing $c_d$ small enough, we can absorb the last term on the left and obtain\be
\frac{\omega_\infty(\ell,\mathcal{S})}{4}\leq\sqrt{\frac{1}{|Q|}\int_Q\left|(1-\eta)v-\frac{1}{|Q|}\int_Q (1-\eta)v\right|^2}.\ee We can then proceed to estimate\bee
&&\sqrt{\frac{1}{|Q|}\int_Q\left|(1-\eta)v-\frac{1}{|Q|}\int_Q (1-\eta)v\right|^2} \\
&\leq& \sqrt{\frac{1}{|Q|}\int_Q\left|(1-\eta)(v-u)-\frac{1}{|Q|}\int_Q (1-\eta)(v-u)\right|^2}\\
&+&\sqrt{\frac{1}{|Q|}\int_Q\left|(1-\eta)u-\frac{1}{|Q|}\int_Q (1-\eta)u\right|^2}\\ 
&\leq& \sqrt{\frac{1}{|Q|}\int_{Q\setminus F}|v-u|^2}+C_d\ell\sqrt{\frac{1}{|Q|}\int_Q|\nabla[(1-\eta)u]|^2}.
\eee In the last line we used the fact that $\eta=1$ on $F$ and Poincar\'e inequality. To justify its application, one can notice that $u$ is in the Sobolev space $H^1(\mathring{Q}\setminus F,\C^m)$, because $u$ is smooth on $\Omega=\{\psi\neq0\}$ and $\overline{Q\setminus F}\cap\Omega^c=\varnothing$, hence $(1-\eta)u\in H^1(\mathring{Q},\C^m)$.

Using \eqref{osc3-v} and Leibnitz rule, \bee
\frac{\omega_\infty(\ell,\mathcal{S})}{4} &\leq& \sqrt{c_d}\cdot\omega_\infty(\ell,\mathcal{S})+C_d\ell\sqrt{\frac{1}{|Q|}\int_Q|\nabla\eta|^2}+C_d\ell\sqrt{\frac{1}{|Q|}\int_{Q\setminus F}|\nabla u|^2}\\ 
&\leq& C_d\sqrt{c_d}\cdot\omega_\infty(\ell,\mathcal{S})+C_d\ell\sqrt{\frac{1}{|Q|}\int_{Q\setminus F}|\nabla u|^2},
\eee where in the last inequality we used \eqref{osc3-grad-eta}. If $c_d$ is appropriately small, we finally obtain \eqref{osc3-grad}. The proof is concluded.

\section{Examples of potentials satisfying\\ the conditions of Theorem \ref{osc2-thm}}\label{ex-sec}

Since the statement of Theorem \ref{osc2-thm} is a bit complicated, we devote this section to showing how to build examples of potentials satisfying its assumptions.\newline

We say that a collection $\mathcal{Q}$ of cubes with sides parallel to the axes is a \emph{good grid} if it is a partition of $\R^d$ up to sets of measure $0$, and there exists a sequence of positive numbers $\{\ell_k\}_{k\in\N}$ such that:\begin{enumerate}
\item $\lim_{k\rightarrow+\infty}\ell_k=0$, 
\item for every $k\in\N$ there exists $R(k)<+\infty$ such that for every $x\in\Z^d\cap\{|y|\geq R(k)\}$ the cube $Q(\ell_k x,\ell_k)$ may be written as a finite union of cubes of $\mathcal{Q}$.
\end{enumerate}
The expression \emph{partition up to sets of measure zero} means that $\cup_{Q\in\mathcal{Q}}Q=\R^d$ and $Q\cap Q'$ has measure $0$ for every $Q,Q'\in\mathcal{Q}$.

If $\mathcal{T}:Q(0,1)\rightarrow \mathcal{V}(m)$ is measurable and $Q=Q(x,\ell)\subseteq \R^d$, we put\be
\mathcal{T}_Q(y):=\mathcal{S}\left(\frac{y-x}{\ell}\right)\qquad\forall y\in Q.
\ee

\begin{dfn}\label{ex-dfn}
Assume that we have a good grid $\mathcal{Q}$, a measurable mapping \be\mathcal{T}:Q(0,1)\rightarrow \mathcal{V}(m),\ee and a locally integrable function \be
\nu:\R^d\rightarrow [0,+\infty).
\ee 
We build a potential $V:\R^d\rightarrow H_m$ defining it separately on every cube $Q\in\mathcal{Q}$. If $x\in Q\in\mathcal{Q}$ and $v\in \mathcal{T}_Q(x)$, we define \be
V(x)v:=0,
\ee while if $v\in \mathcal{T}_Q(x)^\perp$, we define \be
V(x)v:=\nu(x)v.
\ee
Observe that the resulting matrix $V(x)$ is Hermitian non-negative, and that $V$ is locally integrable. To be precise, $V$ is well-defined only almost-everywhere, but what follows is unaffected by this ambiguity.
\end{dfn}

\begin{prop}\label{ex-prop} If $\omega(Q(0,1),\mathcal{T})>0$ and \bel\label{ex-lim}
\lim_{x\rightarrow\infty}\nu(x)=+\infty,
\eel
then the potential $V$ given by Definition \ref{ex-dfn} satisfies the hypotheses of Theorem \ref{osc2-thm}.
\end{prop}

\begin{proof}
We define $\mathcal{S}:\R^d\rightarrow \mathcal{V}(m)$ putting $\mathcal{S}:=\mathcal{T}_Q$ on $Q$ (with the same harmless ambiguity as in Definition \ref{ex-dfn}), and $\mathcal{L}:\R^d\rightarrow \mathcal{V}(m)$ putting $\mathcal{L}(x):=\mathcal{S}(x)^\perp$ for every $x\in\R^d$. Condition (a) of Theorem \ref{osc2-thm} is trivially satisfied. Condition (b) immediately follows from the observation that \be\lambda(V(x)_{|\mathcal{L}(x)})=\nu(x),\ee and the limit \eqref{ex-lim}.

Now recall the definition of a good grid given above and pick $k\in\N$. By property (2) of the good grid $\mathcal{Q}$ and Definition \ref{ex-dfn}, the restrictions of $\mathcal{S}$ on every cube $Q(x\ell_k,\ell_k)$ with $x\in\Z^d\cap\{|y|\geq R(k)\}$, are obtained from $\mathcal{T}$ by rearrangement and rescaling. By part (ii) of Proposition \ref{osc-prop}, we have\be
\omega_\infty(\ell_k,\mathcal{S})=\liminf_{x\rightarrow \infty}\omega(Q(x\ell_k,\ell_k),\mathcal{S})=\omega(Q(0,1),\mathcal{T}).
\ee Therefore \bee
\limsup_{\ell\rightarrow0+}\frac{\omega_\infty(\ell,\mathcal{S})}{\ell}&\geq& \lim_{k\rightarrow+\infty}\frac{\omega_\infty(\ell_k,\mathcal{S})}{\ell_k}\\
&=&\lim_{k\rightarrow+\infty}\frac{\omega(Q(0,1),\mathcal{T})}{\ell_k}=+\infty,
\eee where we used the hypothesis $\omega(Q(0,1),\mathcal{T})>0$.
\end{proof}

By part (i) of Proposition \ref{osc-prop}, it is easy to produce mappings $\mathcal{T}:Q(0,1)\rightarrow \mathcal{V}(m)$ such that $\omega(Q(0,1),\mathcal{T})>0$. An interesting example is obtained splitting $Q(0,1)$ in $m$ measurable pieces of positive measure and defining $\mathcal{T}$ as the $j$-th coordinate hyperplane on the $j$-th piece. Since the intersection of the $m$ coordinate hyperplanes is trivial, $\omega(Q(0,1),\mathcal{T})>0$ and we get the following strong counterexample to the converse of Proposition \ref{schrodpp-lambda}.

\begin{prop}\label{ex-counter}
For every $d\geq1$ and $m\geq2$, there are locally integrable non-negative potentials $V:\R^d\rightarrow H_m$ such that: \begin{enumerate}
\item[(i)] $V$ has everywhere rank $1$,
\item[(ii)] the matrix Schr\"odinger operator $\mathcal{H}_V$ has discrete spectrum.
\end{enumerate}
\end{prop}

\section{A converse to Proposition \ref{schrodpp-lambda}}\label{pol-sec}

In the previous sections we have seen that the converse of Proposition \ref{schrodpp-lambda} fails rather dramatically: there are matrix-valued potentials $V$ such that $\lambda(V)\equiv0$ and hence $\mathcal{H}_{\lambda(V)}$ has not discrete spectrum, but $\mathcal{H}_V$ has discrete spectrum.

In this section we prove a complementary result that shows that under a rigidity assumption on $V$, the discreteness of the spectrum of $\mathcal{H}_V$ is enough to deduce the discreteness of the spectrum of $\mathcal{H}_{\lambda(V)}$.

The precise statement is the following.

\begin{thm}\label{pol-thm}
Assume that $V:\R^d\rightarrow H_2$ is such that: \begin{enumerate}
\item[\emph{(a)}] $V(x)\geq 0$ for every $x\in \R^d$,
\item[\emph{(b)}] $V(x)$ is a real symmetric $2\times 2$ matrix for every $x\in \R^d$,  
\item[\emph{(c)}] $(Vu,u)$ is a polynomial for every $u\in\C^2$.
\end{enumerate}
If the operator $\mathcal{H}_V$ has discrete spectrum, then $\mathcal{H}_{\lambda(V)}$ has discrete spectrum too.
\end{thm}

Hypothesis (c) is the rigidity assumption mentioned before. The assumption $m=2$ and hypothesis (b) are most probably limitations of our proof technique, and we expect the converse of Proposition \ref{schrodpp-lambda} to hold for more general ``rigid" potentials. \newline

By Proposition \ref{schrodpp-prop}, the hypothesis of Theorem \ref{pol-thm} is that for every $\eps>0$ we can find an $R(\eps)<+\infty$ such that
\bel\label{pol-hyp}
\int_{|y|\geq R(\eps)}|\psi|^2\leq \eps\cdot\mathcal{E}_V(\psi)\qquad\forall \psi\in \mathcal{D}(\mathcal{E}_V).
\eel 
 
Thanks to Theorem \ref{mash-thm}, the proof of Theorem \ref{pol-thm} follows easily from the next proposition.
 
\begin{prop}\label{pol-key}
There exist $\gamma, c_1,c_2>0$ depending only on $V$, such that if $\ell>0$, $\eps\leq c_1\ell^2$, $Q(x,\ell)\subseteq \{|y|\geq R(\eps)\}$ and $F\in \mathcal{N}_\gamma(Q(x,\ell))$, then \be
\frac{1}{\ell^d}\int_{Q(x,\ell)\setminus F}\lambda(V(y))dy\geq c_2\eps^{-1}.
\ee
\end{prop}
We will prove Proposition \ref{pol-key} testing \eqref{pol-hyp} on appropriate test functions. To define them, we exploit a dichotomy to which the next section is dedicated.

\section{Good and bad cubes}\label{gb-sec}

Fix $V$ as in Theorem \ref{pol-thm}. We denote by $\lambda(x)$ the minimal eigenvalue function that we previously indicated by $\lambda(V(x))$, and we introduce\be
\mu(x):=\max_{u\in\C^2:\ |u|=1}(V(x)u,u),
\ee i.e., the maximal eigenvalue function. Since $m=2$, $\lambda(x)$ and $\mu(x)$ are the only eigenvalues of $V(x)$. We also put:\be
\text{tr}(x):=\text{tr}(V(x)), \quad \det(x):=\det(V(x)).
\ee By assumption (c) of Theorem \ref{pol-thm}, both $\text{tr}(x)$ and $\det(x)$ are polynomials on $\R^d$, while $\lambda(x)$ and $\mu(x)$ are not polynomials in general.

If $Q$ is a cube, we denote by $\ell_Q$ its side length.

\begin{lem}\label{gb-lem} There is a constant $c>0$ depending only on $d$ and the degree of $\emph{tr}(x)$ and $\det(x)$ such that for every cube $Q$ there is a smaller cube $Q'\subseteq Q$ such that $\ell_{Q'}= c \ell_Q$, $\lambda(x)<\mu(x)$ for every $x\in Q'$ and:\begin{enumerate}
\item[(i)] either $\frac{1}{8}\mu(x)\leq\lambda(x)$ for every $x\in Q'$,
\item[(ii)] or $2\lambda(x)\leq \mu(x)$ for every $x\in Q'$ and $\sup_{x\in Q'}\mu(x)\leq 4\inf_{x\in Q'}\mu(x)$.
\end{enumerate}
\end{lem} 

If the first condition above holds we say that $Q$ is a \emph{good cube}, otherwise we say that $Q$ is a \emph{bad cube}. The second condition imposed on bad cubes is somehow technical and allows to assume $\mu$ approximately constant on $Q'$.
\par To prove Lemma \ref{gb-lem}, we need two elementary results. The first expresses the fact that the zero-set of a polynomial of bounded degree cannot be too dense in a given cube.

\begin{lem}\label{gb-zero}
For every $d,D\in\N$, there exists $c>0$ such that the following holds: if $p$ is a non-zero polynomial in $d$ variables of degree $\leq D$ and $Q\subseteq \R^d$ is a cube, there is a smaller cube $Q'\subseteq Q$ such that $\ell_{Q'}= c \ell_Q$ that does not intersect the zero set $Z(p)$ of $p$.
\end{lem}

\begin{proof}
We argue by contradiction. Appropriately translating, rotating and rescaling one may find a sequence $\{p_k\}_{k\in\N}$ of non-zero polynomials in $d$ variables of degree $\leq D$ such that $Z(p_k)$ intersects any cube of side $2^{-k}$ contained in the unit cube $[0,1]^d$. Multiplying by a constant the $p_k$'s, we can also assume that $||p_k||_{L^\infty([0,1]^d)}=1$. 

Since a space of polynomials of bounded degree is finite-dimensional, there is a subsequence $\{p_{k_\ell}\}_{\ell\in\N}$ converging uniformly on $[0,1]^d$ to a non-zero polynomial $q$. If $x\in [0,1]^d$, there is a sequence $x_\ell\in [0,1]^d$ such that $x_\ell\rightarrow x$ and $p_{k_\ell}(x_\ell)=0$. Then \be
|q(x)|\leq |q(x)-q(x_\ell)|+|q(x_\ell)-p_{k_\ell}(x_\ell)|\leq |q(x)-q(x_\ell)|+||q-p_{k_\ell}||_{L^\infty([0,1]^d)}.
\ee Letting $\ell$ tend to $+\infty$, we come to the conclusion that $q$ is identically $0$, a contradiction.
\end{proof}

The second result we need in the proof of Lemma \ref{gb-lem} says that a polynomial is approximately constant on a significant fraction of every cube.

\begin{lem}\label{gb-approx}
Given $d,D\in\N$, there exists a constant $c'>0$ such that for every polynomial $q$ in $d$ variables of degree $\leq D$ and cube $Q\subseteq \R^d$, there is a smaller cube $Q'\subseteq Q$ of side $\ell_{Q'}=c\ell_Q$ such that \be
\sup_{x\in Q'}|q(x)|\leq 2\inf_{x\in Q'}|q(x)|.
\ee
\end{lem}

\begin{proof}
We may use a compactness and contradiction argument as in the proof of Lemma \ref{gb-zero}. This time, after translation, rotation and rescaling, we have a sequence $\{p_k\}_{k\in\N}$ of polynomials in $d$ variables of degree $\leq D$ such that for any $k\in\N$ and any cube $Q'$ of side $2^{-k}$ contained in $[0,1]^d$, we have\be
\sup_{x\in Q'}|p_k(x)|> 2\inf_{x\in Q'}|p_k(x)|.
\ee
Since the inequality above is homogeneous, we can multiply by a  positive constant and assume that $||p_k||_{L^\infty([0,1]^d)}=1$. Now we extract a subsequence $\{p_{k_\ell}\}_{\ell\in\N}$ converging uniformly to a non-zero polynomial $q$ on $[0,1]^d$.

Given $x\in [0,1]^d$, by construction there are two sequences $\{x_\ell\}_{\ell\in\N}$ and $\{y_\ell\}_{\ell\in\N}$ converging to $x$ such that $|p_{k_\ell}(x_\ell)|>2|p_{k_\ell}(y_\ell)|$ for every $\ell\in\N$. This is clearly in contradiction with the fact that \be\lim_{\ell\rightarrow+\infty}p_{k_{\ell}}(x_\ell)=\lim_{\ell\rightarrow+\infty}p_{k_{\ell}}(y_\ell)=q(x).\ee
\end{proof}

\begin{proof}[Proof of Lemma \ref{gb-lem}]
Consider the polynomial \be
p(x):=\left(\frac{\text{tr}(x)^2}{8}-\det(x)\right)\cdot\left(\frac{\text{tr}(x)^2}{4}-\det(x)\right).
\ee Given a cube $Q$, Lemma \ref{gb-zero} gives $Q'$ such that $\ell_{Q'}= c\ell_Q$ and $p(x)\neq0$ on $Q'$. Notice that $c$ depends only on $d$ and the degrees of $\text{tr}(x)$ and $\det(x)$. Since \be
\det(x)=\lambda(x)\mu(x)\leq\frac{(\lambda(x)+\mu(x))^2}{4}= \frac{\text{tr}(x)^2}{4},\ee $\frac{1}{4}\text{tr}(x)^2-\det(x)>0$ on $Q'$, which implies $\lambda(x)<\mu(x)$ for every $x\in Q'$. For the first factor of $p(x)$ there are two options: either $\frac{\text{tr}(x)^2}{8}<\det(x)$ for every $x\in Q'$, or $\det(x)<\frac{\text{tr}(x)^2}{8}$ for every $x\in Q'$. In the first case, \be
\frac{\mu(x)}{8}\leq\frac{\text{tr}(x)}{8}<\frac{\det(x)}{\text{tr}(x)}\leq \lambda(x)\qquad\forall x\in Q'.
\ee In the second case\be
2\lambda(x)\leq 4\frac{\det(x)}{\text{tr}(x)}<\frac{\text{tr}(x)}{2}\leq \mu(x)\qquad\forall x\in Q'.
\ee In the latter case we can pass to an even smaller cube where the supremum of $\mu$ is controlled by 4 times the infimum. This can be done observing that $\frac{\text{tr}(x)}{2}\leq \mu(x)\leq\text{tr}(x)$ and using Lemma \ref{gb-approx}.
\end{proof}

\section{Proof of Proposition \ref{pol-key}}\label{key-sec}

Let $x\in\R^d$ and $\ell>0$ and consider the cube $Q(x,\ell)$. Lemma \ref{gb-lem} gives us a smaller cube $Q(y,s)$ such that $s= c\ell$, with $c$ depending only on $d$ and the degrees of $\text{tr}(x)$ and $\det(x)$. Modifying by a factor the value of $c$, we can also assume that $Q(y,2s)\subseteq Q(x,\ell)$. For the moment we do not distinguish between good and bad cubes, and we follow the argument in Section 3 of \cite{ma-sh}. 

Fix $F\in\mathcal{N}_\gamma(Q(x,\ell))$, where $\gamma>0$ is to be chosen later. Let $F'$ be a compact set such that: \begin{enumerate}
\item[(a)] $F'$ is the closure of an open set with smooth boundary,
\item[(b)] $F\cap Q(y,s)\subseteq F'\subseteq Q\left(y,\frac{3}{2}s\right)$,
\item[(c)] $\text{Cap}(F')\leq 2\text{Cap}(F\cap Q(y,s))\leq 2\text{Cap}(F)$. 
\end{enumerate}
In fact, the outer regularity of capacity gives an open set $\Omega\supseteq F\cap Q(y,s)$ such that $\overline{\Omega}$ satisfies conditions (b) and (c), and any $F'$ that is the closure of a smooth open set and lies between $F\cap Q(y,s)$ and $\Omega$ will do. 

If $d=2$, one has to interpret $\text{Cap}$ as the capacity relative to $Q(x,2\ell)$, as defined by identity \eqref{mash-cap2}, while if $d=1$, the discussion becomes trivial and $F'=\varnothing$. 

If $d\geq 2$, there exists $A<+\infty$ depending only on $c$ (and hence on $V$) such that\bel\label{key-cap}
\text{Cap}(F')\leq 2\text{Cap}(F)\leq 2\gamma \text{Cap}(Q(x,\ell))\leq2\gamma A\text{Cap}(Q(y,s)),
\eel by the monotonicity of capacity and its behavior under scaling. We can now fix $\gamma$ in such a way that 
\bel \label{key-gamma} 2\gamma A<1.
\eel 
If $d\geq 2$, we denote by $P$ the equilibrium potential of $F'$, i.e., the function satisfying \begin{enumerate}
\item[(a)] $P\in C(\R^d,[0,1])$,
\item[(b)] $P\equiv1$ on $F'$,
\item[(c)] $\lim_{y\rightarrow\infty}P(y)=0$ (if $d\geq3$) or $P(y)\equiv0$ for every $y\notin Q(x,2\ell)$ (if $d=2$),
\item[(d)] $\Delta P=0$ on $\R^d\setminus F'$ (if $n\geq3$) or on $Q(x,2\ell)\setminus F'$ (if $d=2$),
\item[(e)] $\int_{\R^d}|\nabla P|^2=\text{Cap}(F')$.
\end{enumerate} The existence of $P$ is a classical fact of potential theory (see \cite{wermer}). If $d=1$, put $P\equiv0$ in what follows. We now state as a lemma a useful property of $P$.

\begin{lem}\label{key-eqpot}
$\int_{Q(y,s)}(1-P)^2\geq c_d(c\ell)^d.$
\end{lem}
\begin{proof} Lemma $3.1$ of \cite{ko-ma-sh}, or inequality $3.11$ of \cite{ma-sh}, asserts that if $Q$ is a cube, $E\subseteq \frac{3}{2}Q$ is the closure of a smooth open set such that $\text{Cap}(E)<\text{Cap}(Q)$ and $P_E$ is the equilibrium potential of $E$, then
\bel\label{key-eqpot-ineq}
c_d\left(1-\frac{\text{Cap}(E)}{\text{Cap}(Q)}\right)^2\leq \sigma\frac{\text{Cap}(E)}{\text{Cap}(Q)}+\frac{\sigma^{-1}}{|Q|}\int_Q(1-P_E)^2\qquad\forall \sigma>0.
\eel If $d\geq 3$, the thesis follows applying this to $E=F'$ and $Q=Q(y,s)$, recalling \eqref{key-gamma} and choosing $\sigma$ small enough (depending on $d$).

The case $d=1$ is trivial, while the case $d=2$ can be derived using the argument above with a minor observation: \eqref{key-eqpot-ineq} holds for $d=2$ if $\text{Cap}(E)$ denotes the capacity relative to $2Q$, while here we are considering $F'=E$ and we want $\text{Cap}(F')$ to denote capacity relative to $Q(x,2\ell)$. This is not a problem, because the two quantities are easily seen to be comparable up to a constant depending on $c$ and hence, taking $\gamma$ small and recalling \eqref{key-cap} and \eqref{key-gamma}, we conclude.
\end{proof}

Let $\eta_\delta\in C^\infty_c(\R^d,[0,1])$ be equal to $1$ on $Q(y,(1-\delta) s)$, equal to $0$ outside of $Q(y,s)$, and such that $|\nabla \eta_\delta|\leq C_d(\delta s)^{-1}=C_d(\delta c\ell)^{-1} $.\newline

We now analyze separately good and bad cubes.\newline

If $Q(x,\ell)$ is good, we define $\psi:=\eta_\delta(1-P)e_1$, where $e_1=(1,0)$ (but any other unit norm vector of $\C^2$ would do). Notice that $\psi\in \mathcal{D}({\mathcal{E}_V})$. An integration by parts and the harmonicity of $P$ outside $F'$ yield\bee
\int_{\R^d}|\nabla\psi|^2&=&\int_{\R^d}|\nabla\eta_\delta|^2(1-P)^2+\int_{\R^d}\eta_\delta^2|\nabla P|^2-\int_{\R^n}\nabla (\eta_\delta^2)\cdot (1-P)\nabla P\\
&=&\int_{\R^d}|\nabla\eta_\delta|^2(1-P)^2\leq C_d(\delta c\ell)^{-2}\int_{Q(y,s)}(1-P)^2\leq C_d \delta^{-2} (c\ell)^{d-2},
\eee where the first inequality follows from the properties of $\eta_\delta$, while the second follows from the fact that $0\leq P\leq1$.

If $Q(x,\ell)\subseteq\{|y|\geq R(\eps)\}$ (recall how $R(\eps)$ was defined in Section \ref{pol-sec}), we test \eqref{pol-hyp} on $\psi$, using the bound above. The result is:
\bee
\int_{\R^d} \eta_\delta^2(1-P)^2 &\leq& \eps\left(C_d\delta^{-2} (c\ell)^{d-2}+\int_{\R^d}\eta_\delta^2(1-P)^2(Ve_1,e_1)\right)\\
&\leq& \eps\left(C_d\delta^{-2} (c\ell)^{d-2}+8\int_{Q(y,s)\setminus F'}\lambda\right),
\eee where in the second line we used the fact that $P\equiv1$ on $F'$ and the estimate $(Ve_1,e_1)\leq \mu\leq 8\lambda$, which holds on $Q(y,s)$, because $Q(x,\ell)$ is good. Since $F'\supseteq F\cap Q(y,s)$, we have\bel\label{key-good-1}
\int_{\R^d} \eta_\delta^2(1-P)^2 \leq\eps\left(C_d\delta^{-2} (c\ell)^{d-2}+8\int_{Q(x,\ell)\setminus F}\lambda\right)
\eel
Observing that $|Q(y,s)\setminus Q(y,(1-\delta)s)|\leq C_d\delta s^d=C_d\delta (c\ell)^d$ and $\eta_\delta\equiv1$ on $Q(y,(1-\delta)s)$ we can bound\be
\int_{Q(y,s)}(1-P)^2\leq\int_{\R^d}\eta_\delta^2(1-P)^2+C_d\delta (c\ell)^d.
\ee
By Lemma \ref{key-eqpot} we can choose $\delta=\delta_d$ small enough and obtain \bel\label{good-2}c_d(c\ell)^d\leq \int_{\R^d}\eta_{\delta_d}^2(1-P)^2.\eel 
Putting \eqref{key-good-1} (for $\delta=\delta_d$) and \eqref{good-2} together, we find\be
c_d(c\ell)^d\leq \eps\left(C_d(c\ell)^{d-2}+8\int_{Q(x,\ell)\setminus F}\lambda\right).
\ee It is now clear that there exist $c_1,c_2>0$ depending only on $V$ such that if $\eps\leq c_1\ell^2$, then $\frac{1}{\ell^d}\int_{Q(x,\ell)\setminus F}\lambda\geq c_2\eps^{-1}$, as we wanted.

Notice that this argument is a modification of Maz'ya-Shubin necessity argument (see Section $3.$ of \cite{ma-sh}). We now move to the bad cubes, which require more work.\newline

If $Q(x,\ell)$ is a bad cube the main difficulty is that $(Ve_1,e_1)$ is not bounded by $\lambda$ on $Q(y,s)$. The natural idea is to test \eqref{pol-hyp} on $\psi:=\eta_\delta(1-P)v$, where $v:Q(y,s)\rightarrow \R^2$ is smooth and satisfies $Vv\equiv\lambda v$ and $|v|=1$ everywhere. It is possible to define such a smooth selection of real eigenvectors because $V$ is assumed to be symmetric (assumption (b) of Theorem \ref{pol-thm}), the cube is simply-connected and $\lambda<\mu$ on $Q(y,s)$, i.e., there are no eigenvalue crossings (by Lemma \ref{gb-lem}). Notice that $\lambda$ is smooth on $Q(y,s)$ for the same reason.

Observe that\bee
\int_{\R^d}|\nabla\psi|^2&\leq&2\int_{\R^d}|\nabla(\eta_\delta(1-P))|^2+2\int_{\R^d}\eta_\delta^2(1-P)^2|\nabla v|^2\\
&\leq& C_d \delta^{-2} (c\ell)^{d-2}+2\int_{Q(y,s)}|\nabla v|^2,
\eee where the last bound is proved exactly as for good cubes. We claim that \bel\label{key-grad}
\int_{Q(y,s)}|\nabla v|^2\leq C\ell^{d-2},
\eel where $C$ is a constant depending on $V$, but independent of all the other parameters.
The reader may easily check that the argument given for good cubes may then be carried out word by word to prove Proposition \ref{pol-key} also for bad cubes. To prove the claim \eqref{key-grad}, we need a lemma.

\begin{lem}\label{key-grad-lem}
There is a constant $C_1$ depending only on $V$ such that for every cube $Q$\be \int_Q||\partial_jV||_{op}^2\leq C_1\ell_Q^{-2}\int_Q||V||_{\text{op}}^2\qquad\forall j=1,\dots,d.\ee
\end{lem}

\begin{proof}
The set $X_{d,D}$ of functions $W:\R^d\rightarrow H_2$ whose matrix elements are polynomials of degree $\leq D$ is a finite-dimensional vector space and $\partial_j$ is a linear operator on $X_{d,D}$. Since $\sqrt{\int_{Q(0,1)}||W||_{op}^2}$ is a norm on $X_{d,D}$, we have the bound\be
\int_{Q(0,1)}||\partial_jW||_{op}^2\leq C_1\int_{Q(0,1)}||W||_{op}^2 \qquad\forall W\in X_{d,D}.
\ee The application of this inequality to $V$, after rescaling and translating, gives the thesis.
\end{proof}

Let us differentiate the identity $Vv=\lambda v$ (all the computations are on $Q(y,s)$):\bel\label{diff-eig}
\frac{\partial V}{\partial x_j}v+V\frac{\partial v}{\partial x_j}=\frac{\partial\lambda}{\partial x_j} v+\lambda\frac{\partial v}{\partial x_j}.
\eel Since $|v|=1$ and $v$ is real-valued, $v\cdot\frac{\partial v}{\partial x_j}=0$. Taking the scalar product of both sides of \eqref{diff-eig} with $\frac{\partial v}{\partial x_j}$, we obtain\be
\left(\frac{\partial V}{\partial x_j}v\right)\cdot\frac{\partial v}{\partial x_j}+\left(V\frac{\partial v}{\partial x_j}\right)\cdot\frac{\partial v}{\partial x_j}=\lambda\left|\frac{\partial v}{\partial x_j}\right|^2.
\ee Since $m=2$ and $\frac{\partial v}{\partial x_j}$ is orthogonal to $v$, it must be an eigenvector of eigenvalue $\mu$ at every point of $Q(y,s)$. Hence $\left(V\frac{\partial v}{\partial x_j}\right)\cdot\frac{\partial v}{\partial x_j}=\mu\left|\frac{\partial v}{\partial x_j}\right|^2$ and\be
(\mu-\lambda)\left|\frac{\partial v}{\partial x_j}\right|^2=-\left(\frac{\partial V}{\partial x_j}v\right)\cdot\frac{\partial v}{\partial x_j}\leq \left|\left|\frac{\partial V}{\partial x_j}\right|\right|_{op}\left|\frac{\partial v}{\partial x_j}\right|.
\ee Since $\mu\geq2\lambda$ and $4\mu\geq ||\mu||_{L^\infty(Q(y,s))}$ on the bad cube $Q(y,s)$, this implies 
\be||\mu||_{L^\infty(Q(y,s))}\left|\frac{\partial v}{\partial x_j}\right|\leq 8\left|\left|\frac{\partial V}{\partial x_j}\right|\right|_{op}.\ee Squaring and integrating on $Q(y,s)$ this pointwise inequality, and applying Lemma \ref{key-grad-lem}, we find\bee
&&||\mu||_{L^\infty(Q(y,s))}^2\int_{Q(y,s)}\left|\frac{\partial v}{\partial x_j}\right|^2\leq64 \int_{Q(y,s)}\left|\left|\frac{\partial V}{\partial x_j}\right|\right|_{op}^2\\ 
&\leq& 64C_1(c\ell)^{-2}\int_{Q(y,s)}\left|\left|V\right|\right|_{op}^2\leq 64C_1(c\ell)^{n-2}||\mu||_{L^\infty(Q(y,s))}^2,
\eee where we used the fact that $||V||_{op}=\mu$. Simplifying and summing over $j$, we get \eqref{key-grad}. This concludes the proof of Theorem \ref{pol-thm}.

\chapter{Weighted Bergman kernels}\label{w-ch}

The main goal of this chapter is to prove new pointwise bounds on weighted Bergman kernels. Sections 2.1 through 2.3 contain preliminary definitions and results on weighted Bergman spaces, the weighted $\dbar$-problem and weighted Kohn Laplacians. Section 2.4 presents the well-known Morrey-Kohn-H\"ormander formula and a natural Caccioppoli-type inequality related to it. Even though there is no standard reference for this material, it mainly consists in the weighted versions of results that can be found in \cite{chen-shaw}. In Section 2.5 we introduce the key notion of $\mu$-coercivity for weighted Kohn Laplacians. Before proving our results, we need to interrupt the flow of the presentation with two intermezzos: Section 2.6 introduces a few useful technical tools from the theory of Schr\"odinger operators (see, e.g., \cite{shen}), while Section 2.7 proves a variant of Fefferman-Phong inequality (see, e.g., \cite{fefferman-uncertainty}). Sections 2.8 through 2.13 contain the original results anticipated in the Introduction.

\section{Weighted Bergman spaces and kernels}\label{berg-sec}

Let us assume that \bel\label{berg-assumption}
\varphi:\C^n\rightarrow\R\quad\text{ is locally bounded and measurable}.
\eel
Later we will focus on $C^2$ plurisubharmonic weights, but regularity and plurisubharmonicity of $\varphi$ are not needed for the basic results of this section. 

We associate to $\varphi$ the weighted $L^2$ space $L^2(\C^n,\varphi)$ consisting of (equivalence classes of) functions $f:\C^n\rightarrow\C$ such that
\be
\int_{\C^n}|f(z)|^2e^{-2\varphi(z)}d\mathcal{L}(z)<+\infty.
\ee
We insert the factor $2$ in the exponential in order to slightly simplify several formulas later on. The Hilbert space norm and scalar product of $L^2(\C^n,\varphi)$ will be denoted by $||\cdot||_\varphi$ and $(\cdot,\cdot)_\varphi$.

The \emph{weighted Bergman space} with respect to the weight $\varphi$ is then defined as follows:\be
A^2(\C^n,\varphi):=\left\{ h:\C^n\rightarrow\C:\ h\text{ is holomorphic and } h\in L^2(\C^n,\varphi)\right\}.
\ee

\begin{prop}
$A^2(\C^n,\varphi)$ is a closed linear subspace of $L^2(\C^n,\varphi)$. \newline 
If $z\in \C^n$ and $r>0$, we have\bel\label{berg-ev-bound}
|h(z)|\leq \frac{1}{|B(z,r)|}\sqrt{\int_{B(z,r)}e^{2\varphi}}\ ||h||_\varphi\qquad\forall h\in A^2(\C^n,\varphi).
\eel
\end{prop}

\begin{proof}
If $h\in A^2(\C^n,\varphi)$, then in particular it is harmonic and satisfies the mean value property \be
h(z)=\frac{1}{|B(z,r)|}\int_{B(z,r)}h.
\ee The Cauchy-Schwarz inequality yields\bee
|h(z)|&\leq& \frac{1}{|B(z,r)|}\sqrt{\int_{B(z,r)}|h|^2e^{-2\varphi}}\sqrt{\int_{B(z,r)}e^{2\varphi}}\\
&\leq& \frac{1}{|B(z,r)|}\sqrt{\int_{B(z,r)}e^{2\varphi}}\ ||h||_\varphi, 
\eee proving \eqref{berg-ev-bound}.

Now notice that a sequence of elements of $A^2(\C^n,\varphi)$ which converges in $L^2(\C^n,\varphi)$ also converges uniformly on compact sets. This follows from \eqref{berg-ev-bound} and the assumption \eqref{berg-assumption}. Since a locally uniform limit of holomorphic functions is holomorphic, we can conclude that $A^2(\C^n,\varphi)$ is closed. Linearity of the weighted Bergman space is obvious and hence the proposition is proved.  
\end{proof}

Since $A^2(\C^n,\varphi)$ is a closed subspace, we can define the orthogonal projector $B_\varphi$ of $L^2(\C^n,\varphi)$ onto $A^2(\C^n,\varphi)$. The linear operator $B_\varphi$ is called the \emph{weighted Bergman projection}.

Inequality \eqref{berg-ev-bound} shows that the evaluation mapping $h\mapsto h(z)$ is a continuous linear functional on $A^2(\C^n,\varphi)$ for every $z\in\C^n$. By Riesz' Lemma, there exists $k_z\in A^2(\C^n,\varphi)$ such that \be
h(z)=(h,k_z)_\varphi=\int_{\C^n}h(w)\overline{k_z(w)}e^{-2\varphi(w)}d\mathcal{L}(w).
\ee
We define the \emph{weighted Bergman kernel} as follows:\be
B_\varphi(z,w):=\overline{k_z(w)}\qquad (z,w\in\C^n).
\ee

Our usage of the notation $B_\varphi$ both for the projector defined above and for this function is justified by the next proposition.

\begin{prop}\label{berg-integral-prop}
The weighted Bergman kernel is the integral kernel of the weighted Bergman projection: 
\bel\label{berg-integral}
B_\varphi(f)(z)=\int_{\C^n}B_\varphi(z,w)f(w)e^{-2\varphi(w)}d\mathcal{L}(w)\qquad\forall f\in L^2(\C^n,\varphi).
\eel
Moreover, $B_\varphi(z,w)=\overline{B_\varphi(w,z)}$ for every $z,w\in\C^n$.
\end{prop}

Observe that, since $k_z\in A^2(\C^n, \varphi)$, the integral \eqref{berg-integral} is absolutely convergent for every $z$.

\begin{proof} 
Given $f\in L^2(\C^n,\varphi)$, we write $f=B_\varphi(f)+g$, where $g$ is orthogonal to $A^2(\C^n,\varphi)$. We have\bee
B_\varphi(f)(z)&=&(B_\varphi(f),k_z)_\varphi=(f,k_z)_\varphi-(g,k_z)_\varphi\\
&=&\int_{\C^n}\overline{k_z(w)}f(w)e^{-2\varphi(w)}d\mathcal{L}(w)\\
&=&\int_{\C^n}B_\varphi(z,w)f(w)e^{-2\varphi(w)}d\mathcal{L}(w),
\eee where we used the definitions of $k_z$ and $B_\varphi$, and the fact that $k_z\in A^2(\C^n)$. This proves \eqref{berg-integral}.

Next, recall that, by definition, $\overline{B_\varphi(z,\cdot)}\in A^2(\C^n,\varphi)$ and therefore\bee
\overline{B_\varphi(z,w)} &=& \int_{\C^n}B_\varphi(w,w')\overline{B_\varphi(z,w')} e^{-2\varphi(w')}d\mathcal{L}_{\C^n}(w')\\
&=& \overline{\int_{\C^n}B_\varphi(z,w') \overline{B_\varphi(w,w')}e^{-2\varphi(w')}d\mathcal{L}_{\C^n}(w')}\\
&=& \overline{\overline{B_\varphi(w,z)}}=B_\varphi(w,z),
\eee for any $z,w\in \C^n$, as we wanted. 
\end{proof}

Proposition \ref{berg-integral-prop} has a consequence for the regularity of the Bergman kernel. In fact, by the second part of the proposition, $B_\varphi(\cdot,z)\in A^2(\C^n,\varphi)$ for every $z$. In particular $B_\varphi(z,\overline{w})$ is separately holomorphic and Hartogs' theorem allows to conclude that $B_\varphi(z,\overline{w})$ is jointly holomorphic, and hence that $B_\varphi(z,w)$ is jointly real-analytic.

The values of the weighted Bergman kernel on the diagonal have a particularly nice variational significance.

\begin{prop}\label{berg-on-diag-prop} The following identities hold for every $z\in\C^n$:\be
B_\varphi(z,z)=||B_\varphi(z,\cdot)||_\varphi^2=\sup_{h\in A^2(\C^n,\varphi):\ ||h||_\varphi=1} |h(z)|^2
\ee
\end{prop}

\begin{proof}
Since $B_\varphi(\cdot,z)\in A^2(\C^n,\varphi)$, we have\bee
B_\varphi(z,z)&=&\int_{\C^n}B_\varphi(z,w)B_\varphi(w,z)e^{-2\varphi(w)}d\mathcal{L}(w)\\
&=&\int_{\C^n}|B_\varphi(z,w)|^2e^{-2\varphi(w)}d\mathcal{L}(w)\\
&=&||B_\varphi(z,\cdot)||_\varphi^2,
\eee where in the second line we used the identity $B_\varphi(z,w)=\overline{B_\varphi(w,z)}$ of Proposition \ref{berg-integral-prop}. Since $B_\varphi(z,\cdot)=\overline{k_z(\cdot)}$, by the Riesz Lemma the norm of $B_\varphi(z,\cdot)$ equals the operator norm of the evaluation functional $h\mapsto h(z)$, which clearly equals the third term of the identity to be proved.
\end{proof}

Proposition \ref{berg-on-diag-prop} can be used to estimate $B_\varphi(z,z)$. Unfortunately, there is not an equally simple identity involving the off-diagonal values of $B_\varphi$, and estimating them requires more elaborated arguments involving the weighted $\dbar$ problem, to which the next section is dedicated.

\section{Weighted $\dbar$ problem}\label{dbar-sec}

The standing assumption for this section is still \eqref{berg-assumption}.

We begin by recalling the classical formalism of the $\dbar$ complex. We denote by $L^2_{(0,q)}(\C^n,\varphi)$ the Hilbert space of $(0,q)$-forms with coefficients in $L^2(\C^n,\varphi)$.   
Since we will be working only with forms of degree less than or equal to $2$, we confine our discussion to these cases. Adopting the standard notation for differential forms, we have that $L^2_{(0,0)}(\C^n,\varphi)=L^2(\C^n,\varphi)$,\be
L^2_{(0,1)}(\C^n,\varphi):=\left\{u=\sum_{1\leq j\leq n} u_j d\overline{z}_j:\ u_j\in L^2(\C^n,\varphi)\quad\forall j\right\},
\ee and
\be
L^2_{(0,2)}(\C^n,\varphi):=\left\{w=\sum_{1\leq j<k\leq n} w_{jk}\ d\overline{z}_j\wedge d\overline{z}_k:\ w_{jk}\in L^2(\C^n,\varphi)\quad\forall j,k\right\}.
\ee
For the norms and the scalar products in these Hilbert spaces of forms we use the same symbols $||\cdot||_\varphi$ and $(\cdot,\cdot)_\varphi$, i.e., if $u,\widetilde{u}\in L^2_{(0,1)}(\C^n,\varphi)$, we have\be
||u||_\varphi^2=\sum_{1\leq j\leq n}||u_j||_\varphi^2,\quad (u,\widetilde{u})_\varphi=\sum_{1\leq j\leq n}(u_j,\widetilde{u}_j)_\varphi,
\ee while if $w,\widetilde{w}\in L^2_{(0,1)}(\C^n,\varphi)$, we have\be
||w||_\varphi^2=\sum_{1\leq j<k\leq n}||w_{jk}||_\varphi^2,\quad (w,\widetilde{w})_\varphi=\sum_{1\leq j<k\leq n}(w_{jk},\widetilde{w}_{jk})_\varphi.
\ee 
The meaning of $||\cdot||_\varphi$ and $(\cdot,\cdot)_\varphi$ depends on whether the arguments are functions, $(0,1)$-forms or $(0,2)$-forms, but this ambiguity should not be a source of confusion. Observe that the formulas above reveal the nature of product Hilbert space of $L^2_{(0,q)}(\C^n,\varphi)$. 

We now introduce the initial fragment of the \emph{weighted $\dbar$ complex}:
\bel\label{dbar-complex}
L^2(\C^n,\varphi)\stackrel{\dbar}\longrightarrow L^2_{(0,1)}(\C^n,\varphi)\stackrel{\dbar}\longrightarrow L^2_{(0,2)}(\C^n,\varphi).
\eel
The symbol $\dbar$ denotes as usual both the operator $\dbar:L^2(\C^n,\varphi)\rightarrow L^2_{(0,1)}(\C^n,\varphi)$ defined on the domain \be
\mathcal{D}_0(\dbar):=\left\{f\in L^2(\C^n,\varphi):\ \frac{\partial f}{\partial \overline{z}_j}\in L^2(\C^n,\varphi)\ \forall j\right\}
\ee by the formula $\dbar f=\sum_j\frac{\partial f}{\partial \overline{z}_j}d\overline{z}_j$, and the operator $\dbar: L^2_{(0,1)}(\C^n,\varphi)\rightarrow L^2_{(0,2)}(\C^n,\varphi)$ defined on the domain \be
\mathcal{D}_1(\dbar):=\left\{u=\sum_ju_jd\overline{z}_j\in L^2_{(0,1)}(\C^n,\varphi):\ \frac{\partial u_k}{\partial \overline{z}_j}-\frac{\partial u_j}{\partial \overline{z}_k}\in L^2(\C^n,\varphi)\ \forall j,k\right\}
\ee by the formula $\dbar u=\sum_{j<k}\left(\frac{\partial u_k}{\partial \overline{z}_j}-\frac{\partial u_j}{\partial \overline{z}_k}\right)d\overline{z}_j\wedge d\overline{z}_k$. 

\begin{prop}\label{dbar-complex-prop}\begin{enumerate}
\item[\emph{(i)}] The two $\dbar$ operators above are closed and densely-defined.

\item[\emph{(ii)}] The weighted $\dbar$ complex \eqref{dbar-complex} is a complex, i.e., \bel\label{dbar-squared}
\dbar f\in \mathcal{D}_1(\dbar) \quad \text{and}\quad \dbar\dbar f=0\qquad\forall f\in \mathcal{D}_0(\dbar).
\eel

\item[\emph{(iii)}] The kernel of $\dbar:L^2(\C^n,\varphi)\rightarrow L^2_{(0,1)}(\C^n,\varphi)$ is $A^2(\C^n,\varphi)$.

\item[\emph{(iv)}] $C^\infty_c(\C^n)$ is dense in $\mathcal{D}_0(\dbar)$ with respect to the norm $||\cdot||_\varphi+||\dbar\cdot||_\varphi$.

\item[\emph{(v)}] The space of $(0,1)$-forms with smooth compactly supported coefficients is dense in $\mathcal{D}_1(\dbar)$ with respect to the norm $||\cdot||_\varphi+||\dbar\cdot||_\varphi$.
 
\end{enumerate}
\end{prop}

\begin{proof}
We recall two facts about $L^2(\C^n,\varphi)$:\begin{enumerate}
\item[(a)] convergence in $L^2(\C^n,\varphi)$ implies $L^2$ convergence on compact sets and hence convergence in the sense of distributions,
\item[(b)] smooth compactly supported functions are dense in $L^2(\C^n,\varphi)$.
\end{enumerate}
Fact (a) is an immediate consequence of the local boundedness of $\varphi$, while fact (b) can be proved first truncating functions on large balls and then convolving with a smooth approximate identity (and using fact (a)).

Since $(0,q)$-forms with smooth compactly supported coefficients are contained in $\mathcal{D}_q(\dbar)$ and are dense in $L^2_{(0,q)}(\C^n,\varphi)$ by fact (a), the $\dbar$ operators are densely-defined. To see that they are closed, let $g^{(m)}$ be a sequence of $(0,q)$-forms such that $g^{(m)}$ converges to $g$ in $L^2_{(0,q)}(\C^n,\varphi)$ and $\dbar g^{(m)}$ converges to $\widetilde{g}$ in $L^2_{(0,q+1)}(\C^n,\varphi)$. By fact (a), these convergences hold in the sense of distributions too. By the continuity of differential operators with respect to convergence in the sense of distributions and the uniqueness of the limit, we conclude that $\dbar g=\widetilde{g}$, proving (i). 

Property \eqref{dbar-squared} follows from the identity $\frac{\partial}{\partial \overline{z}_j}\frac{\partial}{\partial \overline{z}_k}=\frac{\partial}{\partial \overline{z}_k}\frac{\partial}{\partial \overline{z}_j}$.

Part (iii) follows from the fact that distributional solutions of the equation $\dbar f=0$ are also classical solutions. A short way to prove this is to notice that such solutions are harmonic in the weak sense and then to use the regularity theory for the Laplacian. 

We now prove part (iv), omitting the analogous proof of part (v). Let $u\in \mathcal{D}_0(\dbar)$. Let $\eta$ be smooth, compactly supported and such that $\eta(0)=1$. Put $\eta_\eps(z):=\eta(\eps z)$, for $\eps>0$. The function $\eta_\eps u$ is compactly supported and 
\be
\frac{\partial(\eta_\eps u)}{\partial \overline{z}_j}=\frac{\partial\eta_\eps}{\partial \overline{z}_j} u+\frac{\partial u}{\partial \overline{z}_j} \eta_\eps\qquad\forall j,
\ee 
from which we see that $\eta_\eps u\in \mathcal{D}_0(\dbar)$. The dominated convergence theorem shows that $\eta_\eps u$ converges to $u$ in the graph norm when $\eps$ tends to $0$. Hence compactly supported elements of $\mathcal{D}_0(\dbar)$ are dense in the graph norm.

Let now $u\in \mathcal{D}_0(\dbar)$ be compactly supported and let $\psi$ be a test function such that $\int_{\C^n}\psi=1$. We denote by $\psi^\eps$ the rescaled function $\eps^{-2n}\psi(\eps^{-1}\cdot)$. The function $u*\psi^\eps$ is smooth and compactly supported and \be
\frac{\partial(u*\psi^\eps)}{\partial \overline{z}_j}=\frac{\partial u}{\partial \overline{z}_j}*\psi^\eps\qquad\forall j.
\ee By the standard approximate identity argument plus the local boundedness of $\varphi$, we conclude that $u*\psi^\eps$ converges to $u$ in $L^2(\C^n,\varphi)$ when $\eps$ tends to $0$, and that $\frac{\partial u}{\partial \overline{z}_j}*\psi^\eps$ converges to $\frac{\partial u}{\partial \overline{z}_j}$.\end{proof}

We now formulate the \emph{weighted $\dbar$ problem}. 

Given $u\in L^2_{(0,1)}(\C^n,\varphi)$, the problem is to find a function $f\in L^2(\C^n,\varphi)$ such that \bel\label{dbar-equation}
\dbar f=u, \quad\text{i.e., }\quad \frac{\partial f}{\partial \overline{z}_j}=u_j\quad \forall j.
\eel
As a consequence of part \emph{(ii)} of Proposition \ref{dbar-complex-prop}, a necessary condition for the existence of a solution is that $u$ be in the range of $\dbar$, and hence, by \eqref{dbar-squared}, that $\dbar u=0$. In this case we say that $u$ is $\dbar$-closed.

\begin{prop}\label{dbar-canonical}
Assume that the equation \eqref{dbar-equation} admits a solution $f$. Then the set of all solutions is $f+A^2(\C^n,\varphi)$. Among these, there is a unique solution $S_\varphi(u)$ such that \be
(S_\varphi(u),h)_\varphi=0\qquad\forall h\in A^2(\C^n,\varphi), 
\ee
i.e., that is orthogonal to the weighted Bergman space. $S_\varphi(u)$ is also the solution of minimal $L^2(\C^n,\varphi)$ norm.
\end{prop}

The function $S_\varphi(u)$ is called the \emph{canonical solution} of the weighted $\dbar$ equation with datum $u$.

\begin{proof} To see the minimality of the norm, notice that $S_\varphi(u)=f-B_\varphi(f)$. The rest is elementary linear algebra.
\end{proof}

We say that the weighted $\dbar$ equation is \emph{solvable} if it admits a solution for every datum $u\in L^2_{(0,1)}(\C^n,\varphi)$ such that $\dbar u=0$.

\begin{prop}
Assume that $\varphi$ is such that the weighted $\dbar$ equation is solvable. Then \be
S_\varphi: \{u\in L^2_{(0,1)}(\C^n,\varphi):\ \dbar u=0\}\longrightarrow L^2(\C^n,\varphi)
\ee is a bounded linear operator.
\end{prop}

\begin{proof} The operator is well-defined because of solvability. The domain of $S_\varphi$ is a closed subspace of $L^2_{(0,1)}(\C^n,\varphi)$, because it is the kernel of a closed operator. The closed graph theorem then reduces the proof to showing the closure of $S_\varphi$. 

If $u^{(m)}$ converges to $u$ in $L^2_{(0,1)}(\C^n,\varphi)$ and $S_\varphi(u^{(m)})$ converges to $f$ in $L^2(\C^n,\varphi)$, the closure of $\dbar$ immediately implies that $\dbar f=u$. Since the orthogonality to $A^2(\C^n,\varphi)$ passes to the limit, we can conclude that $S_\varphi(u)=f$, as we wanted.
\end{proof}

\section{Weighted Kohn Laplacians}\label{kohn-sec}

The standing assumption for this section is that $\varphi:\C^n\rightarrow\R$ is $C^2$. In particular \eqref{berg-assumption} holds.

Taking the Hilbert space adjoints of the operators in \eqref{dbar-complex}, we have the dual complex:
\bel\label{kohn-dual}
L^2(\C^n,\varphi)\stackrel{\dbar^*_\varphi}\longleftarrow L^2_{(0,1)}(\C^n,\varphi)\stackrel{\dbar^*_\varphi}\longleftarrow L^2_{(0,2)}(\C^n,\varphi).
\eel

We use the index $\varphi$ in the symbol for these operators to stress the fact that not only the domain, but also the formal expression of $\dbar^*_\varphi$ depends on the weight $\varphi$.

\begin{prop}\label{kohn-dual-prop}
The operators in \eqref{kohn-dual} also define a complex of closed densely-defined linear operators. 
\end{prop}

\begin{proof}
The operators $\dbar^*_\varphi$ are closed densely-defined linear operators, as they are adjoints of closed densely-defined operators. The fact that \eqref{kohn-dual} is a complex is just a matter of unravelling the definition of Hilbert space adjoint and using the identity $\dbar\dbar=0$.
\end{proof}

We denote by $\mathcal{D}_1(\dbar^*_\varphi)\subseteq L^2_{(0,1)}(\C^n,\varphi)$ and $\mathcal{D}_2(\dbar^*_\varphi)\subseteq L^2_{(0,2)}(\C^n,\varphi)$ the dense domains of the two $\dbar^*_\varphi$ operators.

The \emph{weighted Kohn Laplacian} is defined by the formula
\be
\Box_\varphi:=\dbar^*_\varphi\dbar+\dbar\dbar^*_\varphi
\ee 
on the domain of $(0,1)$-forms \be
\mathcal{D}(\Box_\varphi):=\{u\in L^2_{(0,1)}(\C^n,\varphi):\ u\in\mathcal{D}_1(\dbar)\cap\mathcal{D}_1(\dbar^*_\varphi),\ \dbar u\in \mathcal{D}_2(\dbar^*_\varphi)\text{ and }\dbar^* _\varphi u\in \mathcal{D}_0(\dbar)\}.
\ee 

\begin{prop}
The weighted Kohn Laplacian is a densely-defined, closed, self-adjoint and non-negative operator on $L^2_{(0,1)}(\C^n,\varphi)$.
\end{prop}

\begin{proof} If $u\in \mathcal{D}(\Box_\varphi)$, we have\bel\label{kohn-quadratic}
(\Box_\varphi u, u)_\varphi=(\dbar^*_\varphi\dbar u,u)_\varphi+(\dbar\dbar^*_\varphi u,u)_\varphi=||\dbar u||^2_\varphi+||\dbar^*_\varphi u||^2_\varphi.
\eel
Thus $\Box_\varphi$ is non-negative.

To prove closure, let $u^{(m)}$ be a sequence in $\mathcal{D}(\Box_\varphi)$ which converges to $u\in L^2_{(0,1)}(\C^n,\varphi)$ and such that $\Box_\varphi u^{(m)}$ converges to $v\in L^2_{(0,1)}(\C^n,\varphi)$. Applying \eqref{kohn-quadratic} to $u^{(\ell)}-u^{(m)}$ and using Cauchy-Schwarz inequality, we obtain 
\be
||\dbar (u^{(\ell)}-u^{(m)})||^2_\varphi+||\dbar^*_\varphi (u^{(\ell)}-u^{(m)})||^2_\varphi\leq ||\Box_\varphi(u^{(\ell)}-u^{(m)})||_\varphi||u^{(\ell)}-u^{(m)}||_\varphi.
\ee 
Since both $u^{(m)}$ and $\Box_\varphi u^{(m)}$ are Cauchy sequences in $L^2_{(0,1)}(\C^n,\varphi)$, we conclude that $\dbar u^{(m)}$ and $\dbar^*_\varphi u^{(m)}$ are also Cauchy sequences in $L^2_{(0,2)}(\C^n,\varphi)$ and $L^2(\C^n,\varphi)$ respectively. By the closure of $\dbar$ and $\dbar^*_\varphi$, $u\in \mathcal{D}_1(\dbar)\cap\mathcal{D}_1(\dbar^*_\varphi)$, $\dbar u^{(m)}$ converges to $\dbar u$ and $\dbar^*_\varphi u^{(m)}$ converges to $\dbar^*_\varphi u$.

Now recall that $u^{(m)}\in\mathcal{D}(\Box_\varphi)$ and that $\dbar\dbar\dbar^*_\varphi u^{(m)}=0$ (by Proposition \ref{dbar-complex-prop}). By definition of adjoint,
\be
(\dbar\dbar^*_\varphi u^{(k)},\dbar^*_\varphi \dbar u^{(k)})_\varphi= (\dbar\dbar\dbar^*_\varphi u^{(k)},\dbar u^{(k)})_\varphi=0,
\ee i.e., $\dbar\dbar^*_\varphi u^{(k)}$ and $\dbar^*_\varphi \dbar u^{(k)}$ are orthogonal. Since their sum converges, the same is true of both of them separately. Again by the closure of $\dbar$ and $\dbar^*_\varphi$, we deduce that $\dbar u\in \mathcal{D}_2(\dbar^*_\varphi)$, that $\dbar^*_\varphi u\in \mathcal{D}_0(\dbar)$, and that $\dbar\dbar^*_\varphi u^{(m)}$ converges to $\dbar\dbar^*_\varphi u$, while $\dbar^*_\varphi \dbar u^{(m)}$ converges to $\dbar^*_\varphi\dbar u$. In particular $\Box_\varphi u^{(m)}$ converges to $\Box_\varphi u$. This finishes the proof of the closure.

By Hilbert space theory we have that $L^2_{(0,1)}(\C^n,\varphi)\oplus L^2(\C^n,\varphi)$ decomposes as the direct sum of the graph of $\dbar^*_\varphi$ and the image of the graph of $\dbar$ under the mapping $(u,v)\mapsto (v,-u)$. This depends only on the fact that $\dbar$ is closed and densely-defined. In particular, given $u\in L^2_{(0,1)}(\C^n,\varphi)$, there exist unique $f\in L^2_{(0,1)}(\C^n,\varphi)$ and $g\in L^2(\C^n,\varphi)$ such that \bel\label{kohn-orthogonal}
u\oplus 0=f\oplus\dbar^*_\varphi f+ \dbar g\oplus(-g), 
\eel i.e. $u=f+\dbar\dbar^*_\varphi f$, for a unique $f\in \mathcal{D}_1(\dbar^*_\varphi)$ such that $\dbar^*_\varphi f\in \mathcal{D}_0(\dbar)$. Taking norms in \eqref{kohn-orthogonal}, we find\be
||u||_\varphi^2=||f||_\varphi^2 + ||\dbar^*_\varphi f||_\varphi^2+ ||\dbar g||_\varphi^2+||g||_\varphi^2\geq ||f||^2_\varphi.
\ee Hence the mapping $(1+\dbar\dbar^*_\varphi)^{-1}$ that associates $f$ to $u$ is well-defined, linear and bounded, and of course it is the inverse of $1+\dbar\dbar^*_\varphi$ on its natural domain. We now show that $(1+\dbar\dbar^*_\varphi)^{-1}$ is self-adjoint. Let $f,g\in L^2_{(0,1)}(\C^n,\varphi)$. We have\bee
&&((1+\dbar\dbar^*_\varphi)^{-1}f,g)_\varphi=((1+\dbar\dbar^*_\varphi)^{-1}f,(1+\dbar\dbar^*_\varphi)(1+\dbar\dbar^*_\varphi)^{-1}g)_\varphi\\
&=&((1+\dbar\dbar^*_\varphi)^{-1}f,(1+\dbar\dbar^*_\varphi)^{-1}g)_\varphi+((1+\dbar\dbar^*_\varphi)^{-1}f,\dbar\dbar^*_\varphi (1+\dbar\dbar^*_\varphi)^{-1}g)_\varphi.
\eee Since $(1+\dbar\dbar^*_\varphi)^{-1}g$ is in the domain of $\dbar\dbar^*_\varphi$, we can write\bee
&&((1+\dbar\dbar^*_\varphi)^{-1}f,(1+\dbar\dbar^*_\varphi)^{-1}g)_\varphi+((1+\dbar\dbar^*_\varphi)^{-1}f,\dbar\dbar^*_\varphi (1+\dbar\dbar^*_\varphi)^{-1}g)_\varphi\\
&=&((1+\dbar\dbar^*_\varphi)^{-1}f,(1+\dbar\dbar^*_\varphi)^{-1}g)_\varphi+(\dbar\dbar^*_\varphi (1+\dbar\dbar^*_\varphi)^{-1}f,(1+\dbar\dbar^*_\varphi)^{-1}g)_\varphi\\
&=& (f,(1+\dbar\dbar^*_\varphi)^{-1}g)_\varphi,
\eee proving the self-adjointness. Since $(1+\dbar\dbar^*_\varphi)^{-1}$ is clearly injective, it has also dense range. This proves that the natural domain $\{u\in \mathcal{D}_1(\dbar^*_\varphi):\ \dbar^*_\varphi u\in\mathcal{D}_0(\dbar)\}$ of $1+\dbar\dbar^*_\varphi$ is dense.
 As a consequence, $1+\dbar\dbar^*_\varphi$ is also self-adjoint on the natural domain. By an analogous argument one can see that $1+\dbar^*_\varphi\dbar$ is self-adjoint on the natural domain.

Now consider the operator \be
L=(1+\dbar\dbar^*_\varphi)^{-1}+(1+\dbar^*_\varphi\dbar)^{-1}-1,
\ee which is patently bounded and self-adjoint. Writing\bel\label{kohn-L}
L=(1+\dbar\dbar^*_\varphi)^{-1}-(1+\dbar^*_\varphi\dbar-1)(1+\dbar^*_\varphi\dbar)^{-1}=(1+\dbar\dbar^*_\varphi)^{-1}-\dbar^*_\varphi\dbar(1+\dbar^*_\varphi\dbar)^{-1}.
\eel The second term is contained in the domain of $\dbar\dbar^*_\varphi$ (by the complex property $\dbar^*_\varphi\dbar^*_\varphi=0$ of Proposition \ref{kohn-dual-prop}), and hence the range of $L$ is also contained in the domain of $\dbar\dbar^*_\varphi$. By a symmetrical argument, the range of $L$ is contained in the domain of $\dbar^*_\varphi\dbar$ and hence in the domain of $\Box_\varphi$. Since \bel\label{kohn-L-2}
(1+\Box_\varphi)L=L+\dbar\dbar^*_\varphi L+\dbar^*_\varphi\dbar L=1,
\eel where we used the representation \eqref{kohn-L} and the one with $\dbar$ and $\dbar^*$ exchanged. Therefore $L$ is injective and thus, being self-adjoint, it has dense range, and a fortiori $\mathcal{D}(\Box_\varphi)$ is dense. Moreover, by \eqref{kohn-L-2} it is clear that the range of $1+\Box_\varphi$ is $L^2_{(0,1)}(\C^n,\varphi)$. Since $1+\Box_\varphi$ is strictly positive and hence injective on the domain of $\Box_\varphi$, we can conclude that it is the inverse of $L$. In particular $1+\Box_\varphi$ is self-adjoint. This immediately implies that $\Box_\varphi$ is self-adjoint.
\end{proof}

Another important object is the quadratic form \be
\mathcal{E}_\varphi(u,v):=(\dbar u,\dbar v)_\varphi+(\dbar^*_\varphi u,\dbar^*_\varphi v)_\varphi,
\ee defined for $u,v\in\mathcal{D}(\mathcal{E}_\varphi):=\mathcal{D}_1(\dbar)\cap\mathcal{D}_1(\dbar^*_\varphi)$. Notice that, by definition of Hilbert space adjoints, \be
(\Box_\varphi u,v)=\mathcal{E}_\varphi(u,v) \qquad \forall u\in \mathcal{D}(\Box_\varphi),\quad \forall v\in \mathcal{D}(\mathcal{E}_\varphi).
\ee
We will simply write $\mathcal{E}_\varphi(u)$ for $\mathcal{E}_\varphi(u,u)$.

\begin{prop}\label{kohn-dbar-star}\begin{enumerate}
\item[\emph{(i)}] We have
\be
\mathcal{D}_1(\dbar_\varphi^*) =\left\{u\in L^2_{(0,1)}(\C^n,\varphi):\ \sum_j\left(-\frac{\partial u_j}{\partial z_j}+2\frac{\partial \varphi}{\partial z_j}u_j\right)\in L^2(\C^n,\varphi)\right\},
\ee and if $u\in \mathcal{D}_1(\dbar_\varphi^*)$ then $\dbar^*_\varphi u=\sum_j\left(-\frac{\partial u_j}{\partial z_j}+2\frac{\partial \varphi}{\partial z_j}u_j\right)$.
\item[\emph{(ii)}] We have
\be
\mathcal{D}_2(\dbar_\varphi^*) =\left\{u\in L^2_{(0,2)}(\C^n,\varphi):\ \sum_j\left(-\frac{\partial u_{jk}}{\partial z_j}+2\frac{\partial \varphi}{\partial z_j}u_{jk}\right)\in L^2(\C^n,\varphi)\quad\forall k\right\},
\ee and if $u\in \mathcal{D}_2(\dbar_\varphi^*)$ then $\dbar^*_\varphi u=\sum_j\left(-\frac{\partial u_{jk}}{\partial z_j}+2\frac{\partial \varphi}{\partial z_j}u_{jk}\right)$, where we are using the standard convention that $u_{jk}=-u_{kj}$.

\item[\emph{(iii)}] The space of $(0,1)$-forms with smooth compactly supported coefficients is dense in $\mathcal{D}(\mathcal{E}_\varphi)$ with respect to the norm $||\cdot||_\varphi+||\dbar\cdot||_\varphi+||\dbar^*_\varphi\cdot||_\varphi$.
\end{enumerate}
\end{prop}

\begin{proof}
Assume that $u\in\mathcal{D}_1(\dbar^*_\varphi)$. If $g\in C^\infty_c(\C^n)\subseteq\mathcal{D}_0(\dbar)$, we have the identity\bee
\int_{\C^n}\dbar^*_\varphi u\cdot  ge^{-2\varphi}=(\dbar^*_\varphi u,\overline{g})_\varphi=(u,\dbar \overline{g})_\varphi =\sum_j\int_{\C^n}e^{-2\varphi}u_j \frac{\partial g}{\partial z_j}.
\eee By the arbitrariness of $g$, we have \be
\dbar^*_\varphi u=-e^{2\varphi}\sum_j\frac{\partial(e^{-2\varphi}u_j)}{\partial z_j}=\sum_j\left(-\frac{\partial u_j}{\partial z_j}+2\frac{\partial \varphi}{\partial z_j}u_j\right).\ee

To conclude we need to show the converse to what we have just seen, i.e., that any $u\in L^2_{(0,1)}(\C^n,\varphi)$ such that $f:=\sum_j\left(-\frac{\partial u_j}{\partial z_j}+2\frac{\partial \varphi}{\partial z_j}u_j\right)\in L^2(\C^n,\varphi)$ lies in $\mathcal{D}_1(\C^n,\varphi)$. Let $g$ be as above. Then, by the definition of distributional derivative, \bel\label{kohn-adjoint-formula}
\sum_j\int_{\C^n}\frac{\partial g}{\partial \overline{z}_j} \overline{u_j}e^{-2\varphi}=\int_{\C^n} g \overline{f}e^{-2\varphi}.
\eel 
Now, by part (iv) of Proposition \ref{dbar-complex-prop}, given $g\in \mathcal{D}_0(\dbar)$ we can find a sequence $g^{(m)}\in C^\infty_c(\C^n)$ such that $g^{(m)}$ converges to $g$ in $L^2(\C^n,\varphi)$ and $\dbar g^{(m)}$ converges to $\dbar g$ in $L^2_{(0,1)}(\C^n,\varphi)$. Thus we can pass to the limit in \eqref{kohn-adjoint-formula} to conclude that the same identity holds for any $g\in \mathcal{D}_{(0,1)}(\dbar)$. This shows that $u\in \mathcal{D}_{(0,1)}(\dbar^*_\varphi)$ and that $\dbar^*_\varphi u=f$, concluding the proof of (i). 

The proof of (ii) is analogous and hence omitted.

Let now $u\in \mathcal{D}(\mathcal{E}_\varphi)$. Put $\eta_\eps(z):=\eta(\eps z)$, where $\eta\in C^\infty_c(\C^n)$ is such that $\eta(0)=1$. We have\be
\dbar^*_\varphi (\eta_\eps u)=\eta_\eps\dbar^*_\varphi u-\sum_ju_j \frac{\partial \eta_\eps}{\partial z_j},
\ee from which we see that $\eta_\eps u\in \mathcal{D}_1(\dbar^*_\varphi)$ and that, by dominated convergence, $\dbar^*_\varphi (\eta_\eps u)$ converges to $\dbar^*_\varphi u$ in $L^2(\C^n,\varphi)$ when $\eps$ tends to $0$. A similar computation shows that $\dbar(\eta_\eps u)$ converges to $\dbar^*u$. Since $\eta_\eps u$ converges to $u$, this shows that compactly supported elements of $\mathcal{D}(\mathcal{E}_\varphi)$ are dense in the required norm. 
In analogy with the proof of part (iv) of Proposition \ref{dbar-complex-prop}, we now assume that $u\in \mathcal{D}(\mathcal{E}_\varphi)$ is compactly supported and consider its componentwise convolution with an approximate identity $\psi^\eps$:\bee
\dbar^*_\varphi (u*\psi^\eps)&=&\sum_j \left(-\frac{\partial u_j}{\partial z_j}*\psi^\eps+2\frac{\partial \varphi}{\partial z_j}(u_j*\psi^\eps)\right)\\
&=&\dbar^*_\varphi u*\psi^\eps \\
&+&2\sum_j\int_{\C^n}\left(\frac{\partial \varphi}{\partial z_j}(z)-\frac{\partial \varphi}{\partial z_j}(z-w)\right)u_j(z-w)\psi^\eps(w)d\mathcal{L}(w).
\eee
The first term converges to $\dbar^*_\varphi u$ in $L^2(\C^n,\varphi)$ because of the local boundedness of $\varphi$, while to see that the second term converges to $0$ in the same norm we need to use the $C^2$ regularity of $\varphi$, or more precisely the local Lipschitz property of $\frac{\partial \varphi}{\partial z_j}$. This is essentially the classical Friedrichs' argument and we omit the elementary details.
\end{proof}

\section{The Morrey-Kohn-H\"ormander formula \\ and a Caccioppoli-type inequality}\label{MKH-sec}

In this section we assume that our weight $\varphi:\C^n\rightarrow\R$ is $C^2$ and \emph{plurisubharmonic}, i.e., that the \emph{complex Hessian}\be
H_\varphi(z):=\left(\frac{\partial^2\varphi}{\partial z_j\partial \overline{z}_k}(z)\right)_{j,k=1}^n\qquad (z\in\C^n)
\ee
is everywhere non-negative definite:\be
(H_\varphi(z) v,v)=\sum_{j,k}\frac{\partial^2\varphi}{\partial z_j\partial \overline{z}_k}(z)v_j\overline{v}_k\geq0\qquad\forall v=(v_1,\dots,v_n)\in \C^n,\quad\forall z\in\C^n.
\ee
The complex Hessian is an $n\times n$ Hermitian matrix-valued continuous mapping. From now on, we identify the $(0,1)$-form $u=\sum_j u_jd\overline{z}_j$ with the complex vector field $u=(u_1,\dots,u_n):\C^n\rightarrow \C^n$. In particular it makes sense to consider the non-negative function \be
(H_\varphi u,u)(z)=\sum_{j,k=1}^n\frac{\partial^2\varphi}{\partial z_j\partial \overline{z}_k}(z)u_j(z)\overline{u_k(z)},
\ee obtained by evaluating the quadratic form associated to the complex Hessian on the vector field $u$.

The next proposition is the reason for the central role played by $H_\varphi$ in weighted complex analysis.

\begin{prop}\label{MKH-prop} Let $u\in L^2_{(0,1)}(\C^n,\varphi)$. Then $u\in\mathcal{D}(\mathcal{E}_\varphi)$ if and only if \be
\frac{\partial u_j}{\partial\overline{z}_k}\in L^2(\C^n,\varphi)\quad\forall j,k,\quad\sqrt{(H_\varphi u,u)}\in L^2(\C^n,\varphi),
\ee and we have\bel\label{MKH-formula}
\mathcal{E}_\varphi(u,v)=\sum_{j,k}\int_{\C^n}\frac{\partial u_j}{\partial\overline{z}_k}\overline{\frac{\partial v_j}{\partial\overline{z}_k}}e^{-2\varphi}+2\int_{\C^n} (H_\varphi u,v)e^{-2\varphi},
\eel
for every $u,v\in \mathcal{D}(\mathcal{E}_\varphi)$. 
\end{prop}

Equation \eqref{MKH-formula} is called the \emph{Morrey-Kohn-H\"ormander formula}.

\begin{proof} 
We begin by proving \eqref{MKH-formula} for $u,v$ with compactly supported and smooth coefficients. Since $\mathcal{E}_\varphi(u)=||\dbar u||_\varphi^2+||\dbar^*_\varphi u||_\varphi^2$, we can use part (i) of Proposition \ref{kohn-dbar-star} to write\bel\label{MKH-1}
\mathcal{E}_\varphi(u)=\sum_{1\leq j<k\leq n}\int_{\C^n}|\dbar_j u_k-\dbar_k u_j|^2e^{-\varphi}+\int_{\C^n}|\sum_{j=1}^n\partial_{\varphi,j}^*u_j|^2e^{-\varphi},
\eel where $\partial_{\varphi,j}^*f=\partial_j f-2\partial_j \varphi f$. Notice that we are using the lighter symbols $\partial_j$ and $\dbar_j$ for $\frac{\partial}{\partial z_j}$ and $\frac{\partial}{\partial \overline{z}_j}$. We expand the first term of the right-hand side of \eqref{MKH-1}. The result is:\bee
&&\sum_{1\leq j<k\leq n}\int_{\C^n}\left(|\dbar_j u_k|^2+|\dbar_k u_j|^2-\dbar_j u_k\partial_k \overline{u_j}-\partial_j \overline{u_k}\dbar_k u_j\right)e^{-2\varphi}\\
&=&\sum_{j,k=1}^n\int_{\C^n}\left(|\dbar_j u_k|^2-\dbar_j u_k\partial_k \overline{u_j}\right)e^{-2\varphi}.
\eee Notice that we could add the terms corresponding to $j=k$ since they sum up to zero. Two integration by parts show that \bee
&&-\sum_{j,k=1}^n\int_{\C^n}\dbar_j u_k\partial_k \overline{u_j}e^{-2\varphi}=\sum_{j,k=1}^n\int_{\C^n}\partial_{\varphi,k}^*\dbar_j u_k\cdot \overline{u_j} e^{-2\varphi}\\
&=&\sum_{j,k=1}^n\int_{\C^n}[\partial_{\varphi,k}^*,\dbar_j] u_k\cdot \overline{u_j} e^{-2\varphi}+\sum_{j,k=1}^n\int_{\C^n}\dbar_j \partial_{\varphi,k}^*u_k\cdot \overline{u_j} e^{-2\varphi}\\ 
&=&2\sum_{j,k=1}^n\int_{\C^n}\dbar_j\partial_k\varphi u_k\cdot \overline{u_j} e^{-2\varphi}-\sum_{j,k=1}^n\int_{\C^n}\partial_{\varphi,k}^*u_k\cdot \overline{\partial_{\varphi,j}^*u_j} e^{-2\varphi}\\
&=&2\int_{\C^n}(H_\varphi u,u) e^{-2\varphi}-\int_{\C^n}|\sum_{j=1}^n\partial_{\varphi,k}^*u_j|^2 e^{-2\varphi}.
\eee In the third line we used the key commutation relation $[\partial_{\varphi,k}^*,\dbar_j]=2\dbar_j\partial_k\varphi $. Putting all the terms together we obtain \bel\label{MKH-formula-smooth}
\mathcal{E}_\varphi(u)=\sum_{j,k}\int_{\C^n}|\dbar_k u_j|^2e^{-2\varphi}+2\int_{\C^n} (H_\varphi u,u)e^{-2\varphi}.
\eel The regularity of $\varphi$ and $u$ ensures that all the computations are valid.

If $u\in \mathcal{D}(\mathcal{E}_\varphi)$, part (ii) of Proposition \ref{kohn-dbar-star} gives a sequence $u^{(m)}$ of $(0,1)$-forms with smooth compactly supported coefficients such that $\mathcal{E}_\varphi(u-u^{(m)})\rightarrow0$. Using \eqref{MKH-formula-smooth} and the non-negativity of $H_\varphi$, one can prove that $\dbar_ku_j^{(m)}$ is a Cauchy sequence in $L^2(\C^n,\varphi)$ for every $j$ and $k$, and that the same is true of the sequence $\sqrt{(H_\varphi u^{(m)}, u^{(m)})}$. We can pass to the limit and conclude that $\dbar_ku_j\in L^2(\C^n,\varphi)$ for every $j$ and $k$, that $\sqrt{(H_\varphi u,u)}\in L^2(\C^n,\varphi)$, and that \eqref{MKH-formula} holds for $u=v\in \mathcal{D}(\mathcal{E}_\varphi)$. Since both sides of \eqref{MKH-formula} are quadratic forms, the formula for general $u,v$ may be deduced by polarization. 

The proof that\be 
\dbar_ku_j\in L^2(\C^n,\varphi)\quad \forall j,k,\quad\sqrt{(H_\varphi u,u)}\in L^2(\C^n,\varphi)\quad\Longrightarrow\quad  u\in\mathcal{D}(\mathcal{E}_\varphi)
\ee is via a similar approximation argument and we omit it.\end{proof}

In order to prove a Caccioppoli-type inequality for solutions of the $\Box_\varphi$ equation, we need the following auxiliary proposition.

\begin{prop}\label{MKH-computation}
Assume that $u\in \mathcal{D}(\Box_\varphi)$ and let $\eta$ be a real-valued bounded Lipschitz function. Then $\eta u\in \mathcal{D}(\mathcal{E}_\varphi)$ and 
\be 
\mathcal{E}_\varphi(\eta u)=\frac{1}{4}\int_{\C^n}|\nabla \eta|^2|u|^2e^{-2\varphi} + \Re(\eta\Box_\varphi u,\eta u)_\varphi.
\ee
\end{prop}

\begin{proof}
The fact that $\eta u\in \mathcal{D}(\mathcal{E}_\varphi)$ follows easily from part (i) of Proposition \ref{kohn-dbar-star}. Then we have (using the symbols $\partial_j$ and $\dbar_j$ as in the proof of Proposition \ref{MKH-prop}): \bee
|\dbar_k(\eta u_j)|^2&=&|\dbar_k\eta|^2 |u_j|^2+\Re\left(\eta^2|\dbar_ku_j|^2+2\eta\partial_k\eta \cdot\overline{u_j}\dbar_k u_j\right)\\
&=&|\dbar_k\eta|^2 |u_j|^2+\Re\left(\dbar_k u_j(\eta^2\partial_k\overline{u_j}+\partial_k(\eta^2) \overline{u_j})\right)\\
&=&|\dbar_k\eta|^2  |u_j|^2+\Re\left(\dbar_k u_j\overline{\dbar_k(\eta^2u_j)}\right).
\eee 
Integrating this identity and using the Morrey-Kohn-H\"ormander formula we obtain\bee
\mathcal{E}_\varphi(\eta u)&=&\frac{1}{4}\int_{\C^n}|\nabla \eta|^2|u|^2e^{-2\varphi}+\Re\left(\sum_{j,k=1}^n\int_{\C^n}\dbar_k u_j\overline{\dbar_k(\eta^2u_j)}e^{-2\varphi}\right) \\
&+& 2\Re\left(\int_{\C^n}(H_\varphi u,\eta^2 u)e^{-2\varphi}\right)\\
&=&\frac{1}{4}\int_{\C^n}|\nabla \eta|^2|u|^2e^{-2\varphi}+\Re\left(\mathcal{E}_\varphi(u,\eta^2u)\right).
\eee Since $\eta^2u\in \mathcal{D}(\mathcal{E}_\varphi)$ and $u\in \mathcal{D}(\Box_\varphi)$, we have $\mathcal{E}_\varphi(u,\eta^2u)=(\Box_\varphi u,\eta^2u)_\varphi=(\eta\Box_\varphi u,\eta u)_\varphi$. This concludes the proof. 
\end{proof}

We are now in a position to state and prove the Caccioppoli-type inequality.

\begin{lem}\label{MKH-caccioppoli}
Assume that $u\in \mathcal{D}(\Box_\varphi)$ and that $\Box_\varphi u$ vanishes on $B(z,R)$. If $r<R$, then\be
\int_{B(z,r)}|\dbar^*_\varphi u|^2e^{-2\varphi}\leq (R-r)^{-2} \int_{B(z,R)}|u|^2e^{-2\varphi}.
\ee 
\end{lem}

\begin{proof} Let $\eta$ be Lipschitz, real-valued, identically equal to $1$ on $B(z,r)$, and supported on $B(z,R)$. Since $\eta\Box_\varphi u=0$, Proposition \ref{MKH-computation} yields\be
||\dbar^*_\varphi (\eta u)||_\varphi^2\leq \mathcal{E}_\varphi(\eta u)=\frac{1}{4}\int_{\C^n}|\nabla \eta|^2|u|^2e^{-2\varphi}\leq \frac{||\nabla \eta||_\infty^2}{4}||\chi_{B(z,R)} u||_\varphi^2.
\ee
By part (i) of Proposition \ref{kohn-dbar-star}, $\dbar^*_\varphi(\eta u)=\eta\dbar^*_\varphi u - \sum_{j=1}^n\frac{\partial\eta}{\partial z_j}u_j$, we have\bee
||\chi_{B(z,r)} \dbar^*_\varphi u||_\varphi&\leq& ||\eta\dbar^*_\varphi u||_\varphi\\
&\leq& ||\dbar^*_\varphi (\eta u)||_\varphi+\frac{||\nabla\eta||_\infty}{2}||\chi_{B(z,R)} u||_\varphi\\
&\leq &||\nabla\eta||_\infty||\chi_{B(z,R)} u||_\varphi.
\eee It is clear that we can choose $\eta$ such that $||\nabla\eta||_\infty=\frac{1}{R-r}$, and this gives the thesis.\end{proof}

\section{$\mu$-coercivity for weighted Kohn Laplacians}\label{coerc-sec}

In this section we introduce the notion of $\mu$-coercivity for the weighted Kohn Laplacian and we study its most basic consequences. We continue to assume that our weight $\varphi:\C^n\rightarrow\R$ is $C^2$ and plurisubharmonic.

\begin{dfn}\label{coerc-def}
Given a measurable function $\mu:\C^n\rightarrow[0,+\infty)$, we say that $\Box_\varphi$ is \emph{$\mu$-coercive} if the following inequality holds
\bel\label{coerc-formula}
\mathcal{E}_\varphi(u)\geq ||\mu u||_\varphi^2 \qquad\forall u\in \mathcal{D}(\mathcal{E}_\varphi).
\eel
\end{dfn}

The next proposition collects a few basic facts about $\mu$-coercivity. 

\begin{prop}\label{coerc-prop} Assume that $\Box_\varphi$ is $\mu$-coercive, and that 
\bel\label{coerc-inf}\inf_{z\in\C^n}\mu(z)>0.\eel
Then:
\begin{enumerate}
\item[\emph{(i)}] $\Box_\varphi$ has a bounded, self-adjoint and non-negative inverse $N_\varphi$ such that \bel\label{coerc-neumann-bound}
||\mu N_\varphi g||_\varphi\leq ||\mu^{-1} g||_\varphi \qquad\forall g\in L^2_{(0,1)}(\C^n,\varphi).
\eel
\item[\emph{(ii)}] The weighted $\dbar$ equation is solvable, and $\dbar^*_\varphi N_\varphi$ is the canonical solution $S_\varphi$ of Proposition \ref{dbar-canonical}. More precisely, if $f\in L^2_{(0,1)}(\C^n,\varphi)$ is $\dbar$-closed, i.e., $\dbar f=0$, then $u:=\dbar^*_\varphi N_\varphi f$ is such that $\dbar u=f$ and $u\perp A^2(\C^n,\varphi)$. Moreover, \bel\label{coerc-canonical-bound}
||u||_\varphi\leq ||\mu^{-1} f||_\varphi.
\eel 
\item[\emph{(iii)}] We have the identity
\bel\label{coerc-bergman} B_\varphi f= f-\dbar^*_\varphi N_\varphi \dbar f\qquad \forall f\in \mathcal{D}_0(\dbar).\eel
\end{enumerate}
\end{prop}

The operator $N_\varphi$ is customarily called the \emph{$\dbar$-Neumann operator}. 

\begin{proof}

(i) To see that $\Box_\varphi$ is injective, observe that, if $\Box_\varphi u=0$, inequality \eqref{coerc-formula} implies that $mu=0$ and hence, by \eqref{coerc-inf}, that $u=0$. By self-adjointness, $\Box_\varphi$ has dense range. If $g\in L^2_{(0,1)}(\C^n,\varphi)$, the anti-linear functional $\lambda_g:\Box_\varphi u\mapsto (g,u)_\varphi$ is then well-defined on a dense subset of $L^2_{(0,1)}(\C^n,\varphi)$. It satisfies the bound\bel\label{coerc-lambda}
|\lambda_g(\Box_\varphi u)|=|(g,u)_\varphi|\leq ||\mu^{-1} g||_\varphi||\mu u||_\varphi.
\eel
An application of the Cauchy-Schwarz inequality shows that $\mu$-coercivity implies, for any $u\in\mathcal{D}(\Box_\varphi)$, \be
||\mu u||_\varphi^2\leq \mathcal{E}_\varphi(u)=(\Box_\varphi u,u)_\varphi\leq ||\mu^{-1}\Box_\varphi u||_\varphi||\mu u||_\varphi,
\ee i.e., $||\mu u||_\varphi\leq||\mu^{-1}\Box_\varphi u||_\varphi$ ($||\mu u||_\varphi$ is finite for any $u\in\mathcal{D}(\Box_\varphi)$ by $\mu$-coercivity).
Plugging this inequality into \eqref{coerc-lambda}, we obtain \bel\label{coerc-existence}
|\lambda_g(\Box_\varphi u)|\leq ||\mu^{-1} g||_\varphi||\mu^{-1}\Box_\varphi u||_\varphi .
\eel
Since $\mu^{-1}$ is bounded, $\lambda_g$ may be uniquely extended to a continuous anti-linear functional on $L^2_{(0,1)}(\C^n,\varphi)$ and hence there exists $N_\varphi g\in L^2_{(0,1)}(\C^n,\varphi)$ such that $(u,g)_\varphi=(\Box_\varphi u,N_\varphi g)_\varphi$ for every $u\in \mathcal{D}(\Box_\varphi)$. This means that $N_\varphi g\in\mathcal{D}(\Box_\varphi)$ and that $\Box_\varphi N_\varphi g=g$. In particular $\Box_\varphi$ is surjective and $N_\varphi$, being the inverse of $\Box_\varphi$, is a bounded, self-adjoint, and non-negative operator. Inequality \eqref{coerc-neumann-bound} follows from \eqref{coerc-existence}:\be
||\mu N_\varphi g||_\varphi=\sup_{||w||_\varphi=1}|(\mu N_\varphi g,w)_\varphi|=\sup_{||w||_\varphi=1,||\mu w||_\varphi<+\infty}|\lambda_g(\mu w)|\leq ||\mu^{-1} g||_\varphi.
\ee In the second identity we used the fact that $\mu w\in L^2_{(0,1)}(\C^n,\varphi)$ for $w$ in a dense subspace as a consequence of $\mu$-coercivity.

(ii) Let $f$ and $u$ be as in the statement. We compute\be
\dbar u = \dbar\dbar^*_\varphi N_\varphi f= \Box_\varphi N_\varphi f-\dbar^*_\varphi \dbar N_\varphi f=f-\dbar^*_\varphi \dbar N_\varphi f.
\ee Notice that $N_\varphi f\in \mathcal{D}(\Box_\varphi)$ and hence the computation is meaningful. Since $f$ and $\dbar u$ are both $\dbar$-closed, the identity above implies that $\dbar^*_\varphi \dbar N_\varphi f\in \mathcal{D}_1(\dbar)$ and $\dbar\dbar^*_\varphi \dbar N_\varphi f=0$. In particular \be
0=(\dbar\dbar^*_\varphi \dbar N_\varphi f,\dbar N_\varphi f)_\varphi=||\dbar^*_\varphi \dbar N_\varphi f||_\varphi^2
\ee and hence $\dbar u=f$. Notice that one can repeat the above argument to show that $\dbar N_\varphi f=0$, but we don't need this fact. The solution $u$ is obviously orthogonal to $A^2(\C^n,\varphi)$, because it is in the range of $\dbar^*_\varphi$. To obtain the bound on $u$ we observe that\bee
||u||_\varphi^2&=&(\dbar^*_\varphi N_\varphi f,u)_\varphi=(N_\varphi f,\dbar u)_\varphi\\
&=&(N_\varphi f,f)_\varphi=(\mu N_\varphi f,\mu^{-1}f)_\varphi\leq ||\mu^{-1} f||_\varphi^2,
\eee where the last inequality is \eqref{coerc-neumann-bound}.

(iii) Let $f\in\mathcal{D}_0(\dbar)$. Since $\dbar f$ is $\dbar$-closed, part \emph{(ii)} shows that $\dbar^*_\varphi N_\varphi \dbar f$ is orthogonal to $A^2(\C^n,\varphi)$ and that $f-\dbar^*_\varphi N_\varphi \dbar f\in A^2(\C^n,\varphi)$. Hence $B_\varphi f=f-\dbar^*_\varphi N_\varphi \dbar f$.\end{proof}

The next result ties $\mu$-coercivity and discreteness of the spectrum of $\Box_\varphi$.

\begin{thm}\label{coerc-pp-thm}
Assume that $\Box_\varphi$ is $\mu$-coercive for some $\mu$ such that \be
\lim_{z\rightarrow}\mu(z)=+\infty.
\ee
Then the operator $\Box_\varphi$ has discrete spectrum.
\end{thm}

We recall that we say that a self-adjoint operator has discrete spectrum if its spectrum is a discrete subset of $\R$ consisting of eigenvalues of finite multiplicity.

\begin{proof}
Haslinger proved in \cite{haslinger-funct} that $\Box_\varphi$ admits a compact inverse $N_\varphi$ if and only if for every $\eps>0$ there exists $R<+\infty$ such that if \be
u\in\mathcal{D}(\mathcal{E}_\varphi)\quad\text{is such that}\quad\mathcal{E}_\varphi(u)\leq1,
\ee then
\be
\int_{|z|\geq R}|u|^2e^{-2\varphi}\leq \eps.
\ee
This condition is clearly equivalent to $\mu$-coercivity for some $\mu$ diverging at infinity. To conclude the proof, recall that a compact operator has discrete spectrum and that if the inverse of a self-adjoint operator has discrete spectrum, the same is true of the operator itself.
\end{proof}

\section{First intermezzo: Radius functions and associated distances}\label{radius-sec}

We say that $\rho:\R^d\rightarrow (0,+\infty)$ is a \emph{radius function} if it is Borel and there exists a constant $C<+\infty$ such that for every $x\in\R^d$ we have\bel\label{radius-ineq}
C^{-1}\rho(x)\leq \rho(y)\leq C\rho(x) \qquad \forall y\in B(x,\rho(x)).
\eel

In other words, a radius function $\rho$ is approximately constant on the ball centered at $x$ of radius $\rho(x)$. 

To any radius function $\rho$, we associate the Riemannian metric $\rho(x)^{-2}dx^2$. In fact, we are interested only in the associated Riemannian distance, which we describe explicitly. If $I$ is a compact interval and $\gamma:I\rightarrow\R^d$ is a piecewise $C^1$ curve, we define\be
L_\rho(\gamma):=\int_I\frac{|\gamma'(t)|}{\rho(\gamma(t))}dt.
\ee 
Notice that the integrand $\frac{|\gamma'(t)|}{\rho(\gamma(t))}$ is defined on the complement of the finite set of times where $\gamma'$ is discontinuous, and it is a measurable function, because $\rho$ is assumed to be Borel. Moreover, the integral is absolutely convergent because $\rho^{-1}$ is locally bounded. 

Given $x,y\in\R^d$, we put\be
d_\rho(x,y):=\inf_\gamma L_\rho(\gamma),
\ee where the $\inf$ is taken as $\gamma$ varies over the collection of curves connecting $x$ and $y$. Finally, we define $B_\rho(x,r):=\{y\in\R^d:\ d_\rho(x,y)<r\}$. 

\begin{prop}\label{radius-prop} \par The function $d_\rho$ just defined is a distance and \be
B_\rho(x,C^{-1}r)\subseteq B(x,r\rho(x))\subseteq B_\rho(x,Cr) \qquad\forall r\leq1,\ x\in\R^d,
\ee where $C$ is the constant appearing in \eqref{radius-ineq}. Moreover, the function \be
y\mapsto d_\rho(x,y)\ee 
is locally Lipschitz for every $x$, and for almost every $y\in \R^d$ we have \bel\label{radius-nabla}
|\nabla_y d_\rho(x,y)|\leq \frac{C}{\rho(y)}.
\eel
\end{prop}

\begin{proof}
Let $x,y\in\R^d$ be such that $|x-y|=s\rho(x)$, for some $s>0$. Take any piecewise $C^1$ curve $\gamma:[0,T]\rightarrow\R^d$ connecting $x$ and $y$, and let $T_0$ be the minimum time such that $|x-\gamma(T_0)|=\min\{s,1\}\rho(x)$. By \eqref{radius-ineq}, $\rho(\gamma(t))\leq C\rho(x)$ for every $t\in[0,T_0)$. Hence, \bee
L_\rho(\gamma)&\geq& \int_0^{T_0}\frac{|\gamma'(t)|}{\rho(\gamma(t))}dt\geq \frac{C^{-1}}{\rho(x)}\int_0^{T_0}|\gamma'(t)|dt\\
&\geq& \frac{C^{-1}}{\rho(x)} \min\{s,1\}\rho(x)=C^{-1}\min\{s,1\}.
\eee By the arbitrariness of $\gamma$, we conclude that $d_\rho(x,y)\geq C^{-1}\min\left\{\frac{|x-y|}{\rho(x)},1\right\}$. This implies that $d_\rho$ is non-degenerate and hence a genuine distance, triangle inequality and symmetry being obvious. 
It also shows that if $y$ lies in $B_\rho(x,C^{-1}r)$ ($r\leq1$), then $r>\min\left\{\frac{|x-y|}{\rho(x)},1\right\}$, and then the minimum has to be equal to $\frac{|x-y|}{\rho(x)}$, proving the first inclusion of the statement.

To prove the second inclusion, we use the fact that $\rho(u)\geq C^{-1}\rho(x)$ for every $u\in B(x,r\rho(x))$, if $r\leq 1$. Given $y\in B(x,r\rho(x))$, define $\sigma(t)=x+t(y-x)$ and notice that \be
d_\rho(x,y)\leq \int_0^1\frac{|\sigma'(t)|}{\rho(\sigma(t))}dt\leq \frac{C}{\rho(x)}|x-y|< Cr,
\ee so that $B(x,r\rho(x))\subseteq B_\rho(x,Cr)$, that is the second inclusion to be proved.

Fix now $x,y\in \R^d$ and let $h\in\R^d$ be such that $|h|<\rho(y)$. As above, we have $d_\rho(y,y+h)\leq C\frac{|h|}{\rho(y)}$. The triangle inequality then yields\bel\label{radius-lip}
|d_\rho(x,y+h)-d_\rho(x,y)|\leq d_\rho(y,y+h)\leq C\frac{|h|}{\rho(y)}\qquad\forall h: |h|<\rho(y).
\eel By the local boundedness of $\rho^{-1}$, we conclude that $d_\rho(x,\cdot)$ is locally Lipschitz. Rademacher's theorem implies that $d_\rho$ is almost everywhere differentiable and \eqref{radius-lip} translates into \eqref{radius-nabla}.\end{proof}

We conclude this section with two elementary propositions. The second one is a very classical construction of a covering.

\begin{prop}\label{radius-max}
If $\rho_1$ and $\rho_2$ are two radius functions on $\R^d$, then $\rho_1\vee\rho_2:=\max\{\rho_1,\rho_2\}$ is a radius function.
\end{prop}

\begin{proof}
Assume that $C<+\infty$ is a constant for which \eqref{radius-ineq} holds both for $\rho_1$ and $\rho_2$.

Fix $x\in\R^d$ and assume that $\rho_1\vee\rho_2(x)=\rho_1(x)$. If $y\in B(x,\rho_1\vee\rho_2(x))$, then the first inequality in \eqref{radius-ineq} for $\rho_1$ yields \be
\rho_1\vee\rho_2(x)=\rho_1(x)\leq C\rho_1(y)\leq C\rho_1\vee\rho_2(y).
\ee
If $\rho_1\vee\rho_2(x)=\rho_2(x)$ the conclusion would be the same (using \eqref{radius-ineq} for $\rho_2$).

Now there are two possibilities: either $\rho_1\vee\rho_2(x)\leq \rho_1\vee\rho_2(y)$, in which case the same argument with $x$ and $y$ swapped gives the bound $\rho_1\vee\rho_2(y)\leq C\rho_1\vee\rho_2(x)$, or the converse inequality $\rho_1\vee\rho_2(y)< \rho_1\vee\rho_2(x)$ holds. In both cases the proof is completed.
\end{proof}

\begin{prop}\label{radius-covering}
If $\rho$ is a radius function there is a countable set $\{x_k\}_{k\in\N}\subseteq \R^d$ such that: \begin{enumerate}
\item[\emph{(i)}] $\left\{B(x_k,\rho(x_k))\right\}_{k\in\N}$ is a covering of $\R^d$,
\item[\emph{(ii)}] any $x\in\R^d$ lies in at most $K$ of the balls of the covering, where $K$ depends only on $C$ and $d$.
\end{enumerate}
\end{prop}

\begin{proof}
Let $\{x_k\}_{k\in\N}$ be such that $\mathcal{B}=\left\{B(x_j,\frac{\rho(x_j))}{1+C^2}\right\}_{j\in\N}$ be any maximal disjoint subfamily of $\left\{B(x,\frac{\rho(x))}{1+C^2}\right\}_{x\in\R^d}$ (of course, any maximal subfamily is countable). If $x\neq x_j$ for every $j\in\N$, then by maximality there exists a $k$ such that $B\left(x,\frac{\rho(x)}{1+C^2}\right)$ intersects $B\left(x_k,\frac{\rho(x_k)}{1+C^2}\right)$. Picking a point in the intersection and using twice \eqref{radius-ineq}, we see that $\rho(x)\leq C^2\rho(x_k)$, and thus that $x\in B(x_k,\rho(x_k))$. This proves (i).

To see that (ii) holds, fix $k$ and consider the indices $j_1,\dots,j_N$ corresponding to balls of the covering intersecting $B(x_k,\rho(x_k))$. By the same argument as above, we see that $C^{-2}\rho(x_k)\leq \rho(x_{j_\ell})\leq C^2\rho(x_k)$. This means that $B(x_k,(1+2C^2)\rho(x_k))$ contains $B\left(x_{j_\ell}, \frac{\rho(x_{j_\ell})}{1+C^2}\right)$ for every $\ell$. These balls are disjoint by construction and their radius is $\geq\frac{\rho(x_k)}{C^2(1+C^2)} $, therefore $N$ has to be bounded by a constant which depends only on $C$ and the dimension $d$.
\end{proof}

Now consider a measurable function\be
V:\R^d\rightarrow [0,+\infty).
\ee 
We assume that $V$ is locally bounded, not almost everywhere zero, and satisfies the following \emph{$L^\infty$-doubling condition}:\bel\label{pot-doubling}
||V||_{L^\infty(B(x,2r))}\leq D ||V||_{L^\infty(B(x,r))}\qquad\forall x\in\R^d,\ r>0,
\eel
where $D<+\infty$ is a constant independent of $x$ and $r>0$.

We want to associate to every such $V$ a certain radius function. Before giving the detailed arguments, let us describe the heuristics behind it. \newline

\par If we have a free quantum particle moving in $\R^d$ and $B$ is a ball of radius $r$, the uncertainty principle asserts that in order to localize the particle on the ball $B$ one needs an energy of the order of $r^{-2}$. If the particle is not free, but it is subject to a potential $V$, this energy increases by the size of $V$ on $B$. This means in particular that if $B$ is such that $\max_B V\leq r^{-2}$, then the amount of energy required for the localization is comparable to the one in the free case: in this case \emph{one does not feel the potential on $B$}. The radius function $\rho_V$ we are going to describe gives at every point $x$ the largest radius $\rho_V(x)$ such that one cannot feel the potential $V$ on $B(x,\rho_V(x))$.\newline

To formalize the discussion above, we begin by defining the function \be
f(x,r):=r^2||V||_{L^\infty(B(x,r))}\qquad (x\in\R^d, r>0).\ee 
We highlight two properties of $f$:\begin{enumerate}
\item $f(x,r)$ is strictly monotone in $r$ for every fixed $x$. 
\item $\lim_{r\rightarrow0+}f(x,r)=0$ and $\lim_{r\rightarrow+\infty}f(x,r)=+\infty$ for every $x$.
\end{enumerate}
To verify them, it is useful to observe that since $V$ is not almost everywhere $0$, an iterated application of \eqref{pot-doubling} shows that $||V||_{L^\infty(B)}>0$ for every non empty ball $B$.

We define \be
\rho_V(x):=\sup\{r>0:\ f(x,r)\leq 1\}.
\ee
By properties (1) and (2) above the $\sup$ exists and it is positive and finite. 

\begin{prop}\label{pot-rsquared}
We have\be
\frac{\rho_V(x)^{-2}}{4D}\leq||V||_{L^\infty(B(x,\rho_V(x)))} \leq \rho_V(x)^{-2}\qquad\forall x\in\R^d.
\ee
\end{prop}

\begin{proof}
The right inequality follows immediately from the definition of $\rho_V$. To prove the one on the left, observe that \eqref{pot-doubling} implies\bee
f(x, 2\rho_V(x))&=&4\rho_V(x)^2||V||_{L^\infty(B(x,2\rho_V(x)))}\\
&\leq& 4D\rho_V(x)^2||V||_{L^\infty(B(x,\rho_V(x)))}.
\eee The definition of $\rho_V(x)$ shows that the last term is smaller than $4D$, while the first one is larger than $1$. This finishes the proof.
\end{proof}

The next two results together prove that $\rho_V$ is a radius function.

\begin{prop} The function $\rho_V$ is Borel. 
\end{prop}

\begin{proof}
We have to see that $\{x:\rho_V(x)>t\}$ is a Borel set for every $t>0$, but \bee
\{x:\rho_V(x)>t\}&=&\{x:\ \exists r>t \text{ s.t. } f(x,r)\leq 1\}\\
&=&\{x:\ \exists r\in\Q\cap(t,+\infty) \text{ s.t. } f(x,r)\leq 1\}\\
&=&\cup_{r\in\Q\cap(t,+\infty)}\{x:\ ||V||_{L^\infty(B(x,r))}\leq r^{-2}\}.
\eee
It then suffices to verify that $||V||_{L^\infty(B(x,r))}$ is Borel in $x$ for every fixed $r>0$. In fact, $||V||_{L^\infty(B(\cdot,r))}$ is lower semi-continuous: $||V||_{L^\infty(B(x,r))}>u$ if and only if $V>u$ on a subset of positive measure of $B(x,r)$, and this property is clearly preserved by small perturbations of the center $x$.
\end{proof}

\begin{prop}\label{pot-to-rad} The function $\rho_V$ satisfies the following inequalities for every $x,y\in\R^d$: \be
 C^{-1}  \max\left\{\frac{|x-y|}{\rho_V(x)},1\right\}^{-M_1}\rho_V(x)\leq \rho_V(y)\leq C  \max\left\{\frac{|x-y|}{\rho_V(x)},1\right\}^{M_2}\rho_V(x),
\ee where $C, M_1, M_2$ depends only on the $L^\infty$-doubling constant $D$ appearing in \eqref{pot-doubling}. 

In particular $\rho_V$ is a radius function.
\end{prop}

\begin{proof} We have already seen that $\rho_V:\R^d\rightarrow(0,+\infty)$ is well-defined and Borel. Assume that $|x-y|<2^k \rho_V(x)$, for some integer $k\geq 1$. If $|x-y|<s<2^k \rho_V(x)$ we have\bee
s^2||V||_{L^\infty(B(y,s))}&\leq& s^2||V||_{L^\infty(B(x,2s))}\\
&\leq&D^{k+1}s^2||V||_{L^\infty(B(x,2^{-k}s))}\\
&=& 2^{2k}D^{k+1}(2^{-k}s)^2||V||_{L^\infty(B(x,2^{-k}s))}\\
&\leq&2^{2k}D^{k+1},
\eee
where in the third line we used $k+1$ times \eqref{pot-doubling} and in the last one we used the fact that $2^{-k}s<\rho_V(x)$.

 In particular $f(y,2^{-k}D^{-\frac{k+1}{2}}s)\leq 1$ and hence $2^{-k}D^{-\frac{k+1}{2}}s\leq \rho_V(y)$. By the arbitrariness of $s<2^k \rho_V(x)$, we conclude that \bel\label{pot-half-bound}
 \rho_V(x)\leq D^{\frac{k+1}{2}}\rho_V(y)\leq 2^{Mk}\rho_V(y),
 \eel for an integer $M$ depending only on $D$. 
 
 Inequality \eqref{pot-half-bound} gives $|x-y|<2^{(M+1)k}\rho_V(y)$, so that we can apply the above argument with $x$ and $y$ inverted, we conclude that $\rho_V(y)\leq 2^{M(M+1)k}\rho_V(x)$. Now the thesis follows choosing $k$ such that $2^k$ is comparable to $\max\left\{\frac{|x-y|}{\rho_V(x)},1\right\}$.\end{proof}

\section{Second intermezzo: a variant \\ of Fefferman-Phong inequality}\label{fph-sec}

We now prove a version of an inequality going back to Fefferman and Phong (see, e.g., \cite{fefferman-uncertainty} or \cite{shen}).\newline

Let $V:\R^d\rightarrow [0, +\infty)$ be a locally bounded function satisfying the reverse H\"older inequality\bel\label{fph-RH}
||V||_{L^\infty(B(x,r))}\leq Ar^{-d}\int_{B(x,r)}V\qquad\forall x\in\R^d, r>0,
\eel where $A<+\infty$ is a constant which is independent of $x$ and $r$. It is well-known that $V(x)dx$ is a doubling measure, and hence $V$ satisfies \eqref{pot-doubling} (see, e.g., \cite{stein-bigbook}). In particular the radius function $\rho_V$ is well-defined.

\begin{prop}\label{fph-lem}
There is a constant $C$ which depends only on $V$ such that for every $f\in C^1_c(\R^d)$ we have\be
\int_{\R^d}\rho_V^{-2}|f|^2\leq C\left(\int_{\R^d}|\nabla f|^2+\int_{\R^d}V|f|^2\right).
\ee
\end{prop}

\begin{proof}
The function $f$ is fixed throughout the proof. If $x\in\R^d$, we put $B=B(x,\rho_V(x))$. Integrating in $(y,y')\in B\times B$ the trivial bound\be
V(y)|f(y')|^2\leq 2V(y)|f(y)-f(y')|^2+2V(y')|f(y')|^2, 
\ee we get\be
\int_BV \int_B|f|^2\leq2 ||V||_{L^\infty(B)}\int_B\int_B|f(y)-f(y')|^2dydy'+\omega_d\rho_V(x)^d\int_B V|f|^2,
\ee 
where $\omega_d$ is the measure of the unit ball of $\R^d$.
We recall that we have the following form of Poincar\'e inequality: \be
\int_{B\times B}|f(y)-f(y')|^2dydy'\leq C_d r^{d+2}\int_B|\nabla f|^2,
\ee where $C_d$ is a constant depending only on $d$ and $B$ is any euclidean ball of radius $r$. Combining it with Proposition \ref{pot-rsquared} we find\bel\label{fph-quad}
\int_BV \int_B|f|^2\leq2C_d \rho_V(x)^d\int_B|\nabla f|^2+\omega_d\rho_V(x)^d\int_B V|f|^2,
\eel  
The reverse H\"older condition and Proposition \ref{pot-rsquared} give\bel\label{fph-int}
\int_BV\geq \frac{1}{A}\rho_V(x)^d||V||_{L^\infty(B)}\geq \frac{1}{4DA}\rho_V(x)^{d-2}.
\eel
Putting \eqref{fph-quad} and \eqref{fph-int} together, and using Proposition \ref{pot-to-rad} to bring $\rho_V^{-2}$ inside the integral, we obtain\bel\label{fph-quad-2}
\int_B\rho_V^{-2}|f|^2\leq C'\left( \int_B|\nabla f|^2+\int_B V|f|^2\right),
\eel where $C'$ depends on $V$, but not on $x$ or $f$. Summing the inequalities \eqref{fph-quad-2} corresponding to the points $x_j$ given by Proposition \ref{radius-covering} (applied to $\rho_V$) we obtain\be
\int_{\R^d}\rho_V^{-2}|f|^2\leq C'K\left( \int_{\R^d}|\nabla f|^2+\int_{\R^d} V|f|^2\right),
\ee where $K$ is the constant appearing in Proposition \ref{radius-covering}. Putting $C=C'K$ we obtain the statement.\end{proof}

\section{Weighted Kohn Laplacians\\ and matrix Schr\"odinger operators}\label{kohnschrod-sec}

In this section we show that weighted Kohn Laplacians are unitarily equivalent to certain matrix Schr\"odinger operators. This observation in the special case $n=1$ already appeared in the literature, for instance in \cite{berndtsson}.  

Let $\varphi:\C^n\rightarrow \R$ be $C^2$ and plurisubharmonic. We identify $\C^n$ with $\R^{2n}$ using the real coordinates $(x_1,y_1,\dots,x_n,y_n)$ such that $z_j=x_j+iy_j$ for every $j$. It will be useful to define:\bel\label{kohnschrod-perp}
\nabla^\perp\varphi:=\left(-\frac{\partial\varphi}{\partial y_1},\frac{\partial\varphi}{\partial x_1},\dots,-\frac{\partial\varphi}{\partial y_n},\frac{\partial\varphi}{\partial x_n}\right),
\eel which is the \emph{symplectic gradient of $\varphi$}.

It is easy to verify that the mapping\bee
U_\varphi: L^2(\C^n,\C^n)&\longrightarrow &L^2_{(0,1)}(\C^n,\varphi)\\
 u=(u_1,\dots,u_n)&\longmapsto& \sum_{j=1}^n e^\varphi u_jd\overline{z}_j
\eee
is a surjective unitary transformation. If $u\in C^2_c(\C^n,\C^n)$ then $U_\varphi u\in\mathcal{D}(\Box_\varphi)$, because it is a $(0,1)$-form with $C^2_c$ coefficients, and by Proposition \ref{kohn-dbar-star} these forms are in the domain of $\Box_\varphi$.

\begin{prop}\label{kohnschrod-prop}
Consider the $n\times n$ Hermitian matrix-valued electric potential \be
V=8H_\varphi-4\text{tr}(H_\varphi)I_n,
\ee where $I_n$ is the $n\times n$ identity matrix, and the magnetic potential \be
A=\nabla^\perp\varphi.
\ee
We have the following identity\bel\label{kohnschrod-id}
\Box_\varphi (U_\varphi u)=U_\varphi\left(\frac{1}{4}\mathcal{H}_{A,V}u\right)\qquad \forall u\in C^2_c(\C^n,\C^n).
\eel
\end{prop}


Let us stress the fact that while $\Box_\varphi$ is a genuine self-adjoint operator, the matrix Schr\"odinger operator $\mathcal{H}_{A,V}$ has been defined only formally. Identity \eqref{kohnschrod-id} may be used to extend $\mathcal{H}_{V,A}$ to a domain on which it is self-adjoint and unitarily equivalent to the weighted Kohn Laplacian.

The proof of Proposition \ref{kohnschrod-prop} is based on the following computation, which we present as a separate lemma in order to be able to use it again later.

\begin{lem}\label{kohnschrod-comp}
If $A$ is as in Proposition \ref{kohnschrod-prop}, we have\be
\sum_{j=1}^n\int_{\C^n}\left|\frac{\partial (e^\varphi f)}{\partial \overline{z}_j} \right|^2e^{-2\varphi}=\frac{1}{4}\int_{\C^n}|\nabla_A f|^2-\frac{1}{4}\int_{\C^n}\Delta\varphi|f|^2\qquad\forall f\in C^2_c(\C^n).
\ee
\end{lem}

\begin{proof}

Define the vector fields \be
X_j:=\frac{\partial}{\partial x_j}+i\frac{\partial\varphi}{\partial y_j},\quad Y_j:=\frac{\partial}{\partial y_j}-i\frac{\partial\varphi}{\partial x_j}\qquad (j=1,\dots,n).
\ee  Recalling \eqref{kohnschrod-perp} and the definition of magnetic gradient (see Section \ref{mag-sec}), we have $|\nabla_Af|^2=\sum_{j=1}^n\left(|X_jf|^2+|Y_jf|^2\right)$. Notice that\be
\frac{\partial}{\partial \overline{z}_j}+\frac{\partial\varphi}{\partial \overline{z}_j}=\frac{1}{2}\left(X_j+iY_j\right).
\ee and that the formal adjoints of $X_j$ and $Y_j$ are $-X_j$ and $-Y_j$ respectively. Therefore we have \bee
&&\int_{\C^n}\left|\frac{\partial (e^\varphi f)}{\partial \overline{z}_j} \right|^2e^{-2\varphi}=\sum_{j=1}^n\int_{\C^n}\left|\frac{\partial f}{\partial \overline{z}_j}+\frac{\partial\varphi}{\partial \overline{z}_j}f\right|^2\\
&=&\frac{1}{4}\sum_{j=1}^n\int_{\C^n}\left|(X_j+iY_j)f\right|^2\\
&=&\frac{1}{4}\sum_{j=1}^n\left(\int_{\C^n}\left|X_jf\right|^2+\int_{\C^n}\left|X_jf\right|^2+\int_{\C^n}X_jf\overline{iY_j f}+\int_{\C^n}iY_jf\overline{X_j f}\right)\\
&=&\frac{1}{4}\sum_{j=1}^n\left(\int_{\C^n}|\nabla_Af|^2-i\int_{\C^n}[X_j,Y_j]f\cdot \overline{f}\right)
\eee
Observing that $\sum_{j=1}^n[X_j,Y_j]=-i\Delta\varphi$, we obtain the thesis.
\end{proof}

\begin{proof}[Proof of Proposition \ref{kohnschrod-prop}]

Since both $\Box_\varphi$ and $\mathcal{H}_{V,A}$ are formally self-adjoint, by polarization it is enough to prove that \be
(\Box_\varphi U_\varphi u, U_\varphi u )_\varphi=(U_\varphi\mathcal{H}_{V,A}u,U_\varphi u)_\varphi=(\mathcal{H}_{V,A}u,u)_0,
\ee where the parenthesis on the right represent the scalar product in the unweighted space $L^2(\C^n,\C^n)$. Using Proposition \ref{MKH-prop} and Lemma \ref{kohnschrod-comp}, we can write \bee
&&(\Box_\varphi U_\varphi u, U_\varphi u )_\varphi=\mathcal{E}_\varphi(U_\varphi u)\\
&=&\sum_{j,k=1}^n\int_{\C^n}\left|\frac{\partial (e^\varphi f)}{\partial \overline{z}_j} \right|^2e^{-2\varphi}+2\int_{\C^n}(H_\varphi u,u)\\
&=&\sum_{k=1}^n\left(\frac{1}{4}\int_{\C^n}|\nabla_A u_k|^2-\frac{1}{4}\int_{\C^n}\Delta\varphi|u_k|^2\right)+2\int_{\C^n}(H_\varphi u,u)\\
&=&\frac{1}{4}\left(\int_{\C^n}|\nabla_A u|^2+\int_{\C^n}((8H_\varphi-\Delta\varphi I_n) u,u)\right).
\eee To complete the proof notice that \be
\text{tr}(H_\varphi)=\sum_{j=1}^n\frac{\partial^2\varphi}{\partial z_j\partial \overline{z}_j}=\frac{1}{4}\Delta\varphi.
\ee and recall \eqref{schrod-energy}.\end{proof}

\section{One complex variable versus several complex variables in the analysis of $\Box_\varphi$}\label{onevsmany-sec}

Proposition \ref{kohnschrod-prop} reveals a radical difference between the one-dimensional case ($n=1$) and the higher-dimensional case ($n\geq2$) in the analysis of the weighted Kohn Laplacian. If $n=1$, the potential $V$ is the scalar function \be
8H_\varphi-4\text{tr}(H_\varphi)=4\text{tr}(H_\varphi)=\Delta\varphi,
\ee which is non-negative, while if $n\geq2$ the potential $V$ is matrix-valued and \be
\text{tr}(V)=\text{tr}(8H_\varphi-4\text{tr}(H_\varphi)I_n)=(8-4n)\text{tr}(H_\varphi)=(2-n)\Delta\varphi
\ee is non-positive. As a consequence, the potential $V$ always has non-positive eigenvalues if $n\geq2$. 

In the one-variable case one may combine Proposition \ref{kohnschrod-prop}, identity \eqref{schrod-energy} and the diamagnetic inequality of Lemma \ref{mag-diamag}, obtaining (for $u\in L^2(\C,\C)=L^2(\C)$):\bee
\mathcal{E}_\varphi(U_\varphi u)&=&\frac{1}{4}\mathcal{E}_{V,A}(u)\\
&=&\frac{1}{4}\left(\int_\C|\nabla_Au|^2+\int_\C V|u|^2\right)\\
&\geq&\frac{1}{4}\left(\int_\C|\nabla|u||^2+\int_\C\Delta\varphi|u|^2\right).
\eee

The last term is the energy $\mathcal{E}_{\Delta\varphi,0}(|u|)$ of the compactly supported Lipschitz function $|u|$ in presence of the scalar electric potential $\Delta\varphi$ and of no magnetic field. If one can prove the bound \bel\label{onevsmany-bound}
\mathcal{E}_{\Delta\varphi,0}(u)\geq \int_\C \mu^2 |u|^2,
\eel for some $\mu:\C\rightarrow[0,+\infty)$ and for all Lipschitz functions $u$, one can immediately deduce that \be
\mathcal{E}_\varphi(u)\geq  \int_\C \left(\frac{\mu}{2}\right)^2 |u|^2e^{-2\varphi}, 
\ee i.e., that $\Box_\varphi$ is $\frac{\mu}{2}$-coercive. Notice that \eqref{onevsmany-bound} is the Fefferman-Phong inequality (see, e.g., Lemma \ref{fph-lem} of Section \ref{fph-sec}). This is the approach followed by Christ in \cite{christ}.\newline

Such a route is not viable in general in several variables: if $u\in C^\infty_c(\C^n,\C^n)$, applying the diamagnetic inequality we get:\bee
\mathcal{E}_\varphi(U_\varphi u)&=&\frac{1}{4}\mathcal{E}_{V,A}(u)\\
&=&\frac{1}{4}\left(\sum_{k=1}^n\int_{\C^n}|\nabla_Au_k|^2+\int_{\C^n} (Vu,u)\right)\\
&\geq&\frac{1}{4}\sum_{k=1}^n\left(\int_{\C^n}|\nabla|u_k||^2+\int_{\C^n}\lambda(V)|u_k|^2\right),
\eee
 where $\lambda(V)$, the minimal eigenvalue of $V$, is everywhere non-positive. The last term is $\sum_{k=1}^n\mathcal{H}_{\lambda(V),0}(|u_k|)$, which is not even non-negative in general, so no estimate like \eqref{onevsmany-bound} can hold.

Nevertheless, a variant of this approach can be very useful in the special case in which the eigenvalues of $H_\varphi$ are comparable, as we show in the next section.

\section{$\mu$-coercivity when the eigenvalues are comparable}\label{mucomparable-sec}

We  are going to use an argument which appears, e.g., in the proof of Theorem 5.6 of \cite{haslinger-helffer}. 

If $\varphi:\C^n\rightarrow\R$ is $C^2$ and plurisubharmonic, we say that \emph{$H_\varphi$ has comparable eigenvalues} if there exists $\delta>0$ such that \bel\label{mucomparable-ineq}
(H_\varphi(z)v,v)\geq \delta \Delta\varphi(z)|v|^2\qquad\forall z\in\C^n, \ v\in\C^n.
\eel Since $\Delta\varphi/4$ is the trace of $H_\varphi$, the condition above is clearly equivalent to the global comparability of any pair of eigenvalues of $H_\varphi$. If \eqref{mucomparable-ineq} holds, necessarily $4\delta\leq1$.

Notice that \eqref{mucomparable-ineq} holds automatically in one complex dimension. 

\begin{lem}\label{mucomparable-lem}
Let $\varphi:\C^n\rightarrow\R$ be $C^2$, plurisubharmonic and such that \begin{enumerate}
\item[\emph{(i)}] the eigenvalues of $H_\varphi$ are comparable, i.e., \eqref{mucomparable-ineq} holds,
\item[\emph{(ii)}] $\Delta\varphi$ satisfies the reverse-H\"older inequality
\bel\label{mucomparable-RH}
||\Delta\varphi||_{L^\infty(B(z,r))}\leq Ar^{-2n}\int_{B(z,r)}\Delta\varphi\qquad\forall z\in\C^n, r>0.
\eel 
\end{enumerate}
Then $\Box_\varphi$ is $\mu$-coercive, where \be
\mu=c\rho_{\Delta\varphi}^{-1}.
\ee Here $c>0$ is a constant which depends only on $\Delta\varphi$ and $\delta$, and $\rho_{\Delta\varphi}$ is the radius function associated to the potential $\Delta\varphi$, as described in the Appendix.
\end{lem}

\begin{proof}
Let $u\in C^2_c(\C^n,\C^n)$. We have \bee
\mathcal{E}_\varphi(U_\varphi u)&=&\sum_{j,k}\int_{\C^n}\left|\frac{\partial (e^\varphi u_j)}{\partial\overline{z}_k}\right|^2e^{-2\varphi}+2\int_{\C^n} (H_\varphi u,u)\\
&\geq&\sum_{j=1}^n\left(4\delta \sum_{k=1}^n\int_{\C^n}\left|\frac{\partial (e^\varphi u_j)}{\partial\overline{z}_k}\right|^2e^{-2\varphi}+2\delta\int_{\C^n}\Delta\varphi|u_j|^2\right),
\eee 
where we used Proposition \ref{MKH-prop}, hypothesis (i), and the inequality $4\delta\leq 1$. We invoke Lemma \ref{kohnschrod-comp} to obtain\be
\mathcal{E}_\varphi(U_\varphi u)\geq \delta\sum_{j=1}^n\left(\int_{\C^n}|\nabla_A u_j|^2+\int_{\C^n}\Delta\varphi|u_j|^2\right).
\ee
We can now apply the diamagnetic inequality (Lemma \ref{mag-diamag}) to deduce that\be
\mathcal{E}_\varphi(U_\varphi u)\geq  \delta\sum_{j=1}^n\left(\int_{\C^n}|\nabla |u_j||^2+\int_{\C^n}\Delta\varphi|u_j|^2\right).
\ee 
Since we assumed hypothesis (ii), we can apply the Fefferman-Phong inequality of Lemma \ref{fph-lem}:\be
\mathcal{E}_\varphi(U_\varphi u)\geq C^{-1} \delta\int_{\C^n}\rho_{\Delta\varphi}^{-2}|u|^2.
\ee To complete the proof, we replace $u$ with $e^{-\varphi} u$ and recall the density of $(0,1)$-forms with smooth compactly supported coefficients in $\mathcal{D}(\mathcal{E}_\varphi)$ (Proposition~\ref{kohn-dbar-star}).
\end{proof}

\section{Admissible weights}\label{adm-sec}

It is now time to introduce the class of weights to which our main results apply.

\begin{dfn}\label{adm-dfn} A $C^2$ plurisubharmonic weight $\varphi:\C^n\rightarrow\R$ is said to be admissible if:\begin{enumerate}
\item[\emph{(1)}] the following $L^\infty$ doubling condition holds:\be
\sup_{B(z,2r)}\Delta\varphi\leq D\sup_{B(z,r)}\Delta\varphi \quad\forall z\in \C^n,\ r>0,
\ee for some finite constant $D$ which is independent of $z$ and $r$,
\item[\emph{(2)}] there exists $c>0$ such that 
\bel\label{adm-lower}
\inf_{z\in\C^n}\sup_{w\in B(z,c)}\Delta\varphi(w)>0.
\eel
\end{enumerate}
\end{dfn}

If $\varphi$ is an admissible weight, then \be V\equiv\Delta\varphi:\C^n\rightarrow[0,+\infty)\ee satisfies condition \eqref{pot-doubling} of Section \ref{radius-sec} (we are identifying $\C^n$ and $\R^{2n}$), and is continuous and not everywhere zero, because of \eqref{adm-lower}. Thus we have the associated radius function $\rho_{\Delta\varphi}$. Since here we are dealing only with one fixed weight $\varphi$, we can drop the subscript and denote this radius function just by $\rho$. To this radius function we can in turn associate a distance function $d$ on $\C^n\equiv\R^{2n}$. We call $\rho$ the \emph{maximal eigenvalue radius function} and $d$ the \emph{maximal eigenvalue distance} corresponding to the weight $\varphi$. The details are in the Appendix. The reason for this name is simple: as we have already remarked, $\Delta\varphi$ is four times the trace of the complex Hessian of  $\varphi$ and hence it is comparable to its maximal eigenvalue.

\begin{prop}\label{adm-bounded}
The maximal eigenvalue radius function associated to an admissible weight is bounded.
\end{prop}

\begin{proof}
By the definition of the radius function associated to a potential (see Section \ref{radius-sec}), \be
\rho(z):=\sup\{r>0: \sup_{w\in B(z,r)}\Delta\varphi(w)\leq r^{-2}\},
\ee and the statement follows immediately from \eqref{adm-lower}.
\end{proof}

The next proposition will play a key role in later sections.

\begin{prop}\label{adm-hol} Let $\varphi$ and $\rho$ be as above. There exists a constant $C$ depending only on $\varphi$ such that if $h:B(z,r)\rightarrow\C$ is holomorphic and $r\leq \rho(z)$, then \be
|h(z)|^2e^{-2\varphi(z)}\leq \frac{C}{|B(z,r)|}\int_{B(z,r)}|h|^2e^{-2\varphi}.
\ee 
\end{prop}
Notice that the above estimate holds for every ball if $\varphi=0$. One can think of $\rho(z)$ as the maximal scale at which \emph{one does not feel the weight}. This should be compared with the heuristic discussion in Section \ref{radius-sec}. 

The proof of Proposition \ref{adm-hol} is based on the following lemma.

\begin{lem}\label{adm-gauge} Let $\varphi$ and $\rho$ be as above. For every $z\in\C^n$ there exists a $C^2$ function $\psi:B(z,\rho(z))\rightarrow\R$ such that $H_\varphi=H_\psi$, i.e., $\frac{\partial^2\varphi}{\partial z_j\partial \overline{z}_k}=\frac{\partial^2\psi}{\partial z_j\partial \overline{z}_k}$ $\forall j,k$, and such that \be
\sup_{w\in B(z,\rho(z))}|\psi(w)|\leq C_n,
\ee where $C_n$ is a constant which depends only on the dimension $n$.
\end{lem}

\begin{proof} We recall the following fact: if $\omega$ is a continuous and bounded $(1,1)$-form defined on $B(z,r)\subseteq\C^n$ such that:\begin{enumerate}
\item $\overline{\omega}=\omega$,   
\item it is $d$-closed in the sense of distributions,
\end{enumerate}
then there exists a real-valued, bounded and continuous function $\psi$ on $B(z,r)$ such that $i\partial\dbar\psi=\omega$ and $||\psi||_\infty\leq C_n r^2||\omega||_\infty$. The latter $L^\infty$ norm is the maximum of the $L^\infty$ norms of the coefficients of $\omega$. This is Lemma 4 of \cite{delin}, where a proof can be found. 

To deduce our lemma notice that $i\partial\dbar\varphi$, restricted to $B(z,\rho(z))$, satisfies conditions (1) and (2) above (recall that $\partial$ and $\dbar$ anti-commute).  The $L^\infty$ norm of $i\partial\dbar\varphi$ on $B(z,\rho(z))$ is bounded by a constant times $\rho(z)^{-2}$ by the definition of $\rho$ and the elementary observation that the coefficients of a non-negative matrix are bounded by its trace. Therefore there is a real-valued function $\psi$ on $B(z,\rho(z))$ such that $\partial\dbar\psi=\partial\dbar\varphi$ and $||\psi||_\infty\leq C_n$, as we wanted. Notice that $\psi-\varphi$ is harmonic, and hence smooth, so that $\psi$ has the same regularity as $\varphi$.
\end{proof}

\begin{proof}[Proof of Proposition \ref{adm-hol}]
By $A\lesssim B$ we mean $A\leq CB$, where $C$ is a constant depending only on $\varphi$. Fix $z$ and $r$ and let $\psi$ be the function given by Lemma \ref{adm-gauge}. Since $\psi-\varphi$ is pluriharmonic, there exists a holomorphic function $H$ on $B(z,r)$ such that $\Re(H)=\psi-\varphi$. If $h$ is as in the statement, using the $L^\infty$ bound on $\psi$, we can estimate
\be 
|h(z)|^2e^{-2\varphi(z)}\lesssim|h(z)|^2e^{2\psi(z)-2\varphi(z)}= |h(z)e^{H(z)}|^2.
\ee Applying the mean-value property and the Cauchy-Schwarz inequality to the holomorphic function $he^H$, we find
\bee
|h(z)e^{H(z)}|^2&\leq& \frac{1}{|B(z,r)|}\int_{B(z,r)}|he^H|^2\\
&=& \frac{1}{|B(z,r)|}\int_{B(z,r)}|h|^2 e^{\widetilde{\varphi}-\varphi}\\
&\lesssim& \frac{1}{|B(z,r)|}\int_{B(z,r)}|h|^2 e^{-\varphi}.
\eee This concludes the proof. \end{proof}

\section{Exponential decay of canonical solutions}\label{exp-sec}

Now that all the ingredients are in place, in this section we prove that if $\varphi$ is an admissible weight such that $\Box_\varphi$ is $\mu$-coercive and $\mu$ satisfies certain assumptions, then the canonical solutions of the weighted $\dbar$ problem exhibit a fast decay outside the support of the datum, in a way which is described in terms of $\mu$.

In the statement of the result we use the following terminology: a constant $C$ is \emph{allowable} if it depends only on $\varphi$, $\mu$ and the dimension $n$, and $A\lesssim B$ stands for the inequality $A\leq CB$, where the implicit constant $C$ is allowable. 

\begin{thm}\label{exp-thm} Let $\varphi$ be an admissible weight and assume that there exists $\kappa:\C^n\rightarrow(0,+\infty)$ such that:\begin{enumerate}
\item $\kappa$ is a bounded radius function, 
\item $\kappa(z)\geq\rho(z)$ for every $z\in\C^n$,
\item $\Box_\varphi$ is $\kappa^{-1}$-coercive.
\end{enumerate}
Recall that $\rho$ is the maximal eigenvalue function introduced in Section \ref{adm-sec}.

Then there are allowable constants $\eps, r_0,R_0>0$ such that the following holds true. Let $z\in\C^n$ and let $u\in L^2_{(0,1)}(\C^n,\varphi)$ be $\dbar$-closed and identically zero outside $B_\kappa(z,r_0)$. If $f$ is the canonical solution of \be
\dbar f=u,
\ee which exists by part \emph{(ii)} of Proposition \ref{coerc-prop}, then the pointwise bound\be
|f(w)|\lesssim e^{\varphi(w)}\kappa(z)e^{-\eps d_\kappa(z,w)} \rho(w)^{-n}||u||_\varphi
\ee
 holds for every $w$ such that $d_\kappa(z,w)\geq R_0$.
\end{thm}

A few comments before the proof:\begin{enumerate}
\item The distance $d_\kappa$ and the corresponding metric balls $B_\kappa(z,r)$ associated to $\kappa$ are defined in Section \ref{radius-sec}.
\item The definition of $\mu$-coercivity (Definition \ref{coerc-def}) shows that $\mu$ is dimensionally the inverse of a length, and this is consistent with our requirement that $\kappa=\mu^{-1}$ be a radius function.
\item If $\Box_\varphi$ is $\kappa^{-1}$-coercive for some bounded radius function $\kappa$ that does not satisfy condition (2) of the statement, then we can put $\widetilde{\kappa}:=\kappa\vee\rho$. By Proposition \ref{radius-max} of Section \ref{radius-sec}, $\widetilde{\kappa}$ is a radius function, condition (2) is trivially satisfied, and $\Box_\varphi$ is $\widetilde{\kappa}^{-1}$-coercive, because $\widetilde{\kappa}^{-1}\leq \kappa$.
\end{enumerate}

\begin{proof} By Proposition \ref{radius-prop} of Section \ref{radius-sec} we can find allowable constants $r_0\in(0,1)$ and $R_0\geq 2$ such that \bel\label{exp-inclusion}
B_\kappa(z,r_0)\subseteq B(z,\kappa(z)/2)\subseteq B(z,\kappa(z))\subseteq B_\kappa(z,R_0-1).
\eel
There are also allowable constants $r_1,r_2\in(0,1)$ such that \bel\label{exp-inclusion2} B(w,2r_1\kappa(w))\subseteq B_\kappa(w,r_2)\subseteq B(w,\kappa(w)).\eel 
If $d_\kappa(z,w)\geq R_0$, we have $B_\kappa(z,R_0-1)\cap B_\kappa(w,r_2)=\varnothing$. Since $\kappa(w)\geq\rho_{\text{max}}(w)$, the canonical solution $f$ is holomorphic on $B(w,r_1\rho(w))$. By Lemma \ref{adm-hol}, we have\be
|f(w)|^2e^{-2\varphi(w)}\lesssim \rho(w)^{-2n}\int_{B(w,r_1\rho(w))}|f|^2e^{-2\varphi}.
\ee 
Recall from part (ii) of Proposition \ref{coerc-prop} that $f=\dbar^*_\varphi N_\varphi u$. Since $\Box_\varphi N_\varphi u=u$ vanishes on $B(w,2r_1\kappa(w))$, Lemma \ref{MKH-caccioppoli} yields
\bee
\int_{B(w,r_1\rho(w))}|f|^2e^{-2\varphi}&\leq& \int_{B(w,r_1\kappa(w))}|f|^2e^{-2\varphi}\\
&\lesssim& \kappa(w)^{-2}\int_{B(w,2r_1\kappa(w))}|N_\varphi u|^2e^{-2\varphi}.
\eee
Putting our estimates together, we see that we are left with the task of proving the $L^2$ estimate (with $\eps>0$ admissible):\bel\label{exp-L2}
\kappa(w)^{-2}\int_{B(w,2r_1\kappa(w))}|N_\varphi u|^2e^{-2\varphi} \lesssim \kappa(z)^2e^{-2\eps d_\kappa(z,w)}||u||_\varphi^2.
\eel
Let $\ell:[0,+\infty)\rightarrow[0,1]$ be the continuous function equal to $0$ on $[0,\kappa(z)/2]$, equal to $1$ on $[\kappa(z),+\infty)$, and affine in between. By \eqref{exp-inclusion} and \eqref{exp-inclusion2}, the function $\eta(z'):=\ell(|z'-z|)$ is equal to $0$ on $B(w,2r_1\kappa(w))$, equal to $1$ on $B_\kappa(w,r_2)$, and 
\bel\label{exp-nabla-eta}\sup_{z'\in B(z,\kappa(z))}|\nabla\eta(z')|=\frac{\kappa(z)^{-1}}{2}.\eel 
We also need to define $b(z'):=\min\{d_\kappa(z,z'),d_\kappa(z,w)\}$. We know by Proposition \ref{radius-prop} that $d_\kappa(z,\cdot)$ is Lipschitz, and hence $b$ is also Lipschitz. Moreover, estimate \eqref{radius-nabla} gives \bel\label{exp-nabla-b}
|\nabla b(z')|\lesssim\kappa(z')^{-1},\eel  and $||b||_\infty\leq d_\kappa(z,w)$. From these facts, one may easily conclude that $\eta e^{\eps b}$ is a real-valued bounded Lipschitz function, for any $\eps>0$. By Lemma \ref{MKH-computation}, we obtain
\bee 
\mathcal{E}_\varphi(\eta e^{\eps b}N_\varphi u)&=&\frac{1}{4}\int_{\C^n}|\nabla(\eta e^{\eps b})|^2|N_\varphi u|^2e^{-2\varphi} + \Re(\eta e^{\eps b} u,\eta e^{\eps b} N_\varphi u)_\varphi \\
&\lesssim&\int_{\C^n}|\nabla\eta|^2 e^{2\eps b}|N_\varphi u|^2e^{-2\varphi}+\eps^2\int_{\C^n}\eta^2 e^{2\eps b}|\nabla b|^2|N_\varphi u|^2e^{-2\varphi},
\eee
where we used the fact that $u$ vanishes on the support of $\eta$. By the $\kappa^{-1}$-coercivity of $\Box_\varphi$, \eqref{exp-nabla-eta} and \eqref{exp-nabla-b}, we get\bee
\int_{\C^n}\kappa^{-2}\eta^2 e^{2\eps b}|N_\varphi u|^2e^{-2\varphi}&\lesssim&\kappa(z)^{-2}\int_{B(z,\kappa(z))} e^{2\eps b}|N_\varphi u|^2e^{-2\varphi}\\
&+&\eps^2\int_{\C^n}\eta^2 e^{2\eps b}\kappa^{-2}|N_\varphi u|^2e^{-2\varphi}.
\eee If $\eps\leq\eps_0$, where $\eps_0$ is allowable, recalling that on $B(z,\kappa(z))\subseteq B_\kappa(z,R_0-1)$ we have $e^{2\eps b(z')}\leq e^{2\eps d_\kappa(z,z')}\lesssim 1$, we find
\be\int_{\C^n}\kappa^{-2}\eta^2 e^{2\eps b}|N_\varphi u|^2e^{-2\varphi}\lesssim \kappa(z)^{-2}\int_{B(z,\kappa(z))}|N_\varphi u|^2e^{-2\varphi}.\ee
Notice that:\begin{enumerate}
\item[(a)] $b\geq d_\kappa(z,w)-1$ on $B(w, 2r_1\kappa(w))\subseteq B_\kappa(w,r_2)$, and that $\eta\equiv1$ on this ball,
\item[(b)] $\kappa^{-2}\gtrsim \kappa(w)^{-2}$ on $B(w,\kappa(w))$, and hence on $B(w, 2r_1\kappa(w))$, because $\kappa$ is a radius function,
\item[(c)] $\kappa(z)^{-2}\lesssim \kappa^{-2}$ on $B(z,\kappa(z))$, again because $\kappa$ is a radius function.
\end{enumerate}
For $\eps>0$ allowable, we then have
\bel\label{exp-1}
\kappa(w)^{-2}\int_{B(w, 2r_1\kappa(w))}|N_\varphi u|^2e^{-2\varphi}\lesssim e^{-2\eps d_\kappa(z,w)}\int_{\C^n}\kappa^{-2}|N_\varphi u|^2e^{-2\varphi}.
\eel
By part (i) of Proposition \ref{coerc-prop} and the fact that $u$ is supported where $\kappa\lesssim \kappa(z)$, we have\bel\label{exp-2}
\int_{\C^n}\kappa^{-2}|N_\varphi u|^2e^{-2\varphi}\lesssim \kappa(z)^2||u||_\varphi^2.
\eel
Putting \eqref{exp-1} and \eqref{exp-2} together we finally obtain \eqref{exp-L2} and hence the thesis.
\end{proof}

\section{Pointwise bounds for weighted Bergman kernels}\label{bergman-sec}
 
 To prove the pointwise bounds for weighted Bergman kernels we use a technique introduced in \cite{kerzman}, and adapted to the weighted case in \cite{delin}. For the sake of completeness, we state as a lemma the relevant part of \cite{delin} and recall its proof. 
 
 We continue working under the assumptions of Theorem \ref{exp-thm}, that is $\varphi$ is an admissible weight and $\kappa$ is a bounded radius function such that $\kappa\geq \rho$ and $\Box_\varphi$ is $\kappa^{-1}$-coercive.
 
 Let $\eta$ be a radial test function supported on the unit ball of $\C^n$ such that $\int_{\C^n}\eta=1$, and put \be\eta_z(w):=\frac{1}{(\delta\rho_{\text{max}}(z))^{2n}}\eta\left(\frac{w-z}{\delta\rho(z)}\right),\ee where $\delta>0$ is an allowable constant chosen so that the support of $\eta_z$, i.e., $B(z,\delta\rho(z))$, is contained in $B_\kappa(z,r_0)$, with $r_0$ as in Theorem \ref{exp-thm} (this is possible by Proposition \ref{radius-prop} of Section \ref{radius-sec}).

\begin{lem}\label{bergman-fz}
For every $z\in\C^n$ there exists a holomorphic function $H_z$ defined on $B(z,\rho(z))$ that vanishes in $z$ and such that\be
f_z:=\eta_z e^{\overline{H_z}+2\varphi}\in \mathcal{D}_0(\dbar).
\ee Moreover, we have the following inequalities\be
||f_z||_\varphi\lesssim e^{\varphi(z)}\rho(z)^{-n},\quad ||\dbar f_z||_\varphi\lesssim e^{\varphi(z)}\rho(z)^{-n-1}.
\ee
\end{lem}

\begin{proof} Let $\psi$ be the function given by Lemma \ref{adm-gauge} and $F$ the holomorphic function on $B(z,\rho(z))$ such that $\psi-\varphi=\Re(F)$. We define $H_z(w):=F(w)-F(z)$. Lemma \ref{adm-gauge} also gives the bound\bel\label{bergman-gauge}
||\varphi+\Re(F)||_\infty\lesssim1.
\eel
Let us check that $f_z:=\eta_z e^{\overline{H_z}+\varphi}$ verifies the inequalities of the statement. First of all,\bee
||f_z||_\varphi^2&\lesssim&\rho(z)^{-4n}\int_{B(z,\delta\rho(z))}|e^{H_z+2\varphi}|^2e^{-2\varphi}\\
&=&\rho(z)^{-4n}e^{-2\Re(F(z))}\int_{B(z,\delta\rho(z))}e^{2\Re(F)+2\varphi}\\
&\lesssim &\rho(z)^{-4n}e^{2\varphi(z)}\rho(z)^{2n}=\rho(z)^{-2n}e^{2\varphi(z)}, 
\eee where in the third line we used \eqref{bergman-gauge}. This proves the bound on $||f_z||_\varphi$. 

Next, we compute (using again \eqref{bergman-gauge})\bee
||\dbar f_z||_\varphi^2&=&\int_{\C^n}|\dbar f_z|^2e^{-2\varphi}\\
&=&e^{-2\Re(F(z))}\int_{\C^n}|\dbar \eta_z+\eta_z\dbar(\overline{F}+2\varphi)|^2e^{2\Re(F)+2\varphi}\\
&\lesssim& e^{-2\varphi(z)}\int_{\C^n}|\dbar \eta_z|^2+e^{-2\varphi(z)}\int_{\C^n}\eta_z^2|\dbar(2\Re(F)+2\varphi)|^2,
\eee where in the last term we used the fact that $\dbar F=0$. The key observation is that, since $\Delta(\Re(F)+\varphi)=\Delta\varphi\geq0$,\bee
\Delta(e^{2\Re(F)+2\varphi})e^{-2\Re(F)-2\varphi}&=&\Delta(2\Re(F)+2\varphi)+|\nabla(2\Re(F)+2\varphi)|^2\\
&\geq&4|\dbar(2\Re(F)+2\varphi)|^2.
\eee Coming back to our estimate, we have\bee
e^{-2\varphi(z)}\int_{\C^n}\eta_z^2|\dbar(2\Re(F)+2\varphi)|^2&\lesssim&e^{-2\varphi(z)}\int_{\C^n}\eta_z^2\Delta(e^{2\Re(F)+2\varphi})e^{-2\Re(F)-2\varphi}\\
&\lesssim &e^{-2\varphi(z)}\int_{\C^n}\Delta(\eta_z^2)e^{2\Re(F)+2\varphi}\lesssim e^{-2\varphi(z)}\int_{\C^n}\Delta(\eta_z^2), 
\eee where we used an integration by parts and \eqref{bergman-gauge}. Since it is easily seen that \be
\int_{\C^n}|\dbar \eta_z|^2+\int_{\C^n}\Delta(\eta_z^2)\lesssim \rho(z)^{-2n-2}, 
\ee the estimates of the statement are proved.\end{proof}

\begin{thm}\label{bergman-thm}
Under the assumptions of Theorem \ref{exp-thm}, there is an allowable constant $\eps>0$ such that the pointwise bound
\be
|B_\varphi(z,w)|\lesssim e^{\varphi(z)+\varphi(w)}\frac{\kappa(z)}{\rho(z)}\frac{e^{-\eps d_\kappa(z,w)}}{\rho(z)^{n}\rho(w)^{n}}
\ee 
holds for every $z,w\in\C^n$.
\end{thm}

\begin{proof}
For $z\in\C^n$ let $f_z$ be as in Lemma \ref{bergman-fz} and notice that \bee
B_\varphi(f_z)(w)&=&\int_{\C^n}B_\varphi(w,w')f_z(w')e^{-2\varphi(w')}d\mathcal{L}(w')\\
&=&\int_{\C^n}B_\varphi(w,w')\eta_z(w') e^{\overline{H_z}(w')}d\mathcal{L}(w')\\
&=& B_\varphi(w,z)e^{\overline{H_w}(w)}=\overline{B_\varphi(z,w)},
\eee where in the last line we used the fact that $\eta_w$ is radial with respect to $w\in\C^n$, $\int_{\C^n}\eta_w=1$, $B_\varphi(z,\cdot)e^{\overline{H_w}}$ is harmonic, being the product of two anti-holomorphic functions, and $H_w(w)=0$. Hence, by formula \eqref{coerc-bergman} of Proposition \ref{coerc-prop}, we have\bel\label{bergman-bar}
\overline{B_\varphi(z,w)}=f_z(w)-\dbar^*_\varphi N_\varphi \dbar f_z (w).
\eel
Since $\overline{B_\varphi(z,\cdot)}$ is holomorphic, Lemma \ref{adm-hol} yields \be
|B_\varphi(z,w)|\lesssim e^{\frac{\varphi(w)}{2}}\rho(w)^{-n}||B_\varphi(z,\cdot)||_\varphi.
\ee Thanks to \eqref{bergman-bar} and inequality \eqref{coerc-canonical-bound} of Proposition \ref{coerc-prop}, we have
\bee
||B_\varphi(z,\cdot)||_\varphi&\leq&||f_z||_\varphi+||\dbar^*_\varphi N_\varphi \dbar f_z||_\varphi\\
&\lesssim& ||f_z||_\varphi+||\kappa\dbar f_z||_\varphi\\ 
&\lesssim& ||f_z||_\varphi+\left(\max_{B(z,\rho(z))}\kappa\right)||\dbar f_z||_\varphi.
\eee Now recall that $\kappa\geq\rho$ and hence that $\kappa$, being a radius function, is $\lesssim \kappa(z)$ on $B(z,\rho(z))$. Lemma \ref{bergman-fz} finally gives\be
||B_\varphi(z,\cdot)||_\varphi\lesssim  \frac{\kappa(z)}{\rho(z)}e^{\varphi(z)}\rho(z)^{-n}.
\ee
What we obtained until now is \be
|B_\varphi(z,w)|\lesssim e^{\varphi(z)+\varphi(w)}\frac{\kappa(z)}{\rho(z)}\rho(z)^{-n}\rho(w)^{-n}.
\ee
This is equivalent to the conclusion of the theorem if $d_\kappa(z,w)\lesssim 1$. We can then assume from now on that $d_\kappa(z,w)\geq R_0$, with $R_0$ the allowable constant in Theorem \ref{exp-thm}, which then implies \be
|\dbar^*_\varphi N_\varphi \dbar f_z (w)|\lesssim e^{\varphi(w)}\kappa(z)e^{-\eps d_\kappa(z,w)} \rho(w)^{-n}||\dbar f_z||_\varphi.
\ee
We conclude by Lemma \ref{bergman-fz} and the identity $\overline{B_\varphi(z,w)}=-\dbar^*_\varphi N_\varphi \dbar f_z (w)$ (which holds for $d_\kappa(z,w)\geq R_0$).
\end{proof}

We state now as a separate theorem our bound on the weighted Bergman kernel when the eigenvalues of $H_\varphi$ are comparable.

\begin{thm}\label{bergman-comp}
Let $\varphi:\C^n\rightarrow\R$ be $C^2$, plurisubharmonic and such that:
\begin{enumerate}
\item[\emph{(i)}] there exists $c>0$ such that 
\be
\inf_{z\in\C^n}\sup_{w\in B(z,c)}\Delta\varphi(w)>0,
\ee
\item[\emph{(ii)}] $\Delta\varphi$ satisfies the reverse-H\"older inequality
\bel\label{bergman-RH}
||\Delta\varphi||_{L^\infty(B(z,r))}\leq Ar^{-2n}\int_{B(z,r)}\Delta\varphi\qquad\forall z\in\C^n, r>0,
\eel for some $A<+\infty$,
\item[\emph{(iii)}] the eigenvalues of $H_\varphi$ are comparable, i.e., \eqref{mucomparable-ineq} holds.
\end{enumerate}
Then there is an allowable constant $\eps>0$ such that the pointwise bound
\be
|B_\varphi(z,w)|\lesssim e^{\varphi(z)+\varphi(w)}\frac{e^{-\eps d(z,w)}}{\rho(z)^{n}\rho(w)^{n}}
\ee 
holds for every $z,w\in\C^n$, where $d$ is the maximal eigenvalue distance associated to $\varphi$ (see Section \ref{adm-sec}).
\end{thm}

\begin{proof}
The reverse-H\"older inequality \eqref{bergman-RH} implies that the measure with density $\Delta\varphi$ with respect to Lebesgue measure is doubling, i.e.,\be
\int_{B(z,2r)}\Delta\varphi\lesssim\int_{B(z,r)}\Delta\varphi\qquad\forall z\in\C^n,\ r>0.
\ee
This, together with the reverse-H\"older inequality itself, implies condition (1) in Definition \ref{adm-dfn}. Since condition (2) of that definition is our hypothesis (i), the weight $\varphi$ is admissible.

By Lemma \ref{mucomparable-lem}, $\Box_\varphi$ is $c\rho^{-1}$-coercive, where $c>0$ is admissible. An application of Theorem \ref{bergman-thm} with $k=c^{-1}\rho$ gives the thesis.
\end{proof}

\chapter{Model weights in $\C^2$}\label{C2-ch}

In this chapter we discuss, by a combination of explicit computations and more conceptual arguments, $\mu$-coercivity of weighted Kohn Laplacians for \emph{homogeneous model weights}. These are introduced in Section 3.1, where our main result (Theorem \ref{model-thm}) is also stated. Sections 3.2 through 3.6 contain the corollaries and the proof of Theorem \ref{model-thm}.

\section{Homogeneous model weights in $\C^2$}\label{model-sec}

Let us define the class of model weights.

\begin{dfn}\label{model-dfn}
If $\Gamma\subseteq \N^2$ is finite, we define the model weight associated to $\Gamma$ as follows:\be
\varphi_\Gamma(z,w):=\sum_{(\alpha,\beta)\in\Gamma}|z^\alpha w^\beta|^2\qquad\forall (z,w)\in\C^2.
\ee
\end{dfn}

Of course one could consider the analogous definition in $\C^n$, associating a model weight to any finite $\Gamma\subseteq \N^n$, but here we shall only treat only the two-dimensional case.

\begin{prop}\label{model-prop}
If $\Gamma\subseteq\N^2$, the model weight $\varphi_\Gamma$ is admissible in the sense of Definition \ref{adm-dfn}.
\end{prop}

\begin{proof}
Being a sum of squares of holomorphic functions, $\varphi_\Gamma$ is $C^2$ and plurisubharmonic (alternatively, this follows from Proposition \ref{model-formulas}). 

Conditions (1) and (2) of Definition \ref{adm-dfn} only depend on the fact that $\Delta\varphi$ is a non-negative polynomial on $\C^n\equiv\R^{2n}$. Let $d\in\N$ be the degree of this polynomial. The mappings\be
p\mapsto \sup_{B(0,1)}|p|\quad\text{and}\quad p\mapsto \sup_{B(0,2)}|p|
\ee 
are norms on the finite-dimensional vector space of real polynomials in $2n$ real variables of degree $\leq d$ on $\C^n\equiv\R^{2n}$, and therefore they are equivalent. In particular \bee
\sup_{w\in B(z,2r)}\Delta\varphi(w)&=&\sup_{w\in B(0,2)}\Delta\varphi(z+rw)\\
&\leq& D\sup_{w\in B(0,1)}\Delta\varphi(z+rw)=D\sup_{w\in B(z,r)}\Delta\varphi(w).
\eee This proves condition (1).

As $z$ varies in $\C^n$, the polynomial $\Delta\varphi(z+\cdot)$ varies on a hyperplane not containing the origin of the vector space of real polynomials in $2n$ real variables of degree $\leq d$. To see this, just notice that any of the coefficients of a monomial of highest degree of $\Delta\varphi$ is not affected by translations. Since $p\mapsto \sup_{B(0,1)}|p|$ is a norm, we have\be
\inf_{z\in\C^n}\sup_{B(z,1)}\Delta\varphi=\inf_{z\in\C^n}\sup_{B(0,1)}\Delta\varphi(z+\cdot)>0,
\ee that is condition (2).\end{proof}

Since $\varphi_\Gamma(z,w)=\sum_{(\alpha,\beta)\in\Gamma}|z|^{2\alpha} |w|^{2\beta}$, model weights only depend on the squared moduli of the coordinates. In view of this, we introduce the polynomial \bel\label{model-pol}
p_\Gamma(x,y):=\sum_{(\alpha,\beta)\in\Gamma}x^\alpha y^\beta\qquad (x,y)\in \R_+^2,
\eel and in what follows we think of $x$ and $y$ both as independent variables and as denoting $|z|^2$ and $|w|^2$ respectively, so that $\varphi_\Gamma(z,w)=p_\Gamma(|z|^2,|w|^2)=p_\Gamma(x,y)$. This ambiguity will not be a source of confusion.

We now prove a very useful formula for the determinant and the trace of the complex Hessian $H_{\varphi_\Gamma}$ of a model weight. In order to state it, we associate to any $\Gamma\subseteq \N^2$ four further subsets of $\N^2$:
\bee
\Gamma_r&:=&\{(\alpha,\beta)\in\Gamma\colon\ \alpha\neq0\} \quad\text{ ($r$ stands for ``right")},\\
\Gamma_u&:=&\{(\alpha,\beta)\in\Gamma\colon\ \beta\neq0\} \quad\text{($u$ stands for ``upper")},\\
\Gamma^{(1)}&:=&\{(\alpha,\beta)+(\gamma,\delta)\colon\ (\alpha,\beta), (\gamma,\delta)\in\Gamma \text{ linearly independent}\}-(1,1),\\
\Gamma^{(2)}&:=&\left[\Gamma_r-(1,0)\right]\cup\left[\Gamma_u-(0,1)\right].
\eee
Here $\Gamma_r-(1,0)$ denotes the collection $\{(\alpha-1,\beta):\ (\alpha,\beta)\in\Gamma_r\}$, and the other symbols have analogous meanings. Observe that if $(\alpha,\beta)$ and $(\gamma,\delta)$ are linearly independent elements of $\N^2$, then $(\alpha+\gamma-1,\beta+\delta-1)\in\N^2$, and hence $\Gamma^{(1)}\subseteq \N^2$.

\begin{prop}\label{model-formulas} If $\Gamma\subseteq \N^2$ is finite, then
\bee
\text{det}(H_{\varphi_\Gamma}(z,w))&\approx& \varphi_{\Gamma^{(1)}}(z,w),\\
\text{tr}(H_{\varphi_\Gamma}(z,w))&\approx& \varphi_{\Gamma^{(2)}}(z,w),
\eee where the implicit constants depend only on $\Gamma$.
\end{prop}

In particular, this proposition shows that model weights are weakly plurisubharmonic on the set where $\varphi_{\Gamma^{(1)}}$ vanishes. Since $\varphi_{\Gamma^{(1)}}$ is a model weight, this set may be easily determined from $\Gamma^{(1)}$, and may be empty, the origin, a complex coordinate axes ($\{z=0\}$ or $\{w=0\}$), or $\{z=0\}\cup\{w=0\}$. We omit the elementary details.

\begin{proof}
Let $h_1,\cdots,h_N:\C^2\rightarrow \C$ be holomorphic functions and consider the weight \be
\varphi:=\sum_{j=1}^N |h_j|^2.
\ee We have\be
\partial_z\dbar_z\varphi = \sum_j |\partial_zh_j|^2,\quad \partial_w\dbar_w\varphi = \sum_j |\partial_wh_j|^2,
\ee and \be
\partial_z\dbar_w\varphi = \sum_j \partial_zh_j\overline{\partial_w h_j},\quad \partial_w\dbar_z\varphi = \sum_j \partial_wh_j\overline{\partial_z h_j}.
\ee Hence\bee
\text{det}(H_\varphi) &=& \partial_z\dbar_z\varphi \cdot\partial_w\dbar_w\varphi - \partial_z\dbar_w\varphi\cdot\partial_w\dbar_z\varphi\\
&=& \sum_{j,k} |\partial_zh_j|^2|\partial_wh_k|^2 - \sum_{j,k} \partial_zh_j\overline{\partial_w h_j}\partial_wh_k\overline{\partial_z h_k}\\
&=& \frac{1}{2}\left(\sum_{j,k} |\partial_zh_j|^2|\partial_wh_k|^2 + |\partial_wh_j|^2|\partial_zh_k|^2 - 2\Re(\partial_zh_j\overline{\partial_w h_j}\partial_wh_k\overline{\partial_z h_k})\right)\\
&=& \frac{1}{2}\sum_{j,k} |\partial_zh_j\partial_wh_k - \partial_wh_j\partial_zh_k|^2.
\eee We also have\be
\text{tr}(H_\varphi)= \partial_z\dbar_z\varphi + \partial_w\dbar_w\varphi =\sum_j |\partial_zh_j|^2+|\partial_wh_j|^2.
\ee
Specializing to $\varphi_\Gamma(z,w):=\sum_{(\alpha,\beta)\in\Gamma}|z^\alpha w^\beta|^2$, we obtain (here l.i. stands for ``linearly independent"):
\bee
\text{det}(H_{\varphi_\Gamma}(z,w))&=&\frac{1}{2}\sum_{(\alpha,\beta),(\gamma,\delta)\in\Gamma}(\alpha\delta-\beta\gamma)^2|z^{\alpha+\gamma-1}w^{\beta+\delta-1}|^2\\
&\approx& \sum_{(\alpha,\beta),(\gamma,\delta)\in\Gamma \text{ l. i.}}|z^{\alpha+\gamma-1}w^{\beta+\delta-1}|^2\\
&\approx&\varphi_{\Gamma^{(1)}}(z,w),
\eee and \bee
\text{tr}(H_{\varphi_\Gamma}(z,w))&=&\sum_{(\alpha,\beta)\in\Gamma}\alpha^2|z^{\alpha-1}w^\beta|^2+\beta^2|z^\alpha w^{\beta-1}|^2\\
&\approx& \sum_{(\alpha,\beta)\in\Gamma\colon\alpha\neq0}|z^{\alpha-1} w^\beta|^2+\sum_{(\alpha,\beta)\in\Gamma\colon\beta\neq0}|z^\alpha w^{\beta-1}|^2\\
&\approx&\varphi_{\Gamma^{(2)}}(z,w).
\eee
It is easy to see that the implicit constants in the approximate equalities above depend only on $\Gamma$.
\end{proof}

It is time to give the main definition of this section.

\begin{dfn}
A homogeneous model weight is a model weight $\varphi_\Gamma$ such that there are $m,n\geq1$ for which the following two properties hold: \begin{enumerate}
\item $\{(m,0), (0,n)\} \subseteq\Gamma$,
\item every $(\alpha,\beta)\in\Gamma$ lies on the line segment connecting $(m,0)$ and $(0,n)$, i.e.\bel\label{model-segment}
n\alpha+m\beta=nm\qquad\forall (\alpha,\beta)\in\Gamma.
\eel
\end{enumerate}
A homogeneous model weight is said to be \emph{decoupled} if $\Gamma=\{(m,0),(0,n)\}$. 
\end{dfn}

Any homogeneous model weight is homogeneous with respect to a system of non necessarily isotropic dilations, i.e.,
\be\varphi_\Gamma(t^\frac{1}{m} z, t^\frac{1}{n} w)=t^2\varphi_\Gamma(z,w)\qquad \forall t>0 \text{ and } (z,w)\in\C^2.\ee

If $\varphi_\Gamma$ is not decoupled, it contains at least a mixed monomial and hence homogeneity forces $m$ and $n$ to be both greater than or equal to $2$.

In our analysis of homogeneous model weights, a key role will be played by the two quantities $\sigma$ and $\tau$, which are defined as the smallest non-negative real numbers such that \bel\label{model-sigma-tau}
\frac{1}{\sigma}\leq \frac{\beta}{\alpha}\leq \tau\qquad\forall (\alpha,\beta)\in\Gamma\text{ such that }\alpha,\beta\neq0.
\eel
We denote by $(\alpha_1,\beta_1)$ and $(\alpha_2,\beta_2)$ the unique points of $\Gamma$ such that 
\be \alpha_1/\beta_1=\sigma\quad \text{and}\quad  \beta_2/\alpha_2=\tau.\ee 
Notice that $\sigma,\tau<+\infty$. If $\varphi_\Gamma$ is decoupled, then $\sigma$ and $\tau$ equal $0$, because in that case $\Gamma_r=\{(m,0)\}$ and $\Gamma_u=\{(0,n)\}$, while if $\varphi_\Gamma$ is not decoupled both $\sigma$ and $\tau$ are positive. 

\begin{center}

\begin{tikzpicture} 

\tiny

\begin{axis}[
axis x line=bottom, axis y line=left, 
xmin=0, xmax=17, ymin=0, ymax=17, 
xtick={0,4,8,12,16, 17}, ytick={0,3,6,9,12, 17},
xticklabels={0,4,8,12,16,$\alpha$}, yticklabels={0,3,6,9,12, $\beta$}, 
enlargelimits=false] 

\addplot coordinates {(0,12)}
[xshift=12pt, yshift=8pt]
        node {(0,12)}
;

\addplot coordinates {(4,9)}
[yshift=8pt]
        node {$(\alpha_2,\beta_2)$=(4,9)}
;
\addplot coordinates {(8,6)}
[xshift=12pt, yshift=8pt]
        node {(8,6)}
;
\addplot coordinates {(12,3)}
[yshift=8pt]
        node {$(\alpha_1,\beta_1)$=(12,3)}
;
\addplot coordinates {(16,0)}
[xshift=-12pt, yshift=8pt]
        node {(16,0)}
;
\addplot coordinates{(0,12)(4,9)(8,6)(12,3)(16,0)};
;
;

\end{axis}

\end{tikzpicture}
\end{center}

\begin{small}
Figure 1: The plot of the set $\Gamma$ corresponding to the homogeneous model weight $\varphi_\Gamma(z,w)=|z|^{32}+|z|^{24}|w|^6+|z|^{16}|w|^{12}+|z|^8|w|^{18}+|w|^{24}$. In this case $\sigma=4$ and $\tau=\frac{9}{4}$.
\end{small}

\ \newline

Let us state our main result about homogeneous model weights.

\begin{thm}\label{model-thm}
Let $\varphi_\Gamma$ be a homogeneous model weight. 

Then $\Box_{\varphi_\Gamma}$ is $\mu$-coercive, where \be
\mu(z,w)=c(1+|z|^\sigma+|w|^\tau).
\ee Here $c>0$ is a constant that depends on $\Gamma$.
\end{thm}

\section{Corollaries of Theorem \ref{model-thm}}

We can apply Theorem \ref{exp-thm} and Theorem \ref{bergman-thm} to $\varphi_\Gamma$ choosing $\kappa(z,w)=c^{-1}(1+|z|^\sigma+|w|^\tau)^{-1}$, because it is easily seen that $\kappa$ is approximately constant on balls of radius $1$, and hence a radius function.

Theorem \ref{model-thm} also implies the following result about discreteness of the spectrum of Kohn Laplacians.

\begin{thm}\label{model-disc}
If $\varphi_\Gamma$ is a homogeneous model weight, then the Kohn Laplacian $\Box_\varphi$ has discrete spectrum if and only if $\varphi_\Gamma$ is not decoupled.
\end{thm}

\begin{proof}
If $\varphi_\Gamma$ is not decoupled, then $\sigma$ and $\tau$ are strictly positive and Theorem \ref{model-thm} immediately shows that the condition of Theorem \ref{coerc-pp-thm} holds. 

If, on the contrary, $\varphi_\Gamma$ is decoupled, the non-discreteness of spectrum is proved in \cite{haslinger-helffer}.
\end{proof}

\section{The proof of Theorem \ref{model-thm}}\label{modelstruct-sec}

The proof is organized as follows:\begin{enumerate}
\item In Section \ref{lambda-sec} we carefully estimate the minimal eigenvalue 
\be\lambda_\Gamma:=\lambda(H_{\varphi_\Gamma}) \ee
of the complex Hessian of $\varphi_\Gamma$, taking advantage of the special formulas of Proposition \ref{model-formulas}.
\item In Section \ref{hol-sec} we introduce a \emph{holomorphic uncertainty principle} to take care of the regions where $\lambda_\Gamma$ is small.
\item Finally, in Section \ref{est-sec} we estimate from below the energy $\mathcal{E}_{\varphi_\Gamma}$: outside of a hyperbolic neighborhood of the complex coordinate axes we use the estimate of Section \ref{lambda-sec}, while on this neighborhood we exploit the holomorphic uncertainty principle introduced in Section \ref{hol-sec}.
\end{enumerate}

\section{Estimating $\lambda_\Gamma$}\label{lambda-sec}

If $(u,v)\in\R^2$, we consider the curve in the non-negative quadrant $[0,+\infty)^2$\be
C_{u,v}:\quad t\longmapsto (t^u,t^v)\qquad(t\geq1).
\ee
Notice that if $(u',v')$ is proportional to $(u,v)$, $C_{u,v}$ and $C_{u',v'}$ have the same range. If $A\subseteq \N^2$ is finite, using the notation \eqref{model-pol} we have\be
p_A(C_{u,v}(t))=\sum_{(\alpha,\beta)\in A} (t^u)^\alpha(t^v)^\beta=\sum_{(\alpha,\beta)\in A} t^{u\alpha+v\beta}\approx t^{m_{u,v}(A)}\quad(t\geq1),
\ee
where $m_{u,v}(A)$ is the maximum of the linear functional $(\xi,\eta)\mapsto u\xi+v\eta$ on the set $A\subseteq \R^2$, and the implicit constant depends on $\Gamma$ and is independent of $u,v$. 

We are interested in estimating $\lambda_\Gamma(z,w)$ when $(x,y)=(|z|^2,|w|^2)$ lies on the curve $C_{u,v}$. Since the trace of a non-negative matrix is comparable to its maximal eigenvalue we have \be
\det(H_{\varphi_\Gamma}(z,w))\approx\lambda_\Gamma(z,w) \text{tr}(H_{\varphi_\Gamma}(z,w)).
\ee 
Proposition \ref{model-formulas} yields the approximate identity\be
p_{\Gamma^{(1)}}(|z|^2,|w|^2)=\varphi_{\Gamma^{(1)}}(z,w)\approx\lambda_\Gamma(z,w) \varphi_{\Gamma^{(2)}}(z,w)=\lambda_\Gamma(z,w) p_{\Gamma^{(2)}}(|z|^2,|w|^2).
\ee
If $(|z|^2,|w|^2)=(t^u,t^v)=C_{u,v}(t)$ ($t\geq1$), the discussion at the beginning of this section gives\bel\label{lambda-zwt}
\lambda_\Gamma(z,w) \approx t^{m_{u,v}(\Gamma^{(1)})-m_{u,v}(\Gamma^{(2)})}.
\eel

Observe that for homogeneous model weights the definitions of $\Gamma^{(1)}$ and $\Gamma^{(2)}$ take the slightly simpler forms\bee
\Gamma^{(1)}&:=&\{(\alpha+\gamma,\beta+\delta):\ (\alpha,\beta)\neq (\gamma,\delta)\in\Gamma\}-(1,1),\\
\Gamma^{(2)}&:=&\left(\Gamma\setminus \{(0,n)\}-(1,0)\right)\cup\left(\Gamma\setminus \{(m,0)\}-(0,1)\right).
\eee
If $(u,v)\in\R^2$ is fixed, by convexity considerations it is clear that the maximum of $u\xi+v\eta$ on $\Gamma$ is attained at $(m,0)$ if $um\geq vn$, while it is attained at $(0,n)$ if $um\leq vn$. We separately analyze the two cases, assuming without loss of generality that $m\geq n$.

\subsection*{Case I: $v\leq\frac{m}{n}u$} We have\be
m_{u,v}(\Gamma^{(1)})=um+m_{u,v}(\Gamma\setminus \{(m,0)\})-u-v,
\ee and 
\bee
m_{u,v}(\Gamma^{(2)})&=&\max\{m_{u,v}(\Gamma\setminus \{(0,n)\})-u,m_{u,v}(\Gamma\setminus \{(m,0)\})-v\}\\
&=&\max\{um-u,m_{u,v}(\Gamma\setminus \{(m,0)\})-v\}.
\eee
Hence \bel\label{lambda-max-formula}
m_{u,v}(\Gamma^{(1)})-m_{u,v}(\Gamma^{(2)})=\min\{m_{u,v}(\Gamma\setminus \{(m,0)\})-v, u(m-1)\}.
\eel
It is almost immediate to see that the maximum of $u\xi+v\eta$ on $\Gamma\setminus \{(m,0)\}$ is attained at the point $(\alpha_1,\beta_1)$ satisfying $\frac{\alpha_1}{\beta_1}=\sigma$, which we introduced above. Identity \eqref{lambda-max-formula} becomes\be
m_{u,v}(\Gamma^{(1)})-m_{u,v}(\Gamma^{(2)})=\min\{u\alpha_1+v\beta_1-v, u(m-1)\}.
\ee

The inequality $u\alpha_1+v\beta_1-v\leq u(m-1)$ holds if and only if\bel\label{lambda-condition-m-n}
v\leq \frac{m-1-\alpha_1}{\beta_1-1}u.
\eel This condition depends only on the ratio of $u$ and $v$, as it should. Observe that \be
\frac{m}{n}\leq\frac{m-1-\alpha_1}{\beta_1-1}.
\ee In fact, the inequality above is obviously equivalent to $mn-n\geq\alpha_1n+m\beta_1-m$, and recalling the homogeneity condition \eqref{model-segment} we see that this is the same as $m\geq n$, which we assumed before. This shows that condition \eqref{lambda-condition-m-n} is a consequence of $v\leq\frac{m}{n}u$ and thus \bel\label{lambda-caseI}
m_{u,v}(\Gamma^{(1)})-m_{u,v}(\Gamma^{(2)})=u\alpha_1+v\beta_1-v.
\eel

\subsection*{Case II: $u\leq\frac{n}{m}v$} Proceeding analogously to the case $um\geq vn$, this time formula \eqref{lambda-max-formula} is replaced by \be
m_{u,v}(\Gamma^{(1)})-m_{u,v}(\Gamma^{(2)})=\min\{m_{u,v}(\Gamma\setminus \{(0,n)\})-u, v(n-1)\},
\ee and the maximum of $ux+vy$ on $\Gamma\setminus \{(0,n)\}$ is attained at the point $(\alpha_2,\beta_2)$ satisfying $\frac{\beta_2}{\alpha_2}=\tau$. Hence\be
m_{u,v}(\Gamma^{(1)})-m_{u,v}(\Gamma^{(2)})=\min\{u\alpha_2+v\beta_2-u, v(n-1)\}.
\ee Here comes the difference with Case I: the minimum above equals $u\alpha_2+v\beta_2-u$ if and only if \be
u\leq\frac{n-1-\beta_2}{\alpha_2-1}v,
\ee but this condition is not automatically implied by the inequality $u\leq\frac{n}{m}v$. In fact \be
\frac{n-1-\beta_2}{\alpha_2-1}\leq \frac{n}{m},
\ee as may be easily verified using \eqref{model-segment} and the fact that $n\leq m$. Thus there are two further sub-cases: if \be
u\leq\frac{n-1-\beta_2}{\alpha_2-1}v
\ee then \bel\label{lambda-caseIIa}
m_{u,v}(\Gamma^{(1)})-m_{u,v}(\Gamma^{(2)})=u\alpha_2+v\beta_2-u, 
\eel while if \be
\frac{n-1-\beta_2}{\alpha_2-1}v\leq u\leq \frac{n}{m}v
\ee then \bel\label{lambda-caseIIb}
m_{u,v}(\Gamma^{(1)})-m_{u,v}(\Gamma^{(2)})=v(n-1). 
\eel

Let us define the three regions of $\C^2$:\bee
E_1&:=&\{|z|\geq 1,\quad |w|\leq |z|^\frac{m}{n}\},\\
E_2&:=&\{|w|\geq 1,\quad |w|^\nu\leq |z|\leq |w|^\frac{n}{m}\},
\eee and \be
E_3:=\{|w|\geq 1,\quad |z|\leq |w|^\nu\},
\ee where $\nu:=\frac{n-1-\beta_2}{\alpha_2-1}$. The figure below depicts these regions with respect to the coordinates $(x,y)$.

\begin{center}

\begin{tikzpicture} 
\small
\begin{axis}[
axis x line=bottom, axis y line=left, 
xmin=0, xmax=3, ymin=0, ymax=3, 
xtick={1,3},
ytick={1,3},
xticklabels={1,$x=|z|^2$},
yticklabels={1,$y=|w|^2$},
enlargelimits=false] 

\addplot[dotted, domain=1:3, samples=100]{x^3}
[xshift= 100pt,yshift=120pt]
node[pos=1]{$E_2$};
\addlegendentry{$x=y^\nu$}

\addplot[dashed, domain=1:3, samples=100]{x^1.3}
[xshift= 140pt,yshift=60pt]
node[pos=1]{$E_1$};
\addlegendentry{$x=y^\frac{n}{m}$}

\addplot[] coordinates {(1,0)(1,1)}
[xshift= 35pt,yshift=100pt]
node[pos=1]{$E_3$};

\addplot[] coordinates {(0,1)(1,1)}
[xshift= 32pt,yshift=26pt];

\end{axis}

\end{tikzpicture}

\end{center}

Recalling \eqref{lambda-zwt}, \eqref{lambda-caseI}, \eqref{lambda-caseIIa}, and \eqref{lambda-caseIIb}, we can summarize our computations in the following proposition.

\begin{prop}\label{lambda-prop} The following approximate identities hold:
\bee
\lambda_\Gamma(z,w)&\approx& |z|^{2\alpha_1}|w|^{2(\beta_1-1)}\qquad\forall (z,w)\in E_1,\\
\lambda_\Gamma(z,w)&\approx& |w|^{2(n-1)}\qquad\qquad\quad\forall (z,w)\in E_2,\\
\lambda_\Gamma(z,w)&\approx& |z|^{2(\alpha_2-1)}|w|^{2\beta_2}\qquad\forall (z,w)\in E_3.
\eee
\end{prop}

We completed the task of estimating $\lambda_\Gamma$, and now we can discuss our holomorphic uncertainty principle.

\section{A holomorphic uncertainty principle}\label{hol-sec}

In the next lemma, $D(z,r)$ denotes the disc of center $z\in\C$ and radius $r$.

\begin{lem}\label{hol-unc} Let $V:D(z,r)\rightarrow[0,+\infty)$ be a measurable function and define \be
c:=\min_{z'\in D(z,r)\setminus D\left(z,\frac{r}{2}\right)}V(z'). \ee
If $f\in L^2(D(z,r))$ is such that $\frac{\partial f}{\partial \overline{z}}\in L^2(D(z,r))$, then \bel\label{hol-ineq}
\int_{D(z,r)}\left|\frac{\partial f}{\partial \overline{z}}\right|^2+\int_{D(z,r)}V|f|^2\gtrsim \min\left\{c,\frac{1}{r^2}\right\}\int_{D(z,r)}|f|^2.
\eel
\end{lem}

The proof is based on a Poincar\'e-type inequality related to the $\frac{\partial}{\partial \overline{z}}$ operator and an elementary consequence of the Cauchy formula, which we now discuss. 

Put $D:=D(0,1)$. It is well-known (cf., e.g., \cite{chen-shaw}) that the $\frac{\partial}{\partial \overline{z}}$ is solvable in $L^2(D)$, i.e., that if $g\in L^2(D)$ then there exists $f\in L^2(D)$ such that $\frac{\partial f}{\partial \overline{z}}=g$ and \be
\int_D|f|^2\lesssim \int_D|g|^2.
\ee
If $f\in L^2(D)$ is such that $\frac{\partial f}{\partial \overline{z}}\in L^2(D)$, the above solvability result yields $\widetilde{f}\in L^2(D)$ such that $f-\widetilde{f}$ is holomorphic and $\int_D|\widetilde{f}|^2\lesssim \int_D|f|^2$. In particular, denoting by $B:L^2(D)\rightarrow L^2(D)$ the orthogonal projection onto the space of $L^2$ holomorphic functions (i.e., the unweighted Bergman operator), we have \bel\label{hol-poincare}
\int_D|f-B(f)|^2\leq \int_D|f-(f-\widetilde{f})|^2\lesssim \int_D|f|^2.
\eel
This is the inequality we need. One should compare it with the usual Poincar\'e inequality in which $\frac{\partial}{\partial \overline{z}}$ is to be replaced by $\nabla$ and $B$ with $\frac{1}{|D|}\int_D$. Of course one could rescale the estimate to apply it to an arbitrary disc.

The second ingredient is the following inequality, which holds for every holomorphic function $h:D\rightarrow\C$: \bel\label{hol-cauchy}
\int_D|h|^2\lesssim \int_{D\setminus\frac{1}{2}D}|h|^2,
\eel which follows easily from the Cauchy integral formula.

\begin{proof} By a trivial rescaling it is enough to prove the lemma for $z=0$ and $r=1$. Let then $V:D\rightarrow[0,+\infty)$ be such that $V\geq c$ on $D\setminus \frac{1}{2}D$, and $f\in L^2(D)$ be such that $\frac{\partial f}{\partial \overline{z}}\in L^2(D)$. If $\eps>0$ is a parameter to be fixed later and we write $f=f_i+f_e$, where $f_e$ is zero on $\frac{1}{2}D$ and $f_i$ is zero on $D\setminus \frac{1}{2}D$, we have the following dichotomy:\begin{enumerate}
\item either $\int_D|f_e|^2\geq \eps\int_D|f_i|^2$, 
\item or $\int_D|f_e|^2< \eps\int_D|f_i|^2$.
\end{enumerate}

If 1. happens, a significant portion of the $L^2$ mass of $f$ is contained in the corona $D\setminus \frac{1}{2}D$ and \be
\int_D|f_e|^2\geq \frac{\eps}{2}\int_D|f_i|^2+\frac{1}{2}\int_D|f_e|^2\geq  \frac{\eps}{2}\int_D|f|^2.
\ee
Therefore\be
\int_D\left|\frac{\partial f}{\partial \overline{z}}\right|^2+\int_DV|f|^2\geq \int_DV|f|^2\geq c\int_D|f_e|^2\geq c\frac{\eps}{2}\int_D|f|^2,
\ee and \eqref{hol-ineq} holds.

If 2. happens, we use \eqref{hol-poincare}:\be
\int_D\left|\frac{\partial f}{\partial \overline{z}}\right|^2+\int_DV|f|^2\gtrsim \int_D|f-B(f)|^2.\ee
By the linearity of $B$ and condition (2), we have
\bee
\int_D|f-B(f)|^2&\geq&\frac{1}{2}\int_D|f_i-B(f_i)|^2-\int_D|f_e-B(f_e)|^2\\
&\geq&  \frac{1}{2}\int_D|f_i-B(f_i)|^2-\int_D|f_e|^2\\
&\geq&  \frac{1}{2}\int_D|f_i-B(f_i)|^2-\eps\int_D|f_i|^2.
\eee In the second line we used the fact that $1-B$ is an orthogonal projection. 

We claim that \bel\label{hol-orthogonal}
\int_D|f_i-B(f_i)|^2\geq a\int_D|f_i|^2,
\eel where $a$ is some small absolute constant.

Inequality \eqref{hol-orthogonal} immediately implies, choosing $\eps= \frac{a}{4}$, that \be
\int_D\left|\frac{\partial f}{\partial \overline{z}}\right|^2+\int_DV|f|^2\gtrsim\int_D|f_i|^2\gtrsim\int_D|f|^2.
\ee 
We are reduced to proving \eqref{hol-orthogonal}. This follows separating the two cases (for a new parameter $\delta$):\begin{enumerate}
\item either $\int_{D\setminus \frac{1}{2}D}|B(f_i)|^2\geq\delta\int_D|f_i|^2$,
\item or $\int_{D\setminus \frac{1}{2}D}|B(f_i)|^2<\delta\int_D|f_i|^2$.
\end{enumerate}

If 1. holds, \be
\int_D|f_i-B(f_i)|^2\geq\int_{D\setminus \frac{1}{2}D}|f_i-B(f_i)|^2=\int_{D\setminus \frac{1}{2}D}|B(f_i)|^2\geq\delta\int_D|f_i|^2.
\ee

If 2. holds instead, we apply \eqref{hol-cauchy} to the holomorphic function $B(f_i)$ to deduce that $\int_D|B(f_i)|^2\lesssim \delta\int_D|f_i|^2$. If we choose $\delta$ small enough we can write\be
\int_D|f_i-B(f_i)|^2\geq \frac{1}{2}\int_D|f_i|^2-\int_D|B(f_i)|^2\geq \frac{1}{4}\int_D|f_i|^2.
\ee
This concludes the proof of \eqref{hol-orthogonal}.
\end{proof}

Notice how the nature of uncertainty principle of the previous result is revealed by its proof: it shows that a function $f$ defined on a disc cannot be concentrated on a strictly smaller disc without having a large ``holomorphic kinetic energy" $\int_{D(z,r)}\left|\frac{\partial f}{\partial \overline{z}}\right|^2$.

\section{Estimation of the energy functional}\label{est-sec}

Let $u=u_1d\overline{z}+u_2d\overline{w}$ be a $(0,1)$-form with smooth compactly supported coefficients. We introduce the notation\be
\mathcal{E}_\Omega(u):=\sum_{j=1}^2\int_\Omega\left(\left|\frac{\partial u_j}{\partial \overline{z}}\right|^2+\left|\frac{\partial u_j}{\partial \overline{w}}\right|^2\right)e^{-2\varphi_\Gamma}+2\int_\Omega(H_{\varphi_\Gamma} u,u)e^{-2\varphi_\Gamma},
\ee where $\Omega\subseteq \C^2$. Denoting by $\lambda_\Gamma(z,w)$ the minimal eigenvalue of $H_{\varphi_\Gamma}(z,w)$, we have \bel\label{est-min}
\mathcal{E}_\Omega(u)\geq\sum_{j=1}^2\left(\int_\Omega\left(\left|\frac{\partial u_j}{\partial \overline{z}}\right|^2+\left|\frac{\partial u_j}{\partial \overline{w}}\right|^2\right)e^{-2\varphi_\Gamma}+2\int_\Omega\lambda_\Gamma |u_j|^2e^{-2\varphi_\Gamma}\right).
\eel
We introduce the \emph{uncertainty regions} $U_0:=\{0\leq |z|, |w|\leq 2\}$,\be
U_r:=\{|z|> 1, 0\leq |w|\leq |z|^{-\sigma}\} \quad  \text{and}\quad U_u:=\{|w|> 1, 0\leq|z|\leq |w|^{-\tau}\}.
\ee

\begin{center}

\begin{tikzpicture} 
\small
\begin{axis}[
axis x line=bottom, axis y line=left, 
xmin=0, xmax=3, ymin=0, ymax=3, 
xtick={1,3},
ytick={1,3},
xticklabels={1,$x$},
yticklabels={1,$y$},
enlargelimits=false]

\addplot[dashed, domain=1:3, samples=100]{1/x^3}
[xshift= 80pt,yshift=10pt]
node[pos=1]{$U_r$}
;
\addlegendentry{$y=x^{-\sigma}$}

\addplot[dotted, domain=0:1, samples=100]{1/x}
[xshift= 20pt,yshift=100pt]
node[pos=1]{$U_u$}
;
\addlegendentry{$x=y^{-\tau}$}

\addplot[] coordinates{(1,0)(1,1)};
\addplot[] coordinates{(0,1)(1,1)};



\end{axis}

\end{tikzpicture}

\end{center}

\begin{prop}\label{est-prop}
\be
\sup_{(z,w)\in U_0}\varphi_\Gamma(z,w)\lesssim 1,\ \sup_{(z,w)\in U_r}\varphi_{\Gamma_r}(z,w)\lesssim 1\text{ and }\sup_{(z,w)\in U_u}\varphi_{\Gamma_l}(z,w)\lesssim 1,\ee
\end{prop}

\begin{proof} 
The first estimate follows from the compactness of $U_0$.

By the definition of $\sigma$, if $(z,w)\in U_r$ and $(\alpha,\beta)\in\Gamma_r$, then \be
|z^\alpha w^\beta|^2=\left(|z|^{\frac{\alpha}{\beta}}|w|\right)^{2\beta}\leq \left(|z|^\sigma|w|\right)^{2\beta}\leq 1,
\ee where in the first inequality we used the fact that $|z|\geq 1$. Summing over $(\alpha,\beta)\in\Gamma_r$, we obtain $\sup_{(z,w)\in U_r}\varphi_{\Gamma_r}(z,w)\lesssim 1$.

The last bound is proved similarly.
\end{proof}

The complement of $U_0\cup U_r\cup U_u$ will be denoted by $C$. Notice that \bel\label{est-dec}
\mathcal{E}_{\varphi_\Gamma}(u)\gtrsim \mathcal{E}_{U_0}(u)+\mathcal{E}_{U_r}(u)+\mathcal{E}_{U_u}(u)+\mathcal{E}_C(u).
\eel

We now deal with one term of this decomposition at a time.

\subsection*{$\mathcal{E}_{U_0}(u)$:}
By \eqref{est-min} and the first inequality of Proposition \ref{est-prop}, we have
\be
\mathcal{E}_{U_0}(u)\gtrsim\sum_{j=1}^2\left(\int_{U_0}\left(\left|\frac{\partial u_j}{\partial \overline{z}}\right|^2+\left|\frac{\partial u_j}{\partial \overline{w}}\right|^2\right)+\int_{U_0}\lambda_\Gamma |u_j|^2\right).
\ee Since $U_0=(2D)\times(2D)$, by Fubini's theorem and applying Lemma \ref{hol-unc} twice, we get\bee
&&\int_{(2D)\times(2D)}\left(\left|\frac{\partial u_j}{\partial \overline{z}}\right|^2+\left|\frac{\partial u_j}{\partial \overline{w}}\right|^2\right)+\int_{(2D)\times(2D)}\lambda_\Gamma |u_j|^2\\
&=& \int_{2D}\int_{2D}\left(\left|\frac{\partial u_j}{\partial \overline{z}}(z,w)\right|^2+\lambda_\Gamma(z,w) |u_j(z,w)|^2\right)d\mathcal{L}(z)d\mathcal{L}(w)\\
&+&\int_{(2D)\times(2D)}\left|\frac{\partial u_j}{\partial \overline{w}}\right|^2\\
&\gtrsim& \int_{2D}\left\{\left(\min_{1<|z'|\leq2}\lambda_\Gamma(z',w)\right)\int_{2D} |u_j(z,w)|^2d\mathcal{L}(z)\right\}d\mathcal{L}(w)\\
&+&\int_{(2D)\times(2D)}\left|\frac{\partial u_j}{\partial \overline{w}}\right|^2\\
&=&  \int_{2D}\int_{2D}\left(\left|\frac{\partial u_j}{\partial \overline{w}}(z,w)\right|^2+\left(\min_{1<|z'|\leq2}\lambda_\Gamma(z',w)\right) |u_j(z,w)|^2\right)d\mathcal{L}(w)d\mathcal{L}(z)\\
&\gtrsim&  \left(\min_{1<|z'|\leq2,1<|w'|\leq2}\lambda_\Gamma(z',w')\right) \int_{2D}\int_{2D}|u_j(z,w)|^2d\mathcal{L}(w)d\mathcal{L}(z).
\eee 
Notice that $\{1<|z'|\leq2,1<|w'|\leq2\}\subseteq E_1\cup E_2\cup E_3$ and hence Proposition \ref{lambda-prop} shows that the minimum above is $\approx1$. Thus \be
\mathcal{E}_{U_0}(u)\gtrsim \sum_{j=1}^2\int_{U_0}|u_j|^2\geq \int_{U_0}|u|^2e^{-2\varphi_\Gamma}\gtrsim\int_{U_0}(1+|z|^{2\sigma}+|w|^{2\sigma})|u|^2e^{-2\varphi_\Gamma},
\ee
where we used again the fact that $|z|,|w|\leq2$ on $U_0$.

\subsection*{$\mathcal{E}_{U_r}(u)$:}
\par Recall that $U_r=\{|z|\geq1,\quad |w|\leq |z|^{-\sigma}\}$.

By Fubini's theorem $\mathcal{E}_{U_r}(u)\geq \int_{|z|\geq1}I(z)e^{-2|z|^{2m}}d\mathcal{L}(z)$, where  
\be
I(z):=\sum_{j=1}^2\int_{D(0, |z|^{-\sigma})}\left(\left|\frac{\partial u_j}{\partial \overline{w}}(z,w)\right|^2+\lambda_\Gamma(z,w) |u_j(z,w)|^2\right)e^{-2\varphi_{\Gamma_r}(z,w)}d\mathcal{L}(w),
\ee
where we used the fact that $\varphi_\Gamma(z,w)=|z|^{2m}+\varphi_{\Gamma_r}(z,w)$. The second inequality of Proposition \ref{est-prop} and Lemma \ref{hol-unc} yield, for every $z$ of modulus greater than or equal to $1$,\bee
I(z)&\gtrsim&\sum_{j=1}^2\int_{D(0, |z|^{-\sigma})}\left(\left|\frac{\partial u_j}{\partial \overline{w}}(z,w)\right|^2+\lambda_\Gamma(z,w) |u_j(z,w)|^2\right)d\mathcal{L}(w)\\
&\gtrsim& \left(\min_{\frac{|z|^{-\sigma}}{2}<|w'|\leq |z|^{-\sigma}}\lambda_\Gamma(z,w')\right)\int_{D(0, |z|^{-\sigma})}|u(z,w)|^2d\mathcal{L}(w).
\eee
Since the region $U_r$ is contained in $E_1$, Proposition \ref{lambda-prop} gives\bee
\min_{\frac{|z|^{-\sigma}}{2}<|w'|\leq |z|^{-\sigma}}\lambda_\Gamma(z,w')&\approx& \min_{\frac{|z|^{-\sigma}}{2}<|w'|\leq |z|^{-\sigma}} |z|^{2\alpha_1}|w|^{2(\beta_1-1)}\\
&\approx& |z|^{2(\alpha_1-\sigma \beta_1+\sigma)}
\eee for every $|z|\geq1$. Since $\alpha_1/\beta_1=\sigma$, the last quantity equals $|z|^{2\sigma}$. Using again the boundedness of $\varphi_{\Gamma_r}$ in the region of integration, we have\bee
\mathcal{E}_{U_r}(u)&\gtrsim&\int_{|z|\geq1}\left(|z|^{2\sigma}\int_{D(0, |z|^{-\sigma})}|u(z,w)|^2d\mathcal{L}(w)\right)e^{-2|z|^{2m}}d\mathcal{L}(z)\\
&\gtrsim&\int_{|z|\geq1}\left(|z|^{2\sigma}\int_{D(0, |z|^{-\sigma})}|u(z,w)|^2e^{-2\varphi_{\Gamma_r}(z,w)}d\mathcal{L}(w)\right)e^{-2|z|^{2m}}d\mathcal{L}(z)\\
&\gtrsim&\int_{U_r}|z|^{2\sigma}|u(z,w)|^2e^{-2\varphi_\Gamma(z,w)}d\mathcal{L}(z,w)\\
&\gtrsim&\int_{U_r}(1+|z|^{2\sigma}+|w|^{2\tau})|u(z,w)|^2e^{-2\varphi_\Gamma(z,w)}d\mathcal{L}(z,w).
\eee
The last step follows from the inequalities $|w|\leq 1$ and $|z|\geq1$, which hold for any $(z,w)\in U_r$.

\subsection*{$\mathcal{E}_{U_u}(u)$:} This is done in complete analogy with the estimate of $\mathcal{E}_{U_r}(u)$, exchanging the role played by $z$ and $w$ and replacing $\sigma$ with $\tau$. In evaluating the appropriate minimum of $\lambda_\Gamma$ one has to use the fact that $U_u$ is contained in $E_3$ and the corresponding estimate of Proposition \ref{lambda-prop}. We omit the easy details that allow to prove the estimate \be
\mathcal{E}_{U_u}(u)\gtrsim \int_{U_u}(1+|z|^{2\sigma}+|w|^{2\tau})|u(z,w)|^2e^{-2\varphi_\Gamma(z,w)}d\mathcal{L}(z,w).
\ee

\subsection*{$\mathcal{E}_{\mathcal{C}}(u)$:} We use the trivial bound\be
\mathcal{E}_\mathcal{C}(u)\gtrsim \int_{\mathcal{C}}\lambda_\Gamma|u|^2e^{-2\varphi_\Gamma}.
\ee
Observe that $\mathcal{C}\subseteq E_1\cup E_2\cup E_3$ and that $\mathcal{C}\supseteq E_2$. Proposition \ref{lambda-prop} gives\bee
&&\int_{\mathcal{C}}\lambda_\Gamma|u|^2e^{-2\varphi_\Gamma}=\int_{\mathcal{C}\cap E_1}\lambda_\Gamma|u|^2e^{-2\varphi_\Gamma}+\int_{E_2}\lambda_\Gamma|u|^2e^{-2\varphi_\Gamma}+\int_{\mathcal{C}\cap E_3}\lambda_\Gamma|u|^2e^{-2\varphi_\Gamma}\\
&\gtrsim& \int_{\mathcal{C}\cap E_1}|z|^{2\alpha_1}|w|^{2(\beta_1-1)}|u|^2e^{-2\varphi_\Gamma}+\int_{E_2}|w|^{2(n-1)}|u|^2e^{-2\varphi_\Gamma}+\int_{\mathcal{C}\cap E_3}|z|^{2(\alpha_2-1)}|w|^{2\beta_2} |u|^2e^{-2\varphi_\Gamma}.
\eee
If $(z,w)\in \mathcal{C}\cap E_1$ then $|z|^\frac{m}{n}\geq|w|\geq |z|^{-\sigma}$ and therefore\be
|z|^{2\alpha_1}|w|^{2(\beta_1-1)}\geq |z|^{2(\alpha_1-\sigma\beta_1+\sigma)}=|z|^{2\sigma}
\ee and \be
|z|^{2\alpha_1}|w|^{2(\beta_1-1)}\geq |w|^{2(\frac{n}{m}\alpha_1+\beta_1-1)}=|w|^{2(n-1)}.
\ee In the first identity we used the definition of $(\alpha_1,\beta_1)$, while in the second one we used \eqref{model-segment}. Notice that $\tau=\frac{\beta_2}{\alpha_2}\leq \beta_2\leq n-1$, and hence\bel\label{est-E1}
|z|^{2\alpha_1}|w|^{2(\beta_1-1)}\gtrsim 1+|z|^{2\sigma}+|w|^{2\tau}\qquad\forall (z,w)\in \mathcal{C}\cap E_1.
\eel

If $(z,w)\in E_2$, in particular $|z|\leq |w|^\frac{n}{m}$ and $|w|^{2(n-1)}\geq |z|^{2\frac{m}{n}(n-1)}$. Notice that $|w|\geq1$ on $E_2$ and $\nu\geq0$ and thus $|z|$ is also $\geq1$. If we show that $\frac{m}{n}(n-1)\geq \sigma$, we can then deduce that $|w|^{2(n-1)}\geq |z|^{2\sigma}$. To prove the inequality above one can plug in the identity $\sigma=\frac{\alpha_1}{\beta_1}$ and use \eqref{model-segment}. This, together with the already observed fact that $\tau\leq (n-1)$ allows to write that\bel\label{est-E2}
|w|^{2(n-1)}\gtrsim 1+|z|^{2\sigma}+|w|^{2\tau}\qquad\forall (z,w)\in E_2.
\eel

Finally, if $(z,w)\in \mathcal{C}\cap E_3$ then in particular $|w|^\frac{n}{m}\geq|z|\geq |w|^{-\tau}$ and therefore\be
|z|^{2(\alpha_2-1)}|w|^{2\beta_2}\geq |w|^{2(-\tau\alpha_2+\tau+\beta_2)}=|w|^{2\tau}
\ee and \be
|z|^{2(\alpha_2-1)}|w|^{2\beta_2}\geq |z|^{2(\alpha_2-1+\frac{m}{n}\beta_2)}=|z|^{2(m-1)}.
\ee The last identity follows as previously from \eqref{model-segment}. Since $\sigma=\frac{\alpha_1}{\beta_1}\leq \alpha_1\leq (m-1)$, we have\bel\label{est-E3}
|z|^{2(\alpha_2-1)}|w|^{2\beta_2}\gtrsim 1+|z|^{2\sigma}+|w|^{2\tau}\qquad\forall (z,w)\in\mathcal{C}\cap E_3.
\eel

The bounds \eqref{est-E1}, \eqref{est-E2} and \eqref{est-E3} give
\be
\mathcal{E}_\mathcal{C}(u)\gtrsim \int_\mathcal{C}(1+|z|^{2\sigma}+|w|^{2\tau})|u|^2e^{-2\varphi_\Gamma}.
\ee
\ \newline

Recalling \eqref{est-dec}, we can put together all our estimates to obtain:
\be
\mathcal{E}_{\varphi_\Gamma}(u)\gtrsim \int_{\C^2}(1+|z|^{2\sigma}+|w|^{2\tau})|u|^2e^{-2\varphi_\Gamma}, 
\ee 
for every $(0,1)$-form $u$ with smooth compactly supported coefficients. Since these forms are dense in $\mathcal{D}(\mathcal{E}_{\varphi_\Gamma})$ by part (ii) of Proposition \ref{kohn-dbar-star}, this concludes the proof of Theorem \ref{model-thm}.

\bibliographystyle{amsalpha}
\bibliography{tesi}

\end{document}